\documentclass[12pt, reqno]{amsart}
\usepackage{mathrsfs}
\usepackage{amsfonts}
\usepackage[centertags]{amsmath}
\usepackage{amsthm}
\usepackage{amssymb}
\usepackage{graphicx}
\usepackage{caption}
\usepackage{resizegather}
\usepackage[colorlinks]{hyperref}
\usepackage{enumerate}
\usepackage[textwidth=16cm, hmarginratio=1:1]{geometry}
\usepackage{appendix}
\usepackage[all]{xy}
\usepackage{color}
\usepackage{float}
\usepackage{tikz}
\usepackage{amsbsy}
\hypersetup{
	bookmarksnumbered,
	pdfstartview={FitH},
	breaklinks=true,
	linkcolor=blue,
	urlcolor=blue,
	citecolor=blue,
	bookmarksdepth=2
}
\newtheorem{theorem}{Theorem}[section]
\newtheorem{corollary}[theorem]{Corollary}
\newtheorem{lemma}[theorem]{Lemma}
\newtheorem{proposition}[theorem]{Proposition}
\newtheorem{definition}{Definition}[section]

\newtheorem{example}[theorem]{Example}

\newtheorem{conjecture}[theorem]{Conjecture}

\newcommand{\HL}{\text{HL}_r}
\newcommand{\Q}{Q_\xi}
\newcommand{\seed}{K_0(\mathscr{C}_{\ell})}

\begin{document}
\title[Generalized HL-modules and cluster algebras]{Generalized Hernandez-Leclerc modules and cluster algebras}
\date{\today}
	
\author{Jingmin Guo}
\address{Jingmin Guo: School of Mathematics and Statistics, Lanzhou University, Lanzhou 730000, P. R. China.}
\email{guojm18@lzu.edu.cn}
	
\author{Bing Duan}
\address{Bing Duan: School of Mathematics and Statistics, Lanzhou University, Lanzhou 730000, P. R. China.}
\email{duanbing@lzu.edu.cn}
	
\author{Yanfeng Luo$^\dag$}
\address{Yanfeng Luo: School of Mathematics and Statistics, Lanzhou University, Lanzhou 730000, P. R. China.}
\email{luoyf@lzu.edu.cn}
\thanks{$\dag$ Corresponding author}
	
\date{}
	
\maketitle
	
\begin{abstract}
We introduce generalized Hernandez-Leclerc modules over $U_q(\widehat{\mathfrak{sl}_{n+1}})$ as a generalization of Hernandez-Leclerc modules of type A, and prove that they are real and prime via monoidal categorifications of cluster algebras.

\hspace{0.15cm}
		
		
\hspace{0.15cm}
		

\end{abstract}
	
\tableofcontents
	
\setcounter{section}{0}
	
\allowdisplaybreaks
	
\section{Introduction}
Let $\mathfrak{g}$ be a complex simple Lie algebra and $U_q(\widehat{\mathfrak{g}})$ the corresponding quantum affine algebra. Let $\mathscr{C}$ be the category of finite-dimensional $U_q(\widehat{\mathfrak{g}})$-modules, and $K_0(\mathscr{C})$ the Grothendieck ring of $\mathscr{C}$. Every simple module in $\mathscr{C}$ is parameterized by its highest $l$-weight monomial \cite{CP944,FM01,FR99}. The study on simple modules has continuously attracted numerous researchers, for example, fundamental modules \cite{Bit21a, CP91, CP94, C02, CM06, J22, TDL233}, Kirillov-Reshetikhin modules \cite{H06, HL10, HL16, N03, N04, N11},  minimal affinizations \cite{CG11, CMY13, H07, LQ17, TDL23, ZDLL16}, snake modules of type AB \cite{DLL20, DS20, MY12a,MY12b}, Hernandez-Leclerc modules \cite{BC19, DLL2021,DS23, GDL22, HL13}.

A simple module $M$ in $\mathscr{C}$ is said to be \textit{prime} if it is not isomorphic to a tensor product of two nontrivial modules \cite{CP97}. A simple module $M$ in $\mathscr{C}$ is said to be \textit{real} if $M\otimes M$ is simple, otherwise it is said to be $\textit{imaginary}$ \cite{L03}. Classifying prime simple modules and real simple modules in $\mathscr{C}$ is an open and highly challenging question.

Cluster algebras were introduced by Fomin and Zelevinsky \cite{FZ02} as a tool to study canonical basis and total positivity from Lie theory. A cluster algebra is a commutative algebra with distinguished generators called cluster variables, which are grouped into overlapping sets called clusters. The monomials in cluster variables from the same cluster are called cluster monomials.

In \cite{HL10}, Hernandez and Leclerc introduced the concept of monoidal categorifications of cluster algebras. In particular, an abelian monoidal category of finite-dimensional $U_q(\widehat{\mathfrak{g}})$-modules was introduced, which plays an important role as the one of cluster category in additive categorifications of cluster algebras.

For $\mathfrak{g}$ of type ADE, Qin \cite{Q17}, independently, Kang, Kashiwara, Kim, and Oh \cite{KKKO18}, proved that every cluster monomial in $K_0(\mathscr{C})$ is in fact the equivalence class of a real simple module. Kashiwara, Kim, Oh, Park \cite{KKOP22} demonstrated that for $\mathfrak{g}$ of type AB, every cluster monomial in $K_0(\mathscr{C})$ is the equivalence class of a real simple module.  

Hernandez-Leclerc modules were first appeared in \cite{HL10, HL13}, and Hernandez-Leclerc modules of type A were investigated and named by Brito and Chari \cite{BC19}. In \cite{GDL22}, the authors gave a path description of $q$-characters of Hernandez-Leclerc modules. Moreover, they proved that the equivalence classes of Hernandez-Leclerc modules are cluster variables in $K_0(\mathscr{C})$, and the geometric $q$-character formula conjecture is true for Hernandez-Leclerc modules. Very recently, Hernandez-Leclerc modules of type ADE were uniformly defined and completely classified by Duan and Schiffler \cite{DS23} via an algorithm similar to the knitting algorithm of Auslander-Reiten quivers. In particular, Hernandez-Leclerc modules of type ADE are real and prime, and equations satisfied by $q$-characters of Hernandez-Leclerc modules were given \cite{BC19,DS23,GDL22}.

In \cite{BC22}, Brito and Chari constructed directly a family of imaginary modules from the perspective of representations of quantum affine algebras. They proved that there exists an imaginary module in the composition factors of the tensor product of a non-fundamental higher order Kirillov-Reshetikhin module with its dual.  Higher order Kirillov-Reshetikhin modules are prime snake modules.

In this paper, we introduce generalized HL-modules of type A as a generalization of Hernandez-Leclerc modules of type A. Every generalized HL-module of type A is of the form $L((Y_{i_1,a_1})_{r,r_1}(Y_{i_2,a_2})_{r,r_2}\dots (Y_{i_k,a_k})_{r,r_k})$, see Definition \ref{definition of generalizedHL}. In the case where $r\in \mathbb{Z}_{\geq 1}$ and $r_1=r_2=\dots=r_k=0$, we call $L((Y_{i_1,a_1})_{r,0}(Y_{i_2,a_2})_{r,0}\dots (Y_{i_k,a_k})_{r,0})$ an $\HL$-module, see Definition \ref{definition of GHL}. Moreover, if $r=1$, $\HL$-modules recover Hernandez-Leclerc modules (HL-modules for short). 

We study generalized HL-modules via monoidal categorifications of cluster algebras. Compared with HL-modules, the difficulty in the study of generalized HL-modules is that our cluster algebra is no longer of finite type. To attain it, we first study $\HL$-modules, which are related to cluster algebras of finite type, and then use the results of $\HL$-modules to study generalized HL-modules. Let us explain the process in detail as follows.

In Section \ref{section:definition of cluster algebra}, we define a cluster algebra $\mathscr{A}(\mathbf{x}, Q_\xi)$ with initial seed $(\mathbf{x}, Q_\xi)$ and with coefficients, where $\mathbf{x}=(\mathbf{f}'_1, \dots, \mathbf{f}'_n,\mathbf{x}_1,\dots,\mathbf{x}_n,\mathbf{f}''_1, \dots, \mathbf{f}''_n)$ is an initial extended cluster with initial cluster variables $\mathbf{x}_1,\dots,\mathbf{x}_n$ and frozen variables $\mathbf{f}'_1, \dots, \mathbf{f}'_n,\mathbf{f}''_1, \dots, \mathbf{f}''_n$, and $Q_\xi$ is a quiver associated to a height function $\xi$ with a type $A_n$ quiver as its principal part and with $2n$ frozen vertices $\{1',\dots,n',1'',\dots, n''\}$. In particular, for any height function $\xi$, if we delete the first $n$ frozen vertices $1',\dots,n'$ and incident arrows from $Q_\xi$, then the resulting quiver is opposite to a Chari-Brito's quiver \cite{BC19}. 

We next prove that $\mathscr{A}(\mathbf{x}, Q_\xi)$ is isomorphic to a cluster subalgebra of $K_0(\mathscr{C}_\ell)$, where $\mathscr{C_\ell}$ is a monoidal subcategory of $\mathscr{C}$ introduced by Hernandez and Leclerc \cite{HL10}, see Theorem \ref{mutation equivalent}, and give some exchange relations in $K_0(\mathscr{C}_\ell)$, see Corollary \ref{prop:mutation equations}. With the help of \cite{GLS13, KKKO18, KKOP22, Q17}, we give a criterion to compute the highest $l$-weight monomials of real prime simple modules corresponding to cluster variables, see Proposition \ref{lemma:equations satisfy by q-character}. We give a mutation sequence for an arbitrary $\HL$-module in $\seed$, see Section \ref{section of mutation sequence}, and prove that the equivalence classes of $\HL$-modules are cluster variables in $\seed$, see Theorem \ref{theorem:real and prime}. In particular, $\HL$-modules are real and prime. A recursive $q$-character formula of $\HL$-modules by an induction on the length of their highest $l$-weight monomials is given, see Proposition \ref{prop:equation system}.

By revising a mutation sequence of $L((Y_{i_1,a_1})_{r,0}(Y_{i_2,a_2})_{r,0}\dots (Y_{i_k,a_k})_{r,0})$, we formulate a mutation sequence of $L((Y_{i_1,a_1})_{r,r_1}(Y_{i_2,a_2})_{r,r_2}\dots (Y_{i_k,a_k})_{r,r_k})$ in Section \ref{section of generalizedHL}. In particular, the equivalence classes of generalized HL-modules are cluster variables of $\seed$, as a result, these modules are real and prime, see Theorem \ref{generalized HL-modules are cluster variables}.

The organization of this paper is outlined as follows. In Section \ref{preface}, we recall some basic and necessary materials on cluster algebras, quantum affine algebras and their representations, as well as Hernandez-Leclerc modules. In Section \ref{The construction of the cluster algebra}, we define a cluster algebra $\mathscr{A}(\mathbf{x},\Q)$ with coefficients, and give some exchange relations. To study generalized HL-modules, in Section \ref{The study on HLr modules}, we introduce $\HL$-modules, prove that their equivalence classes are cluster variables in $K_0(\mathscr{C}_\ell)$, and give a recursive formula of $q$-characters of $\HL$-modules. In Section \ref{section of generalizedHL}, we define generalized HL-modules by their highest $l$-weight monomials, and prove that generalized HL-modules are real and prime via monoidal categorifications of cluster algebras. In Appendix \ref{appendix}, we prove a typical case for the arrows incident to a mutable vertex in our mutation sequence.

\section{Preliminary} \label{preface}

\subsection{Cluster algebras}

We start by recalling the definition of a cluster algebra \cite{FZ02}. Let $m$ and $n$ be two positive integers with $m\geq n$. Let $\mathcal{F}$ be the field of rational functions over $\mathbb{Q}$ in $m$ independent variables. A pair $\Sigma=(\mathbf{x},Q)$ is called a seed if $\mathbf{x}=(x_1,x_2,\dots,x_m)$ is an $m$-tuple of elements of $\mathcal{F}$, which form a free generating set of $\mathcal{F}$, and $Q$ is a quiver without loops or 2-cycles and with vertices labeled by $1,2,\dots,m$. The $m$-tuple $\mathbf{x}$ is called an \textit{extended cluster}, the $n$-tuple $(x_1,x_2,\dots,x_n)$ is called a \textit{cluster}, the variables $x_i$, $i\in [1,n]$, are called \textit{cluster variables}, and the variables $x_i$, $n+1\leq i \leq m$, are called \textit{frozen variables}. Let $1\leq k \leq n$.  The mutation of the seed $\Sigma$ at $k$ will get a new seed $(\mathbf{x}',Q'):=\mu_k(\mathbf{x},Q)$, where $\mathbf{x}'=(x'_1,x'_2,\dots,x'_m)$ is defined by $x'_i=x_i$ for $i\neq k$, and
\begin{align*}
x'_k=\frac{\prod_{i\rightarrow k} x_i+\prod_{k\rightarrow j} x_j}{x_k},
\end{align*}
where the first (respectively, second) product in the right-hand side is over all arrows of $Q$ with sink (respectively, source) $k$, and $Q'$ is obtained from $Q$ by
\begin{enumerate}
\item for each oriented two-arrow path $i\rightarrow k\rightarrow j$, adding a new arrow $i\rightarrow j$ unless $n+1\leq i, j \leq m$; \label{mutation rule1}
\item reversing the orientations of all arrows incident to $k$;\label{mutation rule2}
\item repeatedly removing oriented 2-cycles until unable to do so.\label{mutation rule3}
\end{enumerate}
Denote by $\mathcal{C}$ the set of all clusters obtained from $\Sigma$ by a finite sequence of mutations. The \textit{cluster algebra} $\mathscr{A}_{\Sigma}$ associated with $\Sigma$ is the $\mathbb{Q}[x_{n+1}, \dots, x_m]$-subalgebra of the ambient field $\mathcal{F}$ generated by all cluster variables in $\mathcal{C}$. \textit{Cluster monomials} are monomials in the cluster variables and frozen variables belonging to the same extended cluster.
	
\subsection{Quantum affine algebras and their representations}\label{q-character}

Let $\mathfrak{g}=\mathfrak{sl}_{n+1}(\mathbb{C})$. Denote by $I=[1,n]$ the set of vertices of the Dynkin diagram of $\mathfrak{g}$ and by $C=(c_{ij})_{i,j\in I}$ the Cartan matrix of $\mathfrak{g}$. Let $\widehat{\mathfrak{g}}$ be the corresponding untwisted affine Lie algebra which is realized as a central extension of the loop algebra $\mathfrak{g}\otimes\mathbb{C}[t, t^{-1}]$. Let $U_q(\widehat{\mathfrak{g}})$ be the Drinfeld-Jimbo quantum enveloping algebra (quantum affine algebra for short) of $\widehat{\mathfrak{g}}$ with parameter $q\in\mathbb{C}^{\times}$ not a root of unity.
	
Let $\mathscr{C}$ be the category of finite-dimensional $U_q(\widehat{\mathfrak{g}})$-modules. Fix $a\in \mathbb{C}^{\times}$, let $Y_{i,s}=Y_{i,aq^s}$ with $i\in I$ and $s\in\mathbb{Z}$. Denote by $\mathcal{P}$ the free abelian group generated by $Y^{\pm1}_{i,s}$ with $i\in I$ and  $s\in\mathbb{Z}$, and by $\mathcal{P}^{+}$ the submonoid of $\mathcal{P}$ generated by $Y_{i,s}$ with $i\in I$ and $s\in\mathbb{Z}$. The elements in $\mathcal{P}^{+}$ are called dominant monomials. Every simple module in $\mathscr{C}$ is of the form $L(M)$, where $M\in\mathcal{P}^{+}$ is called the highest $l$-weight (or sometimes, loop-weight) monomial of $L(M)$. 
	
Denote by $K_0(\mathscr{C})$ the Grothendieck ring of $\mathscr{C}$, and by $[V]\in K_0(\mathscr{C})$ the equivalence class of $V\in \mathscr{C}$. Following \cite{FR99}, the $q$-character map is defined as an injective ring homomorphism $\chi_q: K_0(\mathscr{C}) \to \mathbb{Z}[Y^{\pm1}_{i,s}]_{i\in I; s\in \mathbb{Z}}$, where $\mathbb{Z}[Y^{\pm1}_{i,s}]_{i\in I; s\in \mathbb{Z}}$ is the ring of Laurent polynomials in infinitely many formal variables $(Y^{\pm1}_{i,s})_{i\in I; s\in \mathbb{Z}}$.

Following~\cite{FR99}, for $i\in I$, $s\in\mathbb{Z}$, one can define
\[
A^{-1}_{i,s}=Y^{-1}_{i,s-1}Y^{-1}_{i,s+1}\prod_{j:c_{ji}=-1}Y_{j,s},
\]
where the $c_{ji}$ is the entry of the Cartan matrix $C$ of $\mathfrak{g}$. Obviously, $A_{i,s}$ is a Laurent monomial in the variables $Y_{i,r}$ with $i\in I$, $r\in \mathbb{Z}$.
	
For any simple module $V$ in $\mathscr{C}$, it was shown by Frenkel and Mukhin~\cite{FM01} that the $q$-character of $V$ can be expressed as
\begin{align}\label{Frenkel-Mukhin formula}
\chi_q(V)=m^+(1+\sum_p M_p),
\end{align}
where $m^+\in \mathcal{P}^{+}$ is the highest $l$-weight monomial of $V$, and each $M_p$ is a product of factors $A^{-1}_{i,s}$ with some $i\in I$ and $s\in \mathbb{Z}$. There is a partial order $\leq$ defined on $\mathcal{P}$ as follows:
\begin{align}\label{partial order on weight lattice}
m\leq m'\Leftrightarrow m'm^{-1}~\text{is a monomial generated by } A_{i,s} \text{with } i\in I, s\in\mathbb{Z}.
\end{align}
Then in Equation (\ref{Frenkel-Mukhin formula}), $m^+$ is maximum with respect to $\leq$ among the monomials appearing in $\chi_q(V)$. 
	
\subsection{Category $\mathscr{C}_{\ell}$ and cluster structure on $K_0(\mathscr{C}_{\ell})$} \label{K0Ccluster algebra} 
	
In \cite{HL10,HL16}, Hernandez and Leclerc introduced a series of monoidal subcategories $\mathscr{C}_{\ell}$ of $\mathscr{C}$ with $\ell\in\mathbb{Z}_{\geq0}$. Denote by $\mathcal{P}^{+}_{\ell}$ the submonoid of $\mathcal{P}$ generated by $Y_{i,\xi_i-2k}$, where $i\in I$, $k\in[0,\ell]$, and $\xi_i=-1$ if $i\equiv1 \pmod 2$, $\xi_i=0$ if $i\equiv0 \pmod 2$. An object $V$ in $\mathscr{C}_{\ell}$ is a finite-dimensional $U_q(\widehat{\mathfrak{g}})$-module which satisfies the condition: for every composition factor $S$ of $V$, the highest $l$-weight of $S$ is a monomial in the variables $Y_{i,\xi_i-2k}$ with $i\in I$ and $k\in[0,\ell]$, see \cite{HL10}. Simple modules in $\mathscr{C}_{\ell}$ are of the form $L(M)$, where $M\in\mathcal{P}^{+}_{\ell}$. Denote by $K_{0}(\mathscr{C}_{\ell})$ the Grothendieck ring of $\mathscr{C}_{\ell}$.
		
Following \cite{HL16}, let $Q_{\ell}$ be an iced quiver with the vertex set $V_{\ell}=I\times[1,\ell+1]$, and with arrows:
\begin{align*}
(i,r) & \rightarrow (i-1,r+1), \quad (i,r) \rightarrow (i+1,r+1) \quad \text{if $i \equiv 1 \pmod 2$},\\
(i,r) & \rightarrow (i-1,r), \quad (i,r) \rightarrow (i+1,r) \quad  \text{if $i \equiv 0 \pmod 2$, $r\in [1,\ell]$}, \\
(i,r) & \rightarrow (i,r-1).
\end{align*}
Here the vertices $(i,\ell+1)$ with $i\in I$ are frozen vertices in $Q_{\ell}$.

Let $\mathbf{z}=\{z_{i,r}\mid (i,r)\in V_{\ell}\}$ and let $\mathscr{A}_{\ell}$ be the cluster algebra defined by the initial seed $(\mathbf{z},Q_{\ell})$. For $i \in I$, $s \in \mathbb{Z}$, $k \in \mathbb{Z}_{\geq1}$, let
\[
X^{(s)}_{i,k}=Y_{i,s} Y_{i,s+2} \cdots Y_{i,s+2k-2}.
\]
The modules $L(X^{(s)}_{i,k})$ are called Kirillov-Reshetikhin modules and their equivalence classes $[L(X^{(s)}_{i,k})]$ serve as initial cluster variables. When $k=1$, the modules $L(X^{(s)}_{i,1})=L(Y_{i,s})$ are called fundamental modules. It was shown in \cite{HL10,HL16} that the assignments 
$z_{i,r}\mapsto [L(X^{(\xi_i-2r+2)}_{i,r})]$ with $i\in I$ and $r\in[1,\ell+1]$, extend to a ring isomorphism $\Phi: \mathscr{A}_{\ell}\rightarrow K_{0}(\mathscr{C}_{\ell})$. For example, the initial cluster for $K_{0}(\mathscr{C}_{4})$ in type $A_6$ is shown in Figure \ref{initial seed}. We always read the vertices in $V_\ell$ layer-by-layer.

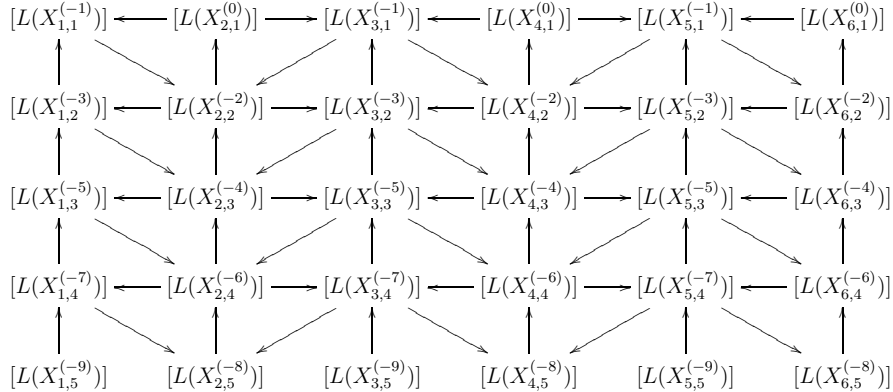
\begin{figure}[H]
\center
\resizebox{12cm}{!}{
\xymatrix{
[L(X^{(-1)}_{1,1})] \ar[rd] & [L(X^{(0)}_{2,1})] \ar[r]\ar[l]& [L(X^{(-1)}_{3,1})]\ar[ld]\ar[rd]& [L(X^{(0)}_{4,1})]  \ar[r]\ar[l] &[L(X^{(-1)}_{5,1})]\ar[rd]\ar[ld]  & [L(X^{(0)}_{6,1})]\ar[l]\\
[L(X^{(-3)}_{1,2})]\ar[rd]\ar[u]  & [L(X^{(-2)}_{2,2})]\ar[r]\ar[l]\ar[u] & [L(X^{(-3)}_{3,2})]\ar[ld]\ar[u] \ar[rd]& [L(X^{(-2)}_{4,2})]\ar[r]\ar[l]\ar[u]& [L(X^{(-3)}_{5,2})]\ar[ld]\ar[u] \ar[rd]& [L(X^{(-2)}_{6,2})] \ar[l]\ar[u] \\
[L(X^{(-5)}_{1,3})]\ar[rd]\ar[u]& [L(X^{(-4)}_{2,3})]\ar[l]\ar[r] \ar[u] &[L(X^{(-5)}_{3,3})]\ar[ld]\ar[u] \ar[rd] & [L(X^{(-4)}_{4,3})] \ar[r]\ar[l]\ar[u]&[L(X^{(-5)}_{5,3})]\ar[ld]\ar[u] \ar[rd] & [L(X^{(-4)}_{6,3})] \ar[l]\ar[u] \\
[L(X^{(-7)}_{1,4})]  \ar[rd]\ar[u]  & [L(X^{(-6)}_{2,4})]\ar[r] \ar[l]\ar[u] &[L(X^{(-7)}_{3,4})] \ar[ld]\ar[u] \ar[rd] & [L(X^{(-6)}_{4,4})]\ar[r]\ar[l]\ar[u]&[L(X^{(-7)}_{5,4})]\ar[ld]\ar[u] \ar[rd] & [L(X^{(-6)}_{6,4})] \ar[l]\ar[u] \\
[L(X^{(-9)}_{1,5})]  \ar[u] & [L(X^{(-8)}_{2,5})]  \ar[u]  &[L(X^{(-9)}_{3,5})] \ar[u] & [L(X^{(-8)}_{4,5})] \ar[u]  &[L(X^{(-9)}_{5,5})] \ar[u]  & [L(X^{(-8)}_{6,5})] \ar[u]
}}
\caption{The initial seed for $K_{0}(\mathscr{C}_{4})$ in type $A_6$.}\label{initial seed}
\end{figure}

Hernandez and Leclerc proposed the following conjectures on $\mathscr{C}_{\ell}$ and $K_{0}(\mathscr{C}_{\ell})$.

\begin{conjecture}[{\cite[Conjecture 13.2]{HL10},\cite[Conjecture 5.2]{HL16},\cite[Conjecture 9.1]{L10}}]\label{conjecture of bijection}
There is a bijection between the cluster monomials (respectively, cluster variables) in $K_{0}(\mathscr{C}_{\ell})$ and real simple (respectively, real prime simple) modules
 in $\mathscr{C}_{\ell}$.
\end{conjecture}

Hernandez and Leclerc proved Conjecture \ref{conjecture of bijection} for $\ell=1$ of type $A_n$ and type $D_4$ in \cite{HL10}, and for $\ell=1$ of type $A_n$ and type $D_n$ in \cite{HL13}, as well as Kirillov-Reshetikhin modules of all types in \cite{HL16}. Nakajima proved Conjecture \ref{conjecture of bijection} for type ADE and $\ell=1$ in \cite{N03}. Conjecture \ref{conjecture of bijection} holds for some special real simple modules, for example, minimal affinizations \cite{LQ17,ZDLL16}, snake modules of type AB \cite{DLL20, DS20}, Hernandez-Leclerc modules \cite{BC19}. A variant of Conjecture \ref{conjecture of bijection} for type ADE was proved in \cite{DS23} for the subcategories whose
Grothendieck rings are cluster algebras of finite type. In addition, some important advancements of Conjecture \ref{conjecture of bijection} were achieved by Geiss, Leclerc and Schr\"oer \cite{GLS13}, Qin \cite{Q17}, Kang, Kashiwara, Kim and Oh \cite{KKKO18}, and Kashiwara, Kim, Oh, and Park \cite{KKOP22}, which are summarized as follows.

\begin{lemma}[{\cite[Corollary 8.6]{GLS13},\cite[Theorem 1.2.1]{Q17},\cite[Corollary 7.1.4]{KKKO18},\cite[Theorem 6.10]{KKOP22}}] \label{theorem of real prime}
The modules corresponding to cluster monomials (cluster variables) of $K_{0}(\mathscr{C}_{\ell})$ are real (real prime) simple modules.
\end{lemma}

Conjecture \ref{conjecture of bijection} predicts that the converse of Lemma \ref{theorem of real prime} is also true. However it is very hard to classify the real simple modules in $\mathscr{C}_{\ell}$. Lapid and M\'inguez \cite{LM18} classified real simple modules satisfying a certain regularity condition in type A. Recently, in \cite{DS23}, Duan and Schiffler classified all real simple modules in the subcategories whose Grothendieck rings are cluster algebras of finite type.

Following \cite{HL21}, let $Q$ be a quiver with the vertex set $I$, and $M$ a representation of $Q$ over $\mathbb{C}$. We denote by $\mathbf{e}=(e_i)_{i\in I}\in \mathbb{N}^{I}$ a dimension vector. The variety $Gr(\mathbf{e},M)$ is a quiver Grassmannian. 

\begin{definition}[{\cite[Definition 5.1]{HL21}}]
The F-polynomial of the representation $M$ of $Q$ is
\[
F_{M}(\mathbf{v}):=\Sigma_{\mathbf{e}\in\mathbb{N}^{I}}\chi(Gr(\mathbf{e},M))\mathbf{v}^{\mathbf{e}},
\]
where $\mathbf{v}:=(v_i)_{i\in I}$ is a sequence of commutative variables, $\mathbf{v}^{\mathbf{e}}:=\prod_{i}v_{i}^{e_i}$, and $\chi(Gr(\mathbf{e},M))$ 
denotes the Euler characteristic of the variety $Gr(\mathbf{e},M)$.
\end{definition}

Let $\Gamma^-$ be a quiver with vertex set $\Gamma_{0}^{-}=\{(i,r)\mid i\in I, r<0, i\equiv r+1 \pmod 2 \}$ and arrows $(i, r)\rightarrow (j, s)$ if $c_{ij}\neq0$ and $s= r+c_{ij}$, which is introduced in \cite{HL16}. 
A potential $P$ is the formal sum of all oriented $3$-cycles in $\Gamma^-$. Let $R$ be the set of all cyclic derivatives of $P$ and let $J$ denote be the two-sided ideal of the path algebra $\mathbb{C}\Gamma^{-}$ generated by $R$.  Following \cite{DWZ08, HL16}, one can define the Jacobian algebra $A=\mathbb{C}\Gamma^{-}\slash J$. 

Let $I_{i, r}$ be the injective $A$-module at the vertex $(i, r)$ with $(i, r)\in \Gamma_{0}^{-}$. Assume that $m$ is a dominant monomial in $Y_{i,r}$, with $(i,r)\in W^{-}=\{(i,a+1)\mid (i,a)\in\Gamma_{0}^{-}\}$. It follows from \cite{HL16} that $m$ can be rewritten as 
\[
m:=\prod\limits_{(i,r)\in W^{-}}z_{i,r}^{g_{i,r}(m)},
\]
where $g_{i, r}(m)\in\mathbb{Z}$, 
$z_{i, r} = Y_{i, r}Y_{i, r+2} \cdots Y_{i,\xi_i}$.

Define 
\[
I(m)^{+} = \bigoplus_{g_{i, r}(m)>0} I_{i, r-1}^{\oplus g_{i, r}(m)}, \quad
I(m)^{-} = \bigoplus_{g_{i, r}(m)<0} I_{i, r-1}^{\oplus \mid g_{i, r}(m)\mid}.
\]
Let $K(m)$ be the kernel of a generic $A$-module homomorphism from $I(m)^{-}$ to $I(m)^{+}$.

\begin{conjecture}[{\cite[Conjecture 5.3]{HL16}}]\label{conjecture of geometric q-character formula}
Suppose that $L(m)$ is a real simple module in $\mathscr{C}_{\ell}$. Then the truncated $q$-character of $L(m)$ is equal to
\[
\chi_q^{-}(L(m))=mF_{K(m)}.
\]
\end{conjecture}

Hernandez and Leclerc proved Conjecture \ref{conjecture of geometric q-character formula} for Kirillov-Reshetikhin modules of all types in \cite{HL16}. If  the equivalences of real simple modules are cluster monomials (equivalently, Conjecture \ref{conjecture of bijection} holds), then it follows from \cite[Section 5.5.4]{HL21} that Conjecture \ref{conjecture of geometric q-character formula} is true. Some real simple modules, for example, snake modules of types A and B \cite{DS20}, HL-modules \cite{BC19, DS23, GDL22}, and minimal affinizations of type $G_2$ \cite{TDL23}, have been verified.

\subsection{Hernandez-Leclerc modules of type $A_n$}\label{section:recall HL-module}
			
Now we recall the definition of Hernandez-Leclerc modules of type $A_n$ in \cite{BC19} and their highest $l$-weight monomials \cite{DS23}. It is often convenient to abbreviate Hernandez-Leclerc modules with HL-modules.
			
Recall that $I=\{1,2,\dots,n\}$. Let $\xi: I \to \mathbb{Z}$ be a height function subject to $\vert \xi(i)-\xi(i-1)\vert =1$ for $2\leq i\leq n$. It will be convenient to extend $\xi$ to $[0, n+1]$ by defining $\xi(0)=\xi(2)$ and $\xi(n+1)=\xi(n-1)$. Define $\omega(i,j)\in \mathcal{P^+}$ as follows:
\begin{align}\label{Brito-Chari notation}
\omega(i,j)=Y_{i_1,a_1} Y_{i_2,a_2} \cdots Y_{i_k,a_k},
\end{align}
where $\{i_1,i_2,\dots,i_k\} \subset I$ such that $i=i_1<i_2<\dots<i_k=j$, $i_{2}<\dots<i_{k-1}$ is an ordered enumeration of the subset $\{\ell \mid i< \ell <j, \ \xi(\ell-1)=\xi(\ell+1)\}$, $a_\ell=\xi(i_\ell) \pm1$ if $\xi(i_\ell)= \xi(i_\ell-1) \pm 1$ for $\ell \geq 2$, and $a_{1}=\xi(i)\pm1$ if $\xi(i)=\xi(i+1)\pm 1$.
			
Let $Pr_{\xi}=\{Y_{i,\xi(i)\pm1}\mid i\in I\}\cup\{\omega(i,j)\mid 1\leq i<j \leq n \}$, and $\mathbf{f}=\{\mathbf{f}_i \mid i\in I\}$, where $\mathbf{f}_i =Y_{i,\xi(i)+1}Y_{i,\xi(i)-1}$. A simple module $L(m)$ for $m\in \mathbf{Pr}_{\xi}\cup\mathbf{f}$ is called an HL-module \cite{BC19}.
			
Following \cite{BC19}, for $i\in [1,n-2]$, let $i_{\diamond}\in [i,n]$ be minimal such that $i_{\diamond}\geq i$, $\xi(i_{\diamond})=\xi(i_{\diamond}+2),$ and set $i_{\diamond}=n-1$ if $i_{\diamond}$ dose not exist in $[i,n-2]$. We  adopt the convention that $(n-1)_{\diamond}=(n-1)$, $n_{\diamond}=n$. For $1\leq j\leq 1_{\diamond}$, set $j_{\bullet}=0$ and for $j>1_{\diamond}$, let $j_{\bullet} \in [1_{\diamond},j)$ be maximal such that $j_{\bullet}=(j_{\bullet})_\diamond$.
			
The next lemma gives the highest $l$-weight monomials of HL-modules.
			
\begin{lemma}[{\cite[Theorem 3.1]{DLL2021},\cite[Theorem 7.2]{DS23}}]  \label{def:HL}
An HL-module corresponding to a cluster variable (excluding frozen variables) is a simple $U_q(\widehat{\mathfrak{g}})$-module with the highest $l$-weight monomial
\begin{equation} \label{equation: HL}
Y_{i_1,a_1}Y_{i_2,a_2}\dots Y_{i_k,a_k},
\end{equation}
where $k\in\mathbb{Z}_{\geq 1}$, $i_j\in I, a_{j}\in \mathbb{Z}$ for $j=1,2,\dots,k$, and
\begin{enumerate}
\item[(1)]  $i_1<i_2<\dots<i_k$,
\item[(2)]  $(a_{j}-a_{j-1})(a_{j+1}-a_{j})<0$ for $2\leq j\leq k-1$,
\item[(3)]  $\lvert a_{j}-a_{j-1}\rvert=i_{j}-i_{j-1}+2$ for $2\leq j\leq k$.
\end{enumerate}
\end{lemma}

\section{Cluster algebra $\mathscr{A}(\mathbf{x}, \Q)$}\label{The construction of the cluster algebra}
In this section, we define a cluster algebra of type $A_n$, and then give its some exchange relations. 
			
\subsection{The quiver $\Q$ and cluster algebra $\mathscr{A}(\mathbf{x},\Q)$}\label{section:definition of cluster algebra}
Let $\xi$ be a height function. Inspired by \cite{BC19}, we define a quiver $\Q$ with $3n$ vertices labeled $\{1',\dots,n',1,\dots,n,1'',\dots,n''\}$ and with the set of edges given as follows:
\begin{itemize}
\item[(1)]  the vertices $\{1',\dots,n',1'',\dots,n''\}$ are frozen vertices, there are no arrows between frozen vertices;

\item[(2)]  if $1\leq j\leq n-1$ and $\xi(j)=\xi(j+1)-1$, the arrows at $j$ are:
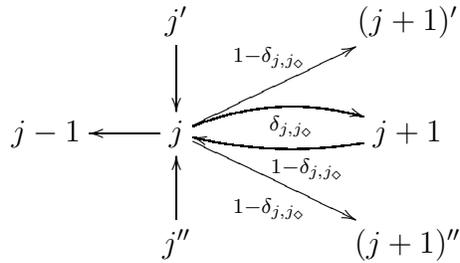
\begin{align*}
\xymatrix{
  &j'\ar[d]  & & (j+1)'\\
 j-1 & j\ar[l]\ar@/^10pt/[rr]_{\delta_{j,j_{\diamond}}} \ar[drr]_{1-\delta_{j,j_{\diamond}}}\ar[urr]^{1-\delta_{j,j_{\diamond}}}& & j+1 \ar@/^6pt/[ll]^{\quad 1-\delta_{j,j_{\diamond}}} \\
  &j''\ar[u]  & & (j+1)''
}
\end{align*}
and reverse the orientations of arrows if $\xi(j)=\xi(j+1)+1$, where $\delta_{i,j}$ is the Kronecker delta function, and we use the convention that a labeled arrow exists if and only if the label is $1$;  

\item[(3)]  if $\xi(n-1)=\xi(n)+1$, the arrows at $n$ are:
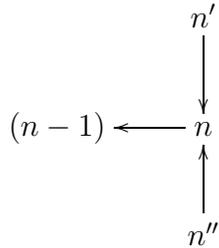
\begin{align*}
\xymatrix{
  &n'\ar[d]  \\
 (n-1) & n\ar[l]  \\
  &n''\ar[u]
}
\end{align*}
and reverse the orientations of arrows if $\xi(n-1)=\xi(n)-1$.
\end{itemize}
It is easy to see that if we ignore the frozen vertices, then $j$ is a sink or source of $\Q$ if and only if $j=1$ or $j=j_\diamond$.
			
\begin{example}
Assume that $\xi$ is defined as follows.
\[
\begin{tabular}{|c|c|c|c|c|c|c|}
\hline
$i$ &  $1$ & $2$  & $3$  & $4$  & $5$  & $6$  \\
\hline
$\xi(i)$ & $-3$ & $-2$  & $-3$  & $-4$  & $-5$  & $ -4$ \\
\hline
\end{tabular}
\]
Then the quiver $\Q$ is shown in Figure \ref{quiver of example}.
\begin{figure}[H]
\centerline{
\xymatrix{
1' \ar[d] & 2' & 3'\ar[ld]& 4' \ar[ld] &5'\ar[d] & 6' \\
1\ar[r]& 2\ar[u]\ar[d]\ar[r] & 3\ar[u]\ar[d]\ar[r]& 4\ar[u]\ar[d]& 5\ar[l]\ar[r]& 6\ar[u]\ar[d]\\
1''\ar[u]& 2''& 3''\ar[lu]& 4''\ar[lu]& 5''\ar[u]& 6''\\
}}
\caption{The quiver $\Q$ associated with $\xi$.}\label{quiver of example}
\end{figure}
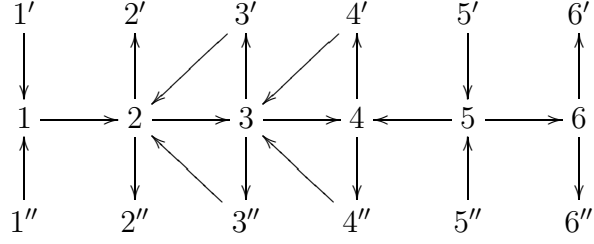
\end{example}
			
Let $\mathbf{x}=\{\mathbf{f}'_{1},\dots,\mathbf{f}'_{n},\mathbf{x}_1,\dots,\mathbf{x}_n,\mathbf{f}''_{1},\dots,\mathbf{f} ''_{n}\}$ be a set of algebraically independent variables and $\mathcal{F}$ the field of rational functions generated by $\mathbf{x}$.
			
Let $\mathscr{A}(\mathbf{x},\Q)$ be a cluster algebra with coefficients and with initial seed $(\mathbf{x},\Q)$. By definition, the principal part of $\Q$ is a Dynkin quiver of type $A_n$. We denote by $\{\alpha_i \mid 1\leq i\leq n\}$ the set of simple roots in type $A_n$, and let $\alpha_{i,j}=\alpha_i+\dots+\alpha_j$ for $1\leq i\leq j$. Then $\Phi_{\geq -1} = \{-\alpha_i, \, \alpha_{i,j} \mid 1\leq i\leq j\leq n\}$ is the set of almost positive roots in type $A_n$. Following \cite{FZ03}, the cluster variables in $\mathscr{A}(\mathbf{x},\Q)$ are parameterized by $\Phi_{\geq-1}$.  Cluster variables and frozen variables in $\mathscr{A}(\mathbf{x},\Q)$ are denoted
\[
\{\mathbf{x}_i:=\mathbf{x}[-\alpha_i],\mathbf{x}[\alpha_{i,j}],\mathbf{f}'_{i},\mathbf{f}''_{i} \mid 1\leq i\leq j\leq n\}.
\]
Moreover, the cluster variable $\mathbf{x}[\alpha_{i,j}]$ is obtained from the initial seed by performing a mutation sequence $i, i+1, \dots, j$. In particular, $\mathbf{x}[\alpha_{i,i}]$ is obtained from the initial seed by mutating at vertex $i$.
			
\subsection{Some exchange relations}
In this subsection, we give some exchange relations in $\mathscr{A}(\mathbf{x},\Q)$, which are useful in the sequel.
			
For $i,j\in[1,n]$, let $\Q[i,j]=\Q$ if $j<i$ and let $\Q[i,j]$ be the quiver obtained from $\Q$ by mutating the sequence $i,i+1,\dots,j$ if $j\geq i$. Let $d_{j}=\delta_{j,j_\diamond}$ for $1\leq j\leq n$. In the following, we describe the arrows incident to $j$ in $Q_\xi[i,j-1]$.
			
\begin{lemma}\label{arrows accident with j}
Suppose that $j>i$ and there is an arrow $j-1\rightarrow j$ in $\Q$. Then the arrows incident to $j$ in $\Q[i,j-1]$ are shown as follows:
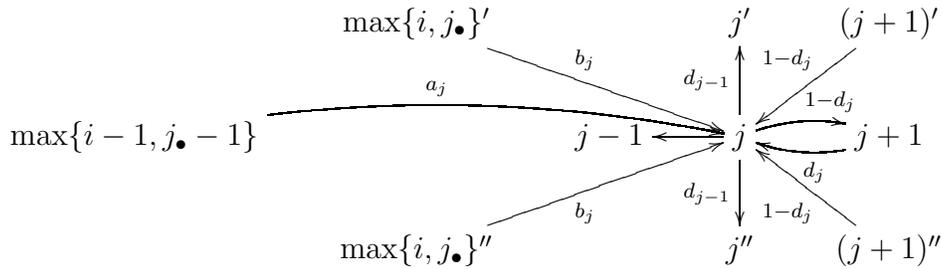
\[
\xymatrix{
& \max\{i,j_{\bullet}\}'\ar[drr]^{b_{j}}  &  &j' &(j+1)' \ar[dl]_{1-d_{j}} \\
\max\{i-1,j_{\bullet}-1\} \ar@/^1pc/[rrr]^{a_{j}}& & j-1 & j\ar[u]^{d_{j-1}} \ar[l]\ar@/^/[r]^{\quad 1-d_{j}}\ar[d]_{d_{j-1}} &  j+1 \ar@/^/[l]^{d_{j}} \\
& \max\{i,j_{\bullet}\}''\ar[urr]_{b_{j}}  &  & j''  & (j+1)''\ar[ul]^{1-d_{j}}}
\]
where $a_j=1-\delta_{i,j_{\bullet}},$ $b_{j}=\min\{1,(1-\delta_{j_{\bullet},i_{\bullet}})d_{j_{\bullet}-1}+\delta_{j_{\bullet},i}\}$. Reversing the orientation of every arrow if there is an arrow $j-1\leftarrow j$ in $\Q$.
\end{lemma}
			
\begin{proof}
We prove the case that there is an arrow $j-1\rightarrow j$ in $\Q$ by induction on $j-i$, the other is similar.
				
Suppose that $j-i=1$. We need to consider two cases: $d_i=1$ or $d_i=0$. In the former case, we have $j_\bullet=(i+1)_\bullet=i$, and so $a_j=0$, $b_{j}=1$. The mutation of $Q_\xi$ at $i$ is shown as follows:
 
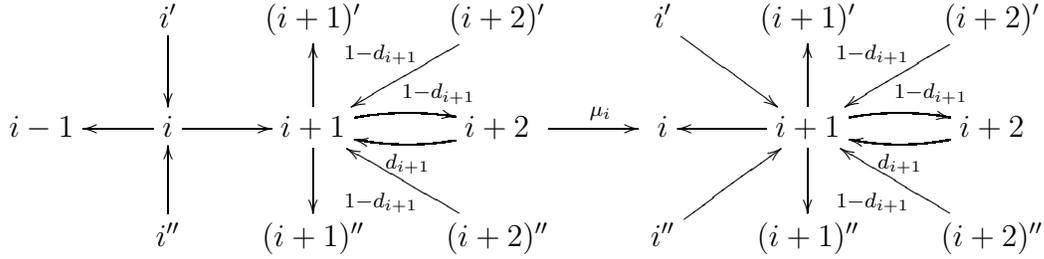
\[
\xymatrix{
&i'\ar[d]& (i+1)'&  (i+2)'\ar[ld]_{1-d_{i+1}}~~&~~  i'\ar[rd]& (i+1)'& (i+2)'\ar[ld]_{1-d_{i+1}}& \\
i-1 &i\ar[r]\ar[l]&  i+1 \ar[u]\ar[d]\ar@/^/[r]^{\quad \quad 1-d_{i+1}}&  i+2 \ar@/^/[l]^{d_{i+1}}  ~~\ar[r]^{~~\quad\mu_{i}}&~~  i& i+1  \ar[u]\ar[d]\ar[l]\ar@/^/[r]^{\quad\quad  1-d_{i+1}}&  i+2 \ar@/^/[l]^{d_{i+1}}&\\
&i''\ar[u]& (i+1)''&(i+2)'' \ar[lu]^{1-d_{i+1}} ~~&~~  i''\ar[ru]& (i+1)''& (i+2)''\ar[lu]^{1-d_{i+1}}&
}
\]
The arrows incident to $i+1$ in $Q_\xi[i,i]$ are as required.
				
In the latter case, we have $(i+1)_\bullet<i$, $j_\bullet=(i+1)_\bullet=i_{\bullet}$, and so $a_j=1$, $b_{j}=0$. The mutation of $Q_\xi$ at $i$ is shown as follows:
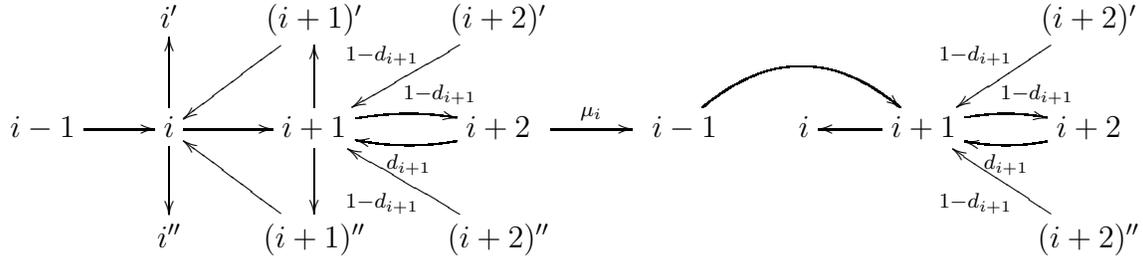
\[
\xymatrix{
&i'& (i+1)'\ar[ld]&  (i+2)'\ar[ld]_{1-d_{i+1}}~~&~~   &  &  & (i+2)'\ar[ld]_{1-d_{i+1}} \\
i-1 \ar[r]&i\ar[d]\ar[u]\ar[r]& i+1 \ar[u]\ar[d]\ar@/^/[r]^{\quad \quad 1-d_{i+1}}&   i+2 \ar@/^/[l]^{d_{i+1}}  ~~\ar[r]^{\mu_{i}}&~~   i-1 \ar@/^{2pc}/[rr]&i&  i+1  \ar[l]\ar@/^/[r]^{\quad \quad 1-d_{i+1}}&  i+2 \ar@/^/[l]^{d_{i+1}}\\
&i''& (i+1)''\ar[lu]&(i+2)'' \ar[lu]^{1-d_{i+1}} ~~&~~   &  & & (i+2)'' \ar[lu]^{1-d_{i+1}}
}
\]
The arrows incident to $i+1$ in $Q_{\xi}[i,i]$ are as required.
				
Suppose that $i+1<j<n$, and the arrows incident to $j$ in $\Q[i,j-1]$ are the required arrows. By induction, we need to prove that the result holds for the arrows incident to $j+1$ in $\Q[i,j]$. We consider two cases: $d_j=1$ or $d_j=0$.
In the former case, we have $j=j_{\diamond}$, $(j+1)_{\bullet}=j$, and then
\begin{align*}
& \max\{i-1,(j+1)_\bullet-1\}=j-1, \quad \max\{i,(j+1)_{\bullet}\}=j, \\
& a_{j+1}=1-\delta_{i,j}=1, \quad b_{j+1}=d_{j-1}.
\end{align*}
By induction the arrows incident to $j$, $j+1$ are shown as follows:
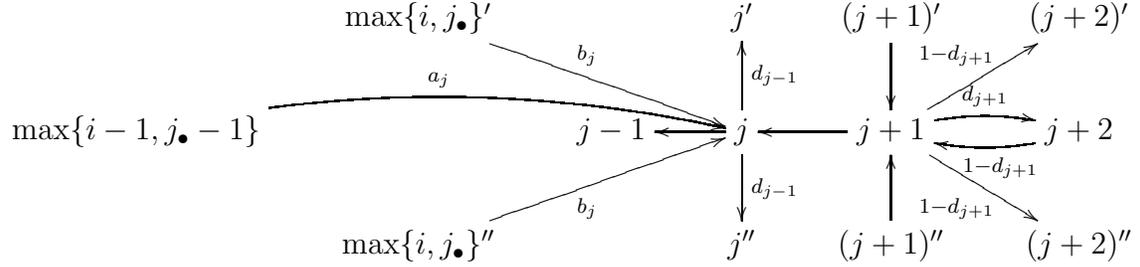
\[
\xymatrix{
&\max\{i,j_{\bullet}\}'\ar[drr]^{b_j}&&j' & (j+1)'\ar[d]&  (j+2)'\\
\max\{i-1,j_{\bullet}-1\}\ar@/^1.1pc/[rrr]^{a_j}&& j-1 &j\ar[l]\ar[d]^{d_{j-1}}\ar[u]_{d_{j-1}}&  j+1 \ar[l] \ar@/^/[r]^{d_{j+1}}\ar[ru]^{1-d_{j+1}}\ar[rd]_{1-d_{j+1}}&  j+2 \ar@/^/[l]^{\quad 1-d_{j+1}}  \\
& \max\{i,j_{\bullet}\}''\ar[urr]_{b_j}& & j''& (j+1)''\ar[u]&(j+2)''
}
\]
After mutating $j$, the arrows incident to $j+1$ are shown as follows:
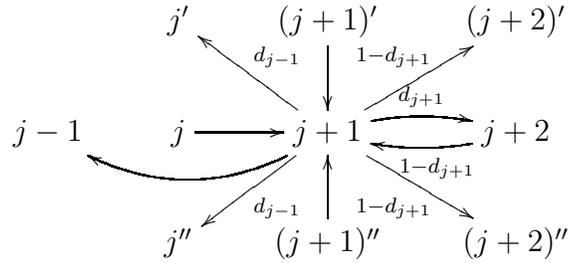
\[
\xymatrix{
&j' & (j+1)'\ar[d]&  (j+2)'\\
j-1 &j\ar[r] & j+1 \ar[ld]^{d_{j-1}} \ar[lu]_{d_{j-1}} \ar[ru]^{1-d_{j+1}} \ar[rd]_{1-d_{j+1}} \ar@/^{1.5pc}/[ll] \ar@/^/[r]^{d_{j+1}} &  j+2 \ar@/^/[l]^{\quad 1-d_{j+1}} \\
&j'' & (j+1)''\ar[u]&(j+2)''
}
\]
Note that $j \leftarrow j+1$ in $Q_\xi$. The arrows incident to $j+1$ in $Q_\xi[i,j]$ are as required.
				
In the latter case, we have $j\neq j_{\diamond}$, $j_{\bullet}=(j+1)_{\bullet}<j$, and then
\begin{align*}
& \max\{i-1, (j+1)_{\bullet}-1\}=\max\{i-1,j_{\bullet}-1\}, \quad  \max\{i,(j+1)_{\bullet}\}= \max\{i,j_{\bullet}\}, \\
& a_{j+1}=a_j, \quad b_{j+1}=b_j.
\end{align*}
By induction the arrows incident to $j$, $j+1$ are shown as follows:
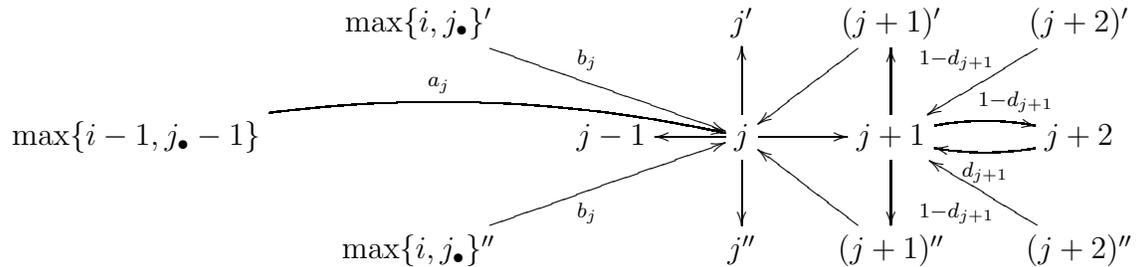
\[
\xymatrix{
& \max\{i,j_{\bullet}\}'\ar[drr]^{b_j}&&j' & (j+1)' \ar[ld]&  (j+2)'\ar[ld]_{1-d_{j+1}}\\
\max\{i-1,j_{\bullet}-1\}\ar@/^1.1pc/[rrr]^{a_j}&& j-1 &j\ar[l]\ar[d]\ar[u]\ar[r]& j+1 \ar[d]\ar[u]\ar@/^/[r]^{\quad \quad 1-d_{j+1}}& j+2 \ar@/^/[l]^{d_{j+1}}  \\
& \max\{i,j_{\bullet}\}''\ar[urr]_{b_j}& & j''& (j+1)'' \ar[lu]&(j+2)''\ar[lu]^{1-d_{j+1}}
}
\]
After mutating $j$, the arrows incident to $j+1$ are shown as follows:
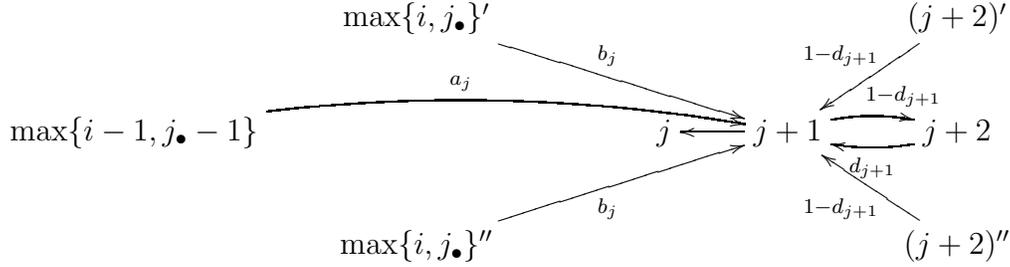
\[
\xymatrix{
& \max\{i,j_{\bullet}\}'\ar[drrr]^{b_j}&&  & &  (j+2)'\ar[ld]_{1-d_{j+1}}\\
\max\{i-1,j_{\bullet}-1\}\ar@/^1pc/[rrrr]^{a_j}&& &j  &\ar[l] j+1 \ar@/^/[r]^{\quad \quad 1-d_{j+1}}&  j+2 \ar@/^/[l]^{d_{j+1}}  \\
& \max\{i,j_{\bullet}\}''\ar[urrr]_{b_j}& &  &  &(j+2)''\ar[lu]^{1-d_{j+1}}
}
\]	
Note that  $j\rightarrow j+1$ in $Q_\xi$. The arrows incident to $j+1$ in $Q_\xi[i,j]$ are as required.
				
By the principle of induction, the assertion is true. The proof is completed.
\end{proof}
			
By Lemma \ref{arrows accident with j}, we have the following exchange relations in $\mathscr{A}(\mathbf{x},\Q)$.
			
\begin{corollary}\label{prop:mutation equations0}
Let $1\leq i<j\leq n$. Then
\begin{align*} 
\mathbf{x}_i \mathbf{x}[\alpha_{i,i}]=&\mathbf{f}'_i \mathbf{f}''_i \mathbf{x}^{1-d_i}_{i+1}+\mathbf{x}_{i-1} \mathbf{f}'^{1-d_i}_{i+1} \mathbf{f}''^{1-d_i}_{i+1}  \mathbf{x}^{d_i}_{i+1},\\
\begin{split} 
\mathbf{x}_j \mathbf{x}[\alpha_{i,j}]=  & \mathbf{f}'^{1-d_j}_{j+1} \mathbf{f}''^{1-d_j}_{j+1} \mathbf{x}^{d_j}_{j+1} 
\big( \mathbf{f}'^{\delta_{i,j_{\bullet}}}_{i} \mathbf{f}''^{\delta_{i,j_{\bullet}}}_{i} \mathbf{x}^{\delta_{i_\bullet,j_{\bullet}}}_{i-1} +(1-\delta_{i_\bullet,j_\bullet}-\delta_{i,j_{\bullet}}) \mathbf{f}'^{d_{j_\bullet-1}}_{j_\bullet}\mathbf{f}''^{d_{j_\bullet-1}}_{j_\bullet} \mathbf{x}[\alpha_{i,j_\bullet-1}]
\big)+\\
&\mathbf{x}[\alpha_{i,j-1}] {\mathbf{f}'_j}^{d_{j-1}}  \mathbf{f}''^{d_{j-1}}_{j}  \mathbf{x}^{1-d_{j}}_{j+1}.
\end{split}
\end{align*}
\end{corollary}
			
\section{$\HL$-modules}\label{The study on HLr modules}
In this section we define and study a family of simple $U_q(\widehat{\mathfrak{sl}_{n+1}})$-modules, and call them $\HL$-modules.
			
\subsection{Definition of $\HL$-modules}

\begin{definition}\label{definition of GHL}

An $\HL$-module is a simple $U_q(\widehat{\mathfrak{sl}_{n+1}})$-module with the highest $l$-weight monomial
\begin{align}\label{Monomial of HLr}
(Y_{i_1,a_1}Y_{i_2,a_2}\dots Y_{i_k,a_k})_r:=\prod^{k}_{s=1} X^{(a_s)}_{i_s,r}=\prod^{k}_{s=1}Y_{i_s,a_s} Y_{i_s,a_s+2} \cdots Y_{i_s,a_s+2r-2},
\end{align}
where $k,r\in\mathbb{Z}_{\geq 1}$, $i_j\in I, a_{j}\in \mathbb{Z}$ for $j=1,2,\dots,k$, and
\begin{enumerate}
\item[(1)]  $i_1<i_2<\dots<i_k$,
\item[(2)]  $(a_{j}-a_{j-1})(a_{j+1}-a_{j})<0$ for $2\leq j\leq k-1$,
\item[(3)]  $\lvert a_{j}-a_{j-1}\rvert=i_{j}-i_{j-1}+2$ for $2\leq j\leq k$.
\end{enumerate}
\end{definition}
			
Comparing Lemma \ref{def:HL} with Definition \ref{definition of GHL}, HL-modules are $\HL$-modules. When $r=1$, $\HL$-modules coincide with HL-modules, and hence we view $\HL$-modules as a generalization of HL-modules. It is obvious that Kirillov-Reshetikhin modules are $\HL$-modules; furthermore, an $\HL$-module is a Kirillov-Reshetikhin module if and only if $k=1$. It is easy to see that an $\HL$-module is a prime snake module if and only if  $k=1$ or $k=2$ and $r=1$. The $\HL$-modules are not same as snake modules, for example, when $k=3$ and $r=2$, the module $L(Y_{2,-6}  Y_{2,-4}  Y_{4,-2}  Y_{4,0}  Y_{6,-6}  Y_{6,-4})$ is an $\HL$-module, but not a (prime) snake module; $L(Y_{1,-7}  Y_{2,-4}  Y_{3,-1})$ is a (prime) snake module, but not an $\HL$-module.

\subsection{Mutation sequences}\label{section of mutation sequence}
Let $\mathscr{A}(\mathbf{x},Q_\xi)$ be the cluster algebra with initial seed $(\mathbf{x},Q_\xi)$ defined in Section \ref{The construction of the cluster algebra}. Fix a positive integer $r$, we choose a large enough $\ell$. 
			
The initial quiver $Q_\ell$ of $K_0(\mathscr{C}_{\ell})$ as a cluster algebra is defined in Section \ref{K0Ccluster algebra}. In the following, for any height function $\xi$, we define a mutation sequence such that after mutating the sequence on $Q_\ell$, a subquiver with vertices at the $(r-1)$th, $r$th, and $(r+1)$th rows of the resulting quiver is the same as our quiver $Q_\xi$ if the vertices at the $(r-1)$th and $(r+1)$th rows are viewed as frozen vertices. In particular, vertices at the $(r-1)$th (respectively, $r$th, $(r+1)$th) row of the resulting quiver correspond precisely to the vertices $1',\dots,n'$ (respectively, $1,\dots,n$, $1'',\dots,n''$) of $Q_\xi$.
			
We use $(i,k)$ to denote the $k$th vertex in $i$th column of $Q_\ell$, counting from the top. Let
\[
s_{k,1}=\{(1,k), (3,k),(5,k),\dots\} \cap V_{\ell},\quad s_{k,2}=\{(2,k), (4,k),(6,k),\dots\} \cap V_{\ell}.
\]
We define an infinite sequence as follows.
\begin{align}\label{sequence of vertices of G-}
\begin{split}
& s_{1,2}, \\
&s_{2,2}, s_{1,1}\\
&s_{3,2}, s_{2,1}, s_{1,2},\\
&s_{4,2}, s_{3,1}, s_{2,2},s_{1,1},\\
&s_{5,2}, s_{4,1}, s_{3,2},s_{2,1},s_{1,2},\\
& \cdots  \quad  \cdots \quad  \cdots.
\end{split}
\end{align}
			
Let $\xi$ be a height function and $r\geq1$. We define $\mathscr{S}$ to be the finite sequence taken the first $p$ rows of (\ref{sequence of vertices of G-}), with  $p\in\mathbb{Z}_{\geq1}$, where $p$ is large enough and satisfies that
\begin{equation*}
\begin{cases}
p\equiv n\pmod 2 & \text{if $\xi(n-1)>\xi(n)$, $r$ is odd, or $\xi(n-1)<\xi(n)$, $r$ is even}, \\
p\equiv n-1\pmod 2 & \text{if $\xi(n-1)>\xi(n)$,  $r$ is even, or $\xi(n-1)<\xi(n)$, $r$ is odd}.
\end{cases}
\end{equation*}
Let $\{j_1,\dots,j_s\}\subset I$ such that $j_1<j_2<\dots<j_s$ and $d_{j_1}=d_{j_2}=\dots=d_{j_s}=0$, where $d_i=\delta_{i,i_\diamond}$ for $i\in I$. For $1\leq t\leq s$, $j_t\equiv 0 \pmod 2$, let
\begin{align*}
\mathscr{S}_t= \{&\dots,s_{r-3,2},s_{r-1,2},s_{r+1,2},s_{r+3,2},\dots,\\
&\dots,s_{r-4,1},s_{r-2,1},s_{r,1},s_{r+2,1},\dots
\} \cap  ([1,j_t] \times \mathbb{Z}_{\geq1})\cap V_{\ell}.
\end{align*}
For $1\leq t\leq s$, $j_t\equiv 1 \pmod 2$, let
\begin{align*}
\mathscr{S}_t= \{&\dots,s_{r-3,1},s_{r-1,1},s_{r+1,1},s_{r+3,1},\dots,\\
&\dots,s_{r-4,2},s_{r-2,2},s_{r,2},s_{r+2,2},\dots
\} \cap  ([1,j_t] \times \mathbb{Z}_{\geq1})\cap V_{\ell}.
\end{align*}
			
Following \cite{KNS94}, the $T$-system equations in type $A_n$ are of the following form
\begin{align*}
T^{(i)}_{k,r+1} T^{(i)}_{k,r-1}  = T^{(i)}_{k-1,r+1} T^{(i)}_{k+1,r-1} + \prod_{j: c_{ij}=-1} T^{(j)}_{k,r},
\end{align*}
where $i\in I$, $k\geq 1$, $r\in \mathbb{Z}$, and $T^{(i)}_{0,r}=1$.
			
Mutating the sequence $\mathscr{S}$ in the initial seed of $K_0(\mathscr{C}_{\ell})$, then all of the exchange relations among cluster variables in the process of mutations are from $T$-system equations, denote the resulting seed by $\mu_{\mathscr{S}}(\seed)$. Then the seed $\mu_{\mathscr{S}}(\seed)$ is shown in Figure \ref{figure obtained 1117} or Figure \ref{figure obtained 11172} up to spectral parameter. 			
\begin{figure}
\centerline{
\resizebox{12cm}{!}{
\xymatrix{
[L(X^{(-2)}_{1,1})]  \ar[d] &  [L(X^{(-1)}_{2,1})] \ar[l] \ar[r]& [L(X^{(-2)}_{3,1})] \ar[d] &  [L(X^{(-1)}_{4,1})] \ar[r]\ar[l] & [L(X^{(-2)}_{5,1})] \ar[d] &  [L(X^{(-1)}_{6,1})] \ar[l]\ar[r] &\cdots \\
[L(X^{(-2)}_{1,2})] \ar[r]  &  [L(X^{(-3)}_{2,2})]  \ar[d]\ar[u] & [L(X^{(-2)}_{3,2})] \ar[l]\ar[r] & [L( X^{(-3)}_{4,2} )]\ar[d]\ar[u]  &  [L(X^{(-2)}_{5,2})]  \ar[l]\ar[r]&  [L(X^{(-3)}_{6,2})] \ar[d]\ar[u]&\cdots\ar[l]\\
[L(X^{(-4)}_{1,3})] \ar[d]\ar[u]&  [L(X^{(-3)}_{2,3})] \ar[l]\ar[r]  & [L(X^{(-4)}_{3,3})]  \ar[u]\ar[d] &  [L(X^{(-3)}_{4,3})] \ar[l]\ar[r] & [L(X^{(-4)}_{5,3})] \ar[d]\ar[u] & [L(X^{(-3)}_{6,3} )] \ar[l]\ar[r]&\cdots\\
[L(X^{(-4)}_{1,4})] \ar[r] &  [L(X^{(-5)}_{2,4})] \ar[d]\ar[u] & [L(X^{(-4)}_{3,4})]  \ar[l]\ar[r] & [L(X^{(-5)}_{4,4})] \ar[u]\ar[d] & [L(X^{(-4)}_{5,4})] \ar[l]\ar[r] & [L(X^{(-5)}_{6,4})] \ar[d]\ar[u] &\cdots\ar[l]\\
[L(X^{(-6)}_{1,5})] \ar[d]\ar[u] & [L(X^{(-5)}_{2,5})] \ar[l]\ar[r]  & [L(X^{(-6)}_{3,5})] \ar[d]\ar[u] & [L(X^{(-5)}_{4,5})]\ar[l]\ar[r]   & [L(X^{(-6)}_{5,5})] \ar[d]\ar[u]& [L( X^{(-5)}_{6,5})] \ar[r]\ar[l] &\cdots\\
\vdots  & \vdots\ar[u]  & \vdots & \vdots \ar[u]  & \vdots & \vdots\ar[u] &\cdots\\
}
}}
\caption{The seed $\mu_{\mathscr{S}}(\seed)$.}\label{figure obtained 1117}
\end{figure}
			
\begin{figure}
\centerline{
\resizebox{12cm}{!}{
\xymatrix{
[L(X^{(-1)}_{1,1})]  \ar[r] & [L(X^{(-2)}_{2,1})]  \ar[d] & [L(X^{(-1)}_{3,1})] \ar[r]\ar[l] &  [L(X^{(-2)}_{4,1})] \ar[d]  &  [L(X^{(-1)}_{5,1})] \ar[r]\ar[l] &  [L(X^{(-2)}_{6,1})] \ar[d] &\cdots\ar[l] \\
[L(X^{(-3)}_{1,2})] \ar[d]\ar[u]  & [L(X^{(-2)}_{2,2})] \ar[l]\ar[r] &  [L(X^{(-3)}_{3,2})] \ar[d]\ar[u]&  [L(X^{(-2)}_{4,2})] \ar[l]\ar[r] &  [L(X^{(-3)}_{5,2})] \ar[d]\ar[u]&  [L(X^{(-2)}_{6,2})] \ar[r]\ar[l] &\cdots \\
[L(X^{(-3)}_{1,3})] \ar[r] &  [L(X^{(-4)}_{2,3})] \ar[d]\ar[u] & [L(X^{(-3)}_{3,3})]  \ar[l]\ar[r]  &  [L(X^{(-4)}_{4,3})] \ar[u]\ar[d] & [L(X^{(-3)}_{5,3})] \ar[l]\ar[r] &  [L(X^{(-4)}_{6,3})]  \ar[d]\ar[u]&\cdots\ar[l]\\
[L(X^{(-5)}_{1,4})] \ar[d]\ar[u]& [L(X^{(-4)}_{2,4})] \ar[l]\ar[r] & [L(X^{(-5)}_{3,4})] \ar[u]\ar[d] &  [L(X^{(-4)}_{4,4})] \ar[l]\ar[r] & [L(X^{(-5)}_{5,4})] \ar[d]\ar[u]& [L(X^{(-4)}_{6,4})]  \ar[l]\ar[r]  &\cdots \\
[L(X^{(-5)}_{1,5})] \ar[r] & [L(X^{(-6)}_{2,5} )]\ar[d]\ar[u]  & [L(X^{(-5)}_{3,5})] \ar[l]\ar[r]& [L(X^{(-6)}_{4,5})] \ar[d]\ar[u] & [L(X^{(-5)}_{5,5})] \ar[l]\ar[r] & [L(X^{(-6)}_{6,5})] \ar[d]\ar[u] &\cdots \ar[l] \\
\vdots\ar[u]   & \vdots & \vdots\ar[u] & \vdots   & \vdots\ar[u] & \vdots &\cdots\\
}
}}
\caption{The seed $\mu_{\mathscr{S}}(\seed)$.}\label{figure obtained 11172}
\end{figure}

Let $(\mathbf{x},Q)$ be a seed of rank $n$. Following \cite{FWZ17}, \textit{freezing} at the index $i$ of a cluster variable $x_i$ is a transformation of the seed that reclassifies $i$ and $x_i$ as frozen, and accordingly removes all arrows connecting frozen vertices to each other. Let $I \cup J$ be a partition of the index set of $\mathbf{x}$ such that there are no arrows connecting the intersection of $I$ and the index set of cluster variables in $\mathbf{x}$ with $J$ to each other. The pair $(\mathbf{x}_I, Q_I)$ is called a \textit{restricted seed} of $(\mathbf{x},Q)$, where $\mathbf{x}_I = (x_i)_{i\in I}$, and $Q_I$ is a full subquiver of $Q$ with vertex set $I$.
 
\begin{definition}[{\cite[Definition 4.2.6]{FWZ17}}]\label{Definition of FWZ3}
Let $\Sigma$ be a seed. Freeze some subset of cluster variables to obtain a seed $\Sigma'$. Next, restrict $\Sigma'$ to obtain a restricted seed $\Sigma''$. We then say that the seed pattern defined by $\Sigma''$ is a \textit{seed subpattern} of the seed pattern defined by $\Sigma$; and the cluster algebra associated to $\Sigma''$ is a \textit{cluster subalgebra} of the cluster algebra associated to $\Sigma$.
\end{definition} 

We further mutate the seed $\mu_{\mathscr{S}}(\seed)$ using the sequences $\mathscr{S}_s, \mathscr{S}_{s-1}, \ldots, \mathscr{S}_1$ in order, and denote the new obtained seed by 
\begin{align} \label{the new obtained seed before HLr}
\mathscr{S}_m(\seed):=\mu_{\mathscr{S}_1} \circ \cdots \circ \mu_{\mathscr{S}_s} \circ \mu_{\mathscr{S}}(\seed).
\end{align}
The following theorem describes the subseed indexed by vertices at $(r-1)$th, $r$th, $(r+1)$th rows of the quiver in $\mathscr{S}_m(\seed)$.

\begin{theorem}\label{mutation equivalent}
The cluster algebra $\mathscr{A}(\mathbf{x},\Q)$ is isomorphic to a cluster subalgebra of $K_0(\mathscr{C}_\ell)$.
\end{theorem}
			
\begin{proof}
Recall that in \cite{KNS94}, the $T$-system equations in type $A_n$ are of the following form
\begin{equation}\label{equ:T-system equation}
T^{(i)}_{k,r+1} T^{(i)}_{k,r-1}  = T^{(i)}_{k-1,r+1} T^{(i)}_{k+1,r-1} + \prod_{j: c_{ij}=-1} T^{(j)}_{k,r},
\end{equation}
where $i\in I$, $k\geq 1$, $r\in \mathbb{Z}$, and $T^{(i)}_{0,r}=1$.
All of the exchange relations among the cluster variables in the aforementioned mutation sequence are from $T$-system equations with respect to the initial seed of $\seed$.
				
It follows from Figures \ref{figure obtained 1117} and \ref{figure obtained 11172} that the seed $\mu_{\mathscr{S}}(\seed)$ is shown in Figure \ref{seed, fist} up to spectral parameter. In this case, the local seed of $\mathscr{S}_m(\seed)$ corresponds to that of $\mathscr{A}(\mathbf{x}, Q_\xi)$, where $\xi$ is a source-sink height function.
\begin{figure}
\center
\resizebox{14cm}{!}{
\xymatrix{
\cdots&\vdots\ar[d]&\vdots&\vdots\ar[d]&\vdots&\cdots\\
\cdots & [L((Y_{j-2,\xi(j-2)+3})_{r-2})] \ar[r] \ar[l]&  [L((Y_{j-1,\xi(j-1)+1})_{r-2})]\ar[d]\ar[u] & [L((Y_{j,\xi(j)+3})_{r-2})] \ar[l]\ar[r] & [L((Y_{j+1,\xi(j+1)+1})_{r-2})] \ar[d]\ar[u] & \cdots\ar[l] \\
\cdots\ar[r] & [L((Y_{j-2,\xi(j-2)+1})_{r-1})]\ar[u]\ar[d] & [L((Y_{j-1,\xi(j-1)+1})_{r-1})] \ar[r]\ar[l] &  [L((Y_{j,\xi(j)+1})_{r-1})]   \ar[u]\ar[d]   & [L((Y_{j+1,\xi(j+1)+1})_{r-1})] \ar[l]\ar[r] & \cdots\\
\cdots & [L((Y_{j-2,\xi(j-1)})_r)]\ar[l]\ar[r] & [L((Y_{j-1,\xi(j)})_r)] \ar[u] \ar[d] & [L((Y_{j,\xi(j+1)})_r)]  \ar[r] \ar[l] & [L((Y_{j+1,\xi(j+2)})_r)] \ar[u]  \ar[d] & \cdots\ar[l]\\
\cdots\ar[r] &[L((Y_{j-2,\xi(j-2)-1})_{r+1})] \ar[u]\ar[d] &  [L((Y_{j-1,\xi(j-1)-1})_{r+1})] \ar[r]\ar[l] & [L((Y_{j,\xi(j)-1})_{r+1})]  \ar[u] \ar[d]& [L((Y_{j+1,\xi(j+1)-1})_{r+1})] \ar[l]\ar[r] & \cdots\\
\cdots & [L((Y_{j-2,\xi(j-2)-1})_{r+2})] \ar[l]\ar[r] &  [L((Y_{j-1,\xi(j-1)-3})_{r+2})] \ar[u] \ar[d] & [L((Y_{j,\xi(j)-1})_{r+2})]  \ar[r] \ar[l] & [L((Y_{j+1,\xi(j+1)-3})_{r+2})] \ar[u]  \ar[d] & \cdots\ar[l]\\
\cdots&\vdots\ar[u]&\vdots&\vdots\ar[u]&\vdots&\cdots
}}
\caption{The local seed corresponds to that of $\mathscr{A}(\mathbf{x}, Q_\xi)$, where $\xi(j)=\xi(j+2)$ for $1\leq j\leq n-2$.}\label{seed, fist}
\end{figure}
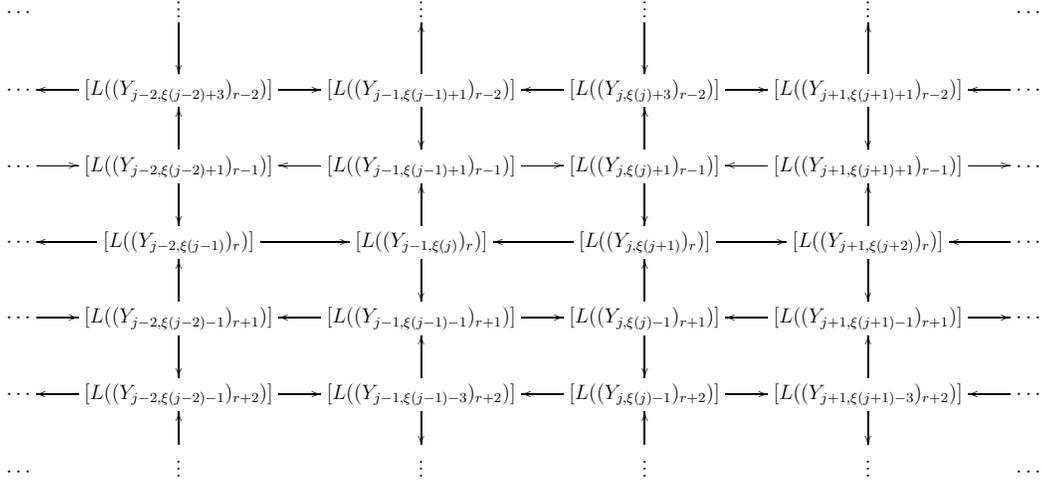
				
Assume that $j=j_s$. Then we proceed to mutate the seed $\mu_{\mathscr{S}}(\seed)$ using the sequence $\mathscr{S}_{s}$ and the local picture of the seed $\mu_{\mathscr{S}_{s}}\circ \mu_{\mathscr{S}}(\seed)$ is shown in Figure \ref{figure obtained by sequence j}. In this case, the local seed corresponds to that of  $\mathscr{A}(\mathbf{x}, Q_\xi)$, where $\xi(j-2)-1=\xi(j-1)= \xi(j)-1=\xi(j+1)$.
				
Assume that $j-1=j_s$. Then we proceed to mutate the seed $\mu_{\mathscr{S}}(\seed)$ using the sequence $\mathscr{S}_{s}$ and the local picture of the seed $\mu_{\mathscr{S}_{s}}\circ \mu_{\mathscr{S}}(\seed)$ is shown in Figure \ref{figure obtained by sequence j-1}. In this case, the local seed corresponds to that of  $\mathscr{A}(\mathbf{x}, Q_\xi)$, where $\xi(j-2)-1=\xi(j-1)= \xi(j)-1=\xi(j+1)-2$.
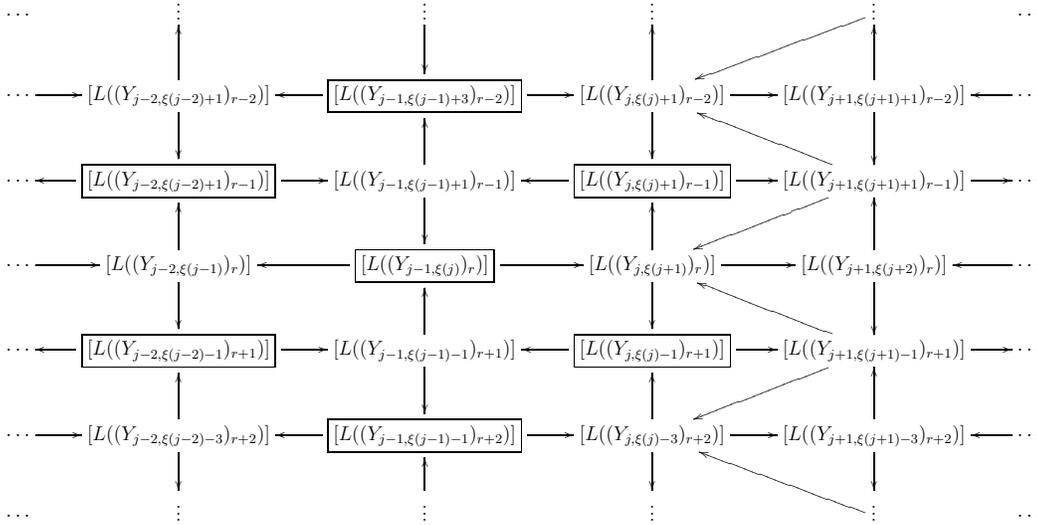
\begin{figure}
\centerline{
\resizebox{14cm}{!}{
\xymatrix{
\cdots&\vdots&\vdots\ar[d]&\vdots&\vdots\ar[ld]&\cdots\\
\cdots\ar[r] & [L((Y_{j-2,\xi(j-2)+1})_{r-2})] \ar[u] \ar[d]& \boxed{[L((Y_{j-1,\xi(j-1)+3})_{r-2})]}\ar[l]\ar[r] & [L((Y_{j,\xi(j)+1})_{r-2})] \ar[r]\ar[u]\ar[d] & [L((Y_{j+1,\xi(j+1)+1})_{r-2})] \ar[d]\ar[u] & \cdots\ar[l] \\
\cdots & \boxed{[L((Y_{j-2,\xi(j-2)+1})_{r-1})]}\ar[l]\ar[r] & [L((Y_{j-1,\xi(j-1)+1})_{r-1})]  \ar[u]\ar[d] & \boxed{[L((Y_{j,\xi(j)+1})_{r-1})]}  \ar[l]\ar[r]   & [L((Y_{j+1,\xi(j+1)+1})_{r-1})] \ar[ld]\ar[lu] \ar[r] & \cdots\\
\cdots \ar[r]& [L((Y_{j-2,\xi(j-1)})_r)] \ar[u]\ar[d] & \boxed{[L((Y_{j-1,\xi(j)})_r)]} \ar[r] \ar[l] & [L((Y_{j,\xi(j+1)})_r)] \ar[r] \ar[u] \ar[d] & [L((Y_{j+1,\xi(j+2)})_r)] \ar[u]  \ar[d] & \cdots\ar[l]\\
\cdots & \boxed{[L((Y_{j-2,\xi(j-2)-1})_{r+1})]} \ar[l]\ar[r] & [L((Y_{j-1,\xi(j-1)-1})_{r+1})] \ar[u]\ar[d] & \boxed{[L((Y_{j,\xi(j)-1})_{r+1})]} \ar[l] \ar[r]& [L((Y_{j+1,\xi(j+1)-1})_{r+1})] \ar[ld]\ar[lu]\ar[r] & \cdots\\
\cdots\ar[r] & [L((Y_{j-2,\xi(j-2)-3})_{r+2})] \ar[u]\ar[d] & \boxed{[L((Y_{j-1,\xi(j-1)-1})_{r+2})]} \ar[l] \ar[r] & [L((Y_{j,\xi(j)-3})_{r+2})] \ar[r] \ar[u] \ar[d] & [L((Y_{j+1,\xi(j+1)-3})_{r+2})] \ar[u]  \ar[d] & \cdots\ar[l]\\
\cdots&\vdots&\vdots\ar[u]&\vdots&\vdots\ar[lu]&\cdots\\
}}}
\caption{The local seed corresponds to $(\mathbf{x},\Q)$, where $\xi(j-2)-1=\xi(j-1)= \xi(j)-1=\xi(j+1)$.}\label{figure obtained by sequence j}
\end{figure}
\begin{figure}
\centerline{
\resizebox{14cm}{!}{
\xymatrix{
\cdots&\vdots &\vdots\ar[d]&\vdots\ar[d]&\vdots&\cdots\\
\cdots\ar[r] & \boxed{[L((Y_{j-2,\xi(j-2)+1})_{r-2})]}\ar[d] \ar[u]& [L((Y_{j-1,\xi(j-1)+3})_{r-2})] \ar[rd]\ar[ru]\ar[l] &  [L((Y_{j,\xi(j)+3})_{r-2})] \ar[l]\ar[r] & [L((Y_{j+1,\xi(j+1)+1})_{r-2})] \ar[d]\ar[u] & \cdots\ar[l] \\
\cdots & [L((Y_{j-2,\xi(j-2)+1})_{r-1})] \ar[l]\ar[r] & \boxed{[L((Y_{j-1,\xi(j-1)+1})_{r-1})]} \ar[d]\ar[u] & [L((Y_{j,\xi(j)+1})_{r-1})]\ar[l]  \ar[u]\ar[d]   & [L((Y_{j+1,\xi(j+1)+1})_{r-1})] \ar[l]\ar[r] & \cdots\\
\cdots\ar[r] & \boxed{[L((Y_{j-2,\xi(j-1)})_{r})]} \ar[d]\ar[u] & [L((Y_{j-1,\xi(j)})_{r})] \ar[ru] \ar[rd]\ar[l] & [L((Y_{j,\xi(j+1)})_r)]  \ar[r] \ar[l] & [L((Y_{j+1,\xi(j+2)})_r)] \ar[u]  \ar[d] & \cdots\ar[l]\\
\cdots  & [L((Y_{j-2,\xi(j-2)-1})_{r+1})] \ar[l]\ar[r] & \boxed{[L((Y_{j-1,\xi(j-1)-1})_{r+1})]}\ar[d]\ar[u] &  [L((Y_{j,\xi(j)-1})_{r+1})] \ar[l]  \ar[u] \ar[d]& [L((Y_{j+1,\xi(j+1)-1})_{r+1})] \ar[l]\ar[r] & \cdots\\
\cdots\ar[r] & \boxed{[L((Y_{j-2,\xi(j-2)-3})_{r+2})]}\ar[u]\ar[d] & [L((Y_{j-1,\xi(j-1)-1})_{r+2})] \ar[l]\ar[ru] \ar[rd] & [L((Y_{j,\xi(j)-1})_{r+2})] \ar[r] \ar[l] & [L((Y_{j+1,\xi(j+1)-3})_{r+2})] \ar[u]  \ar[d] & \cdots\ar[l]\\				\cdots&\vdots&\vdots\ar[u]&\vdots\ar[u]&\vdots&\cdots\\
}}}
\caption{The local seed corresponds to that of $\mathscr{A}(\mathbf{x}, Q_\xi)$, where $\xi(j-2)-1=\xi(j-1)= \xi(j)-1=\xi(j+1)-2$.}\label{figure obtained by sequence j-1}
\end{figure}
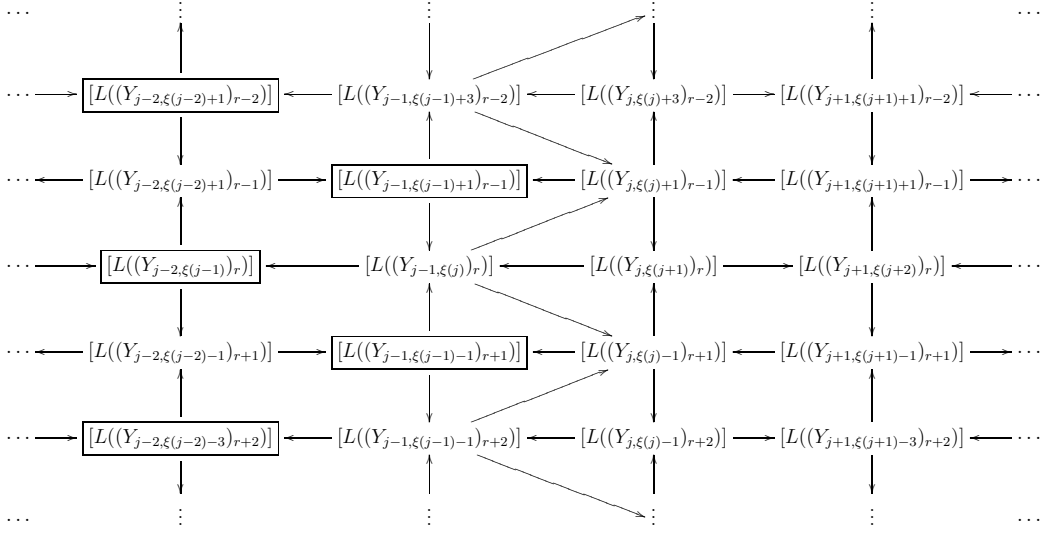
				
Continuing to mutate the seed $\mu_{\mathscr{S}_{s}}\circ \mu_{\mathscr{S}}(\seed)$ using the sequences $\mathscr{S}_{s-1}, \mathscr{S}_{s-2}, \dots,$ $\mathscr{S}_{1}$ in order, then we obtain the seed $\mathscr{S}_m(\seed)$. Based on the argument above, we can obtain $Q_\xi$ from the quiver of $\mathscr{S}_m(\seed)$ by following two steps:
\begin{itemize}
\item[(a)] freezing the vertices at the $(r-1)$th, $(r+1)$th rows, denote the obtained seed as $\Sigma'$.
\item[(b)] restricting to the vertices at the $(r-1)$th, $r$th, $(r+1)$th rows of the seed $\Sigma'$, denote the obtained seed as $\Sigma''$. 
\end{itemize}				

We define the set $I$ to represent the index set of the vertices at the $(r-1)$th, $r$th, $(r+1)$th rows of the quiver of $\Sigma'$, define the set $J$ to represent the index set of the vertices of the quiver of $\Sigma'$ that not in the set $I$, then there are no arrows connecting the intersection of $I$ and the index set of cluster variables with $J$ to each other in seed $\Sigma'$. According to  Definition \ref{Definition of FWZ3}, $\Sigma''$  is a seed subpattern of the seed $\mathscr{S}_m(\seed)$, the cluster algebra associated to $\Sigma''$ is a cluster subalgebra of the cluster algebra $K_0(\mathscr{C}_\ell)$ associated to $\mathscr{S}_m(\seed)$. 

Thus the seed $(\mathbf{x},\Q)$ is isomorphic to a seed subpattern of the seed $\mathscr{S}_m(\seed)$, the cluster algebra $\mathscr{A}(\mathbf{x},\Q)$ is isomorphic to a cluster subalgebra of the cluster algebra $K_0(\mathscr{C}_\ell)$.
\end{proof}

In the seed $\mathscr{S}_m(\seed)$, we denote by $[L(f'_i)]$ the cluster variable corresponding to the vertex $(i,r-1)$, by $[L(x_i)]$ the cluster variable corresponding to the vertex $(i,r)$, and by $[L(f''_i)]$ the cluster variable corresponding to the vertex $(i,r+1)$, $i\in I$, where $f'_{i}= (Y_{i,\xi(i)+1})_{r-1}, x_i=(Y_{i,\xi(i+1)})_r, f''_{i}= (Y_{i,\xi(i)-1})_{r+1}$ up to spectral parameter. It follows from Lemma \ref{theorem of real prime} that the image of $\mathbf{x}[\alpha_{i,j}]$ under the isomorphism in Theorem \ref{mutation equivalent}  is the equivalence class of a simple module, denoted by $[L(x[\alpha_{i,j}])]$. The isomorphism in Theorem \ref{mutation equivalent} is explicitly given by
\[
\mathbf{x}_{i} \to [L(x_i)], \quad \mathbf{f}'_{i} \to [L(f'_i)], \quad \mathbf{f}''_{i} \to [L(f''_i)], \quad \mathbf{x}[\alpha_{i,j}] \to [L(x[\alpha_{i,j}])],
\]
where $x_i=(Y_{i,\xi(i+1)})_r, f'_{i}= (Y_{i,\xi(i)+1})_{r-1}, f''_{i}= (Y_{i,\xi(i)-1})_{r+1}$ up to spectral parameter.
			
Applying the isomorphism in Theorem \ref{mutation equivalent} to the equations in Corollary \ref{prop:mutation equations0}, we obtain some exchange relations in $\seed$ as follows.

\begin{corollary} \label{prop:mutation equations}
Let $1\leq i<j\leq n$. Then
\begin{gather}
\begin{aligned}\label{mutation equationsi=i} 
[L(x_i)] [L(x[\alpha_{i,i}])] & = [L(f'_i)] [L(f''_i)] [L(x_{i+1})]^{1-d_i} +[L(x_{i-1})] [L(f'_{i+1})]^{1-d_i} [L(f''_{i+1})]^{1-d_i} [L(x_{i+1})]^{d_i},
\end{aligned}
\end{gather}

\begin{gather}
\begin{aligned}\label{mutation equations} 
[L(x_j)] [L(x[\alpha_{i,j}])] & =  [L(x[\alpha_{i,j-1}])] [L(f'_j)]^{d_{j-1}} [L(f''_j)]^{d_{j-1}}   [L(x_{j+1})]^{1-d_{j}} +  [L(f'_{j+1})]^{1-d_j} [L(f''_{j+1})]^{1-d_j}    \\ 
& \quad \,\, [L(x_{j+1})]^{d_j} \big(  [L(f'_i)]^{\delta_{i,j_{\bullet}}} [L(f''_i)]^{\delta_{i,j_{\bullet}}} [L(x_{i-1})]^{\delta_{i_\bullet,j_{\bullet}}} + (1-\delta_{i_\bullet,j_\bullet}-\delta_{i,j_{\bullet}}) [L(f'_{j_\bullet})]^{d_{j_\bullet-1}}  \\
& \quad \,\,  [L(f''_{j_\bullet})]^{d_{j_\bullet-1}}  [L(x[\alpha_{i,j_\bullet-1}])] \big). 
\end{aligned}
\end{gather}
\end{corollary}

\subsection{Equivalence classes of $\HL$-modules as cluster variables}
Recall that the definitions of $\xi$ and $\omega(i,j)$ in Section \ref{section:recall HL-module}. The following lemmas provide some properties of $(\omega(i,j))_r$ that we will use in the sequel.
			
\begin{lemma}\label{equation of omega1}
If $i_\diamond<j-1$ and $j-1 \neq (j-1)_{\diamond}$, then
\[
(\omega(i,j))_r = (\omega(i,j_{\bullet}+1))_r(Y_{j,\xi(j+1)})_r.
\]
\end{lemma}
			
\begin{proof}
We only prove the case that $\xi(j-1)<\xi(j)$, the case that $\xi(j-1)>\xi(j)$ is similar.
				
Assume  that $\omega(i,j)=Y_{i_1,a_1} \cdots Y_{i_k,a_k}$ with  $i=i_1<i_2<\dots<i_k=j$. By definition,
\[
(\omega(i,j))_r =(Y_{i_1,a_1} \cdots Y_{i_k,a_k})_r.
\]
It follows from the definition of $\omega(i,j)$ that $i_{k-1}=j_{\bullet}+1$. Since $\xi(j-1)<\xi(j)$, we have $a_k=\xi(j)+1$. Moreover, $j-1\neq (j-1)_{\diamond}$ and $\xi(j-1)<\xi(j)$ imply that $\xi(j)+1=\xi(j+1)$. Therefore
 \[
(\omega(i,j))_r =(Y_{i_1,a_1} \cdots Y_{i_{k-1},a_{k-1}})_r (Y_{i_k,a_k})_r=(\omega(i,j_{\bullet}+1))_r (Y_{j,\xi(j+1)})_r.
\]
\end{proof}
			
\begin{lemma}\label{equation of omega2}
If $i_\diamond<j-1$ and $j-1=(j-1)_{\diamond}$, then
\[
(\omega(i,j))_r = (\omega(i,(j-1)_{\bullet}+1))_r (Y_{j, \xi(j)\pm 1})_r
\]
if $\xi(j-1)\pm1=\xi(j)$.
\end{lemma}
			
\begin{proof}
We only prove the case that $\xi(j-1)<\xi(j)$, the case that $\xi(j-1)>\xi(j)$ is similar.
				
Assume that $\omega(i,j)=Y_{i_1,a_1} \cdots Y_{i_k,a_k}$ with  $i=i_1<i_2<\dots<i_k=j$. By definition,
\[
(\omega(i,j))_r =(Y_{i_1,a_1} \cdots Y_{i_k,a_k})_r.
\]
It follows from the definition of $\omega(i,j)$ that $i_{k-1}=(j-1)_{\bullet}+1$. Since $\xi(j-1)<\xi(j)$, we have $a_k=\xi(j)+1$.  Therefore
\[
 (\omega(i,j))_r =(Y_{i_1,a_1} \cdots Y_{i_{k-1},a_{k-1}})_r (Y_{i_k,a_k})_r=(\omega(i,(j-1)_{\bullet}+1))_r (Y_{j,\xi(j)+1})_r.
\]
\end{proof}
						
The following proposition is very useful in the sequel.

\begin{proposition}\label{lemma:equations satisfy by q-character}
In the cluster algebra $\seed$, each cluster exchange relation is of the following form:
\[
[L(m_1)] [L(m_2)]= [L(m_{a_1})] \cdots[L(m_{a_s})] +  [L(m_{b_1})]\cdots[L(m_{b_t})],
\]
where $s,t\in\mathbb{Z}_{\geq1}$, and $m_1,m_2, m_{a_1},\ldots,m_{a_s}, m_{b_1},\ldots,m_{b_t}$ are in $\mathcal{P}^{+}$. Moreover, $m_1m_2$ equals to the maximum monomial between $m_{a_1} \cdots m_{a_s}$ and  $m_{b_1}  \cdots  m_{b_t}$ with respect to $\leq$ defined in (\ref{partial order on weight lattice}).
\end{proposition}

\begin{proof}
It follows from \cite{Q17} or independently \cite{KKKO18} that cluster monomials in $\seed$ are equivalence classes of real simple modules in $\mathscr{C}_\ell$. Hence we obtain the required cluster exchange relations:
\begin{align}\label{equ: equation satisfied by q-ch}
\begin{split}
\chi_q(L(m_1)) \chi_q(L(m_2))= \chi_q(L(m_{a_1})) \cdots \chi_q(L(m_{a_s})) + \chi_q(L(m_{b_1})) \cdots \chi_q(L(m_{b_t})).
\end{split}
\end{align}

By Equation (\ref{Frenkel-Mukhin formula}), we know that $m_1 m_2$ is the maximum monomial with respect to $\leq$ among monomials in $\chi_q(L(m_1)) \chi_q(L(m_2))$. So $m_1 m_2$ is the maximum monomial among monomials in
\[
\chi_q(L(m_{a_1})) \ldots \chi_q(L(m_{a_s})) +  \chi_q(L(m_{b_1})) \ldots \chi_q(L(m_{b_t})).
\]
Thus $m_1m_2$ equals to the maximum monomial between $m_{a_1} \cdots m_{a_s}$ and  $m_{b_1}  \cdots  m_{b_t}$.
\end{proof}

Proposition \ref{lemma:equations satisfy by q-character} provides a method to compute the highest $l$-weight monomial $m_2$ of the new cluster variable $[L(m_2)]$ in $\seed$ once the cluster variables $[L(m_1)], [L(m_{a_1})], \dots$, $[L(m_{a_s})], [L(m_{b_1})],\dots,[L(m_{b_t})]$ are known by an iteration from initial Kirillov-Reshetikhin modules.

\begin{lemma}\label{Lem:Hight weight monomial}
Let $\xi$ be a height function and $r\in\mathbb{Z}_{\geq1}$. Then for $1\leq i\leq j\leq n$, $L(x[\alpha_{i,j}])$ is an $\HL$-module with the highest $l$-weight monomial
\[
\begin{cases} 
(Y_{i,\xi(i)\pm1})_r\quad\quad&\text{if}~ j=i_{\diamond}~\text{and}~\xi(i)=\xi(i+1)\pm1,\\
(\omega(i,\overline{j}))_r \quad\quad&\text{if}~j\neq i_{\diamond},
\end{cases}
\]
where $\overline{j}=(j+1)(1-\delta_{j,j_{\diamond}})+(j_{\bullet}+1)\delta_{j,j_{\diamond}}$.
\end{lemma}
				
\begin{proof}
We prove it for $\xi(j)<\xi(j+1)$ by induction on $j-i$. The case that $\xi(j)>\xi(j+1)$ is similar.
				
Assume that $j=i$. Then $\xi(i)<\xi(i+1)$ by hypothesis. If $j=i_{\diamond}$, then by Equation~(\ref{mutation equationsi=i}) we have
\[
[L(x_i)] [L(x[\alpha_{i,i}])] =[L(f'_i)] [L(f''_i)] + [L(x_{i-1})] [L(x_{i+1})].
\]
Since $x_i \nmid x_{i-1} x_{i+1}$, we have $x_i x[\alpha_{i,i}]=f'_i f''_i$ by Proposition \ref{lemma:equations satisfy by q-character}.  Hence
\[
x[\alpha_{i,i}]=\frac{f'_i f''_i}{x_{i}}=\frac{(Y_{i,\xi(i)+1})_{r-1} (Y_{i,\xi(i)-1})_{r+1}}{(Y_{i,\xi(i+1)})_r}=(Y_{i,\xi(i)-1})_r.
\]
If $j \neq i_{\diamond}$, then $\xi(i)<\xi(i+1)<\xi(i+2)$, and by Equation (\ref{mutation equationsi=i}),  		
\[
[L(x_i)] [L(x[\alpha_{i,i}])]=[L(f'_i)] [L(f''_i)] [L(x_{i+1})]+  [L(x_{i-1})][L(f'_{i+1})] [L(f''_{i+1})].
\]
Since $x_i \nmid x_{i-1} f'_{i+1} f''_{i+1} $, we have $x_i x[\alpha_{i,i}]= f'_i f''_i x_{i+1}$ by Proposition \ref{lemma:equations satisfy by q-character}. Hence 				
\begin{align*}
x[\alpha_{i,i}]=\frac{f'_i f''_i x_{i+1}}{x_i}&=\frac{(Y_{i,\xi(i)+1})_{r-1} (Y_{i,\xi(i)-1})_{r+1}  (Y_{i+1,\xi(i+2)})_r}{(Y_{i,\xi(i+1)})_r}\\
& =(Y_{i,\xi(i)-1})_r (Y_{i+1,\xi(i+2)})_{r}=(\omega(i,i+1))_r.
\end{align*}
Assume that our result holds for $i<j$. We prove it for $j+1$ by considering the following 8 cases.
				
{\bf Case 1.} Assume that $j+1<i_{\diamond}$.  Then $\xi(i)<\dots<\xi(j+1)<\xi(j+2)<\xi(j+3)$ by hypothesis. It follows from Equation (\ref{mutation equations}) that
\[
[L(x_{j+1})] [L(x[\alpha_{i,j+1}])]= [L(x[\alpha_{i,j}])] [L(x_{j+2})] +[L(x_{i-1})] [L(f'_{j+2})] [L(f''_{j+2})].
\]
Since $x_{j+1}\nmid x_{i-1}f'_{j+2} f''_{j+2}$, we have $x_{j+1} x[\alpha_{i,j+1}]= x[\alpha_{i,j}] x_{j+2}$ by Proposition \ref{lemma:equations satisfy by q-character}. Hence				
\begin{align*}
x[\alpha_{i,j+1}]=\frac{x[\alpha_{i,j}] x_{j+2}}{x_{j+1}}
=\frac{(\omega(i,j+1))_r x_{j+2}}{x_{j+1}}
&=\frac{(Y_{i,\xi(i)-1} Y_{j+1,\xi(j+1)+1})_r(Y_{j+2,\xi(j+3)})_r}{(Y_{j+1,\xi(j+2)})_r}\\
&=(Y_{i,\xi(i)-1})_r (Y_{j+2,\xi(j+3)})_r=(\omega(i,j+2))_r.
\end{align*}
				
{\bf Case 2.} Assume that $i_{\diamond}=j+1$. Then $\xi(i)<\dots<\xi(j+1)<\xi(j+2)$  by hypothesis. It follows from Equation (\ref{mutation equations}) that
\[
[L(x_{j+1})][L(x[\alpha_{i,j+1}])]=[L(x[\alpha_{i,j}])]+ [L(x_{i-1})][L(x_{j+2})].
\]
Since $x_{j+1} \nmid x_{i-1}x_{j+2}$, we have $x_{j+1}x[\alpha_{i,j+1}] =x[\alpha_{i,j}]$ by Proposition \ref{lemma:equations satisfy by q-character}. Hence				
\[
x[\alpha_{i,j+1}] =\frac{x[\alpha_{i,j}] }{x_{j+1}}=\frac{(\omega(i,j+1))_r}{x_{j+1}} =\frac{(Y_{i,\xi(i)-1})_r(Y_{j+1,\xi(j+1)+1})_r}{(Y_{j+1,\xi(j+2)})_r}=(Y_{i,\xi(i)-1})_r.
\]
				
{\bf Case 3.} Assume that $j=i_{\diamond}$ and $j+1=(j+1)_{\diamond}$. Then $d_{j-1}=0$ and $\xi(i)<\dots<\xi(j)<\xi(j+1)>\xi(j+2)$ by hypothesis.
It follows from Equation (\ref{mutation equations}) that
\[
[L(x_{j+1})] [L(x[\alpha_{i,j+1}])]=[L(x[\alpha_{i,j}])] [L(f'_{j+1})] [L(f''_{j+1})] +[L(x[\alpha_{i,j-1}])] [L(x_{j+2})].
\]
Since $x_{j+1} \nmid x[\alpha_{i,j-1}]x_{j+2}$, we have $x_{j+1} x[\alpha_{i,j+1}]= x[\alpha_{i,j}]  f'_{j+1} f''_{j+1}$ by Proposition \ref{lemma:equations satisfy by q-character}. Hence 				
\begin{align*}
x[\alpha_{i,j+1}]=\frac{x[\alpha_{i,j}] f'_{j+1} f''_{j+1} }{x_{j+1}} & = \frac{(Y_{i,\xi(i)-1})_r(Y_{j+1,\xi(j+1)+1})_{r-1}(Y_{j+1,\xi(j+1)-1})_{r+1}}{(Y_{j+1,\xi(j+2)})_r}\\
& =(Y_{i,\xi(i)-1})_r (Y_{j+1,\xi(j+1)+1})_{r}=(\omega(i,j+1))_r.
\end{align*}
				
{\bf Case 4.} Assume that $j=i_{\diamond}$ and $j+1\neq(j+1)_{\diamond}$. Then $d_{j-1}=0$ and $\xi(i)<\dots<\xi(j)<\xi(j+1)>\xi(j+2)>\xi(j+3)$ by hypothesis.
It follows from Equation (\ref{mutation equations}) that
\begin{gather*}
\begin{align*}
[L(x_{j+1})] [L(x[\alpha_{i,{j+1}}])]= [L(x[\alpha_{i,j}])] [L(f'_{j+1})] [L(f''_{j+1})] [L(x_{j+2})] +[L(x[\alpha_{i,j-1}])] [L(f'_{j+2})] [L(f''_{j+2})].
\end{align*}
\end{gather*}
 Since $x_{j+1} \nmid f'_{j+2}   f''_{j+2} x[\alpha_{i,j-1}]$, we have $x_{j+1} x[\alpha_{i,{j+1}}] =   x[\alpha_{i,j}] f'_{j+1} f''_{j+1} x_{j+2}$ by Proposition \ref{lemma:equations satisfy by q-character}. Hence 				
 \begin{align*}
x[\alpha_{i,{j+1}}]=\frac{x[\alpha_{i,j}] f'_{j+1} f''_{j+1} x_{j+2}}{x_{j+1}}
&=\frac{(Y_{i,\xi(i)-1})_r (Y_{j+1,\xi(j+1)+1})_{r-1}(Y_{j+1,\xi(j+1)-1})_{r+1} (Y_{j+2,\xi(j+3)})_r}{(Y_{j+1,\xi(j+2)})_r} \\
&=(Y_{i,\xi(i)-1})_r (Y_{j+1,\xi(j+1)+1} )_r (Y_{j+2,\xi(j+3)})_r=(\omega(i,j+2))_r.
\end{align*}
				
{\bf Case 5.} Assume that $i_{\diamond}<j+1$, $j\neq j_{\diamond}$, and $j+1=(j+1)_{\diamond}$. Then $j_{\bullet}=(j+1)_{\bullet}$ by hypothesis. It follows from Equation (\ref{mutation equations}) that
\begin{gather*}
\begin{align*}
[L(x_{j+1})] [L(x[\alpha_{i,j+1}])]=[L(x[\alpha_{i,j}])]+ [L(x_{j+2})] \big(  [L(f'_{i})]^{\delta_{i,(j+1)_{\bullet}}}[L(f''_{i})]^{\delta_{i,(j+1)_{\bullet}}}+(1-\delta_{i,(j+1)_{\bullet}})   \\ [L(f'_{(j+1)_\bullet})]^{d_{(j+1)_\bullet-1}} [L(f''_{(j+1)_\bullet})]^{d_{(j+1)_\bullet-1}} [L(x[\alpha_{i,(j+1)_\bullet-1}])]\big).
\end{align*}
\end{gather*}
Since $ x_{j+1} \nmid x_{j+2} \big( f'^{\delta_{i,(j+1)_{\bullet}}}_{i}f''^{\delta_{i,(j+1)_{\bullet}}}_{i} + (1-\delta_{i,(j+1)_{\bullet}}) f'^{d_{(j+1)_\bullet-1}}_{(j+1)_\bullet} f''^{d_{(j+1)_\bullet-1}}_{(j+1)_\bullet} x[\alpha_{i,(j+1)_\bullet-1}]
\big),$ we have $x_{j+1}x[\alpha_{i,j+1}]=x[\alpha_{i,j}]$ by Proposition \ref{lemma:equations satisfy by q-character}. Hence
\begin{align*}
x[\alpha_{i,j+1}]=\frac{x[\alpha_{i,j}]}{x_{j+1}}&=\frac{(\omega(i,j+1))_r}{(Y_{j+1,\xi(j+2)})_r}\\
&=(\omega(i,(j+1)_{\bullet}+1))_r \quad(\text{by Lemma}~\ref{equation of omega1}).
\end{align*}
				
{\bf Case 6.} Assume that $i_{\diamond}<j+1$, $j\neq j_{\diamond}$, and $j+1\neq(j+1)_{\diamond}$. Then $j_{\bullet}=(j+1)_{\bullet}$ by hypothesis. It follows from Equation (\ref{mutation equations}) that
\begin{gather*}
\begin{align*}
[L(x_{j+1})] [L(x[\alpha_{i,j+1}])]= [L(x[\alpha_{i,j}])] [L(x_{j+2})] + [L(f'_{j+2})] [L(f''_{j+2})] \big(
 [L(f'_i)]^{\delta_{i,(j+1)_{\bullet}}} [L(f''_{i})]^{\delta_{i,(j+1)_{\bullet}}} +\\
(1-\delta_{i,(j+1)_{\bullet}}) [L(f'_{(j+1)_\bullet})]^{d_{(j+1)_\bullet-1}} [L(f''_{(j+1)_\bullet})]^{d_{(j+1)_\bullet-1}} [L(x[\alpha_{i,(j+1)_\bullet-1}])]
\big).
\end{align*}
\end{gather*}
Since $ x_{j+1} \nmid f'_{j+2} f''_{j+2} ( f'^{\delta_{i,(j+1)_{\bullet}}}_i f''^{\delta_{i,(j+1)_{\bullet}}}_{i} + (1 -\delta_{i,(j+1)_{\bullet}}) f'^{d_{(j+1)_\bullet-1}}_{(j+1)_\bullet} f''^{d_{(j+1)_\bullet-1}}_{(j+1)_\bullet} x[\alpha_{i,(j+1)_\bullet-1}]
\big),$ we have $x_{j+1} x[\alpha_{i,j+1}]= x[\alpha_{i,j}] x_{j+2}$ by Proposition \ref{lemma:equations satisfy by q-character}.  Hence
\begin{align*}
x[\alpha_{i,j+1}]=\frac{x[\alpha_{i,j}] x_{j+2}}{x_{j+1}}&=\frac{(\omega(i,j+1))_r (Y_{j+2,\xi(j+3)})_r}{(Y_{j+1,\xi(j+2)})_r}\\
&=(\omega(i,(j+1)_{\bullet}+1))_r (Y_{j+2,\xi(j+3)})_r \quad(\text{by Lemma}~\ref{equation of omega1})\\
&=(\omega(i,j+2))_r \quad(\text{by Lemma}~\ref{equation of omega1}).
\end{align*}
				
{\bf Case 7.} Assume that $i_{\diamond}<j$, $j= j_{\diamond}$, and $j+1\neq(j+1)_{\diamond}$. Then $\xi(j)<\xi(j+1)>\xi(j+2)$ by hypothesis.
It follows from Equation (\ref{mutation equations}) that
\begin{align*}
[L(x_{j+1})] [L(x[\alpha_{i,j+1}])]= [L(x[\alpha_{i,j}])] [L(f'_{j+1})] [L(f''_{j+1})] [L(x_{j+2})] +[L(x[\alpha_{i,j-1}])] [L(f'_{j})]^{d_{j-1}} \\ 
[L(f''_{j})]^{d_{j-1}}  [L(f'_{j+2})][L(f''_{j+2})].
 \end{align*}
Since $x_{j+1} \nmid x[\alpha_{i,j-1}]  f'^{d_{j-1}}_{j}f''^{d_{j-1}}_{j} f'_{j+2}f''_{j+2}$, we have $x_{j+1} x[\alpha_{i,j+1}] = x[\alpha_{i,j}] f'_{j+1} f''_{j+1} x_{j+2}$ by Proposition \ref{lemma:equations satisfy by q-character}. Hence 	\begin{gather*}
\begin{align*}
x[\alpha_{i,j+1}] =\frac{x[\alpha_{i,j}]f'_{j+1} f''_{j+1} x_{j+2}}{x_{j+1}}
&=\frac{(\omega(i,j_{\bullet}+1))_r(Y_{j+1,\xi(j+1)+1})_{r-1}(Y_{j+1,\xi(j+1)-1})_{r+1}(Y_{j+2,\xi(j+3)})_r}{(Y_{j+1,\xi(j+2)})_r}\\
&=(\omega(i,j_{\bullet}+1))_r (Y_{j+1,\xi(j+1)+1})_{r}(Y_{j+2,\xi(j+3)})_r\\
&=(\omega(i,j+1))_r (Y_{j+2,\xi(j+3)})_r  \quad(\text{by Lemma}~\ref{equation of omega2})\\
&=(\omega(i,(j+2)_{\bullet}+1))_r (Y_{j+2,\xi(j+3)})_r\\
&=(\omega(i,j+2))_r \quad(\text{by Lemma}~\ref{equation of omega1}).
\end{align*}
\end{gather*}
				
{\bf Case 8.} Assume that $i_{\diamond}<j$, $j= j_{\diamond}$, and $j+1=(j+1)_{\diamond}$. Then $\xi(j)<\xi(j+1)>\xi(j+2)$ by hypothesis.
It follows from Equation (\ref{mutation equations}) that
\begin{gather*}
\begin{align*}
[L(x_{j+1})] [L(x[\alpha_{i,{j+1}}])]=[L(x[\alpha_{i,j}])] [L(f'_{j+1})][L(f''_{j+1})]  +[L(x[\alpha_{i,j-1}])]   [L(f'_{j})]^{d_{j-1}} [L(f''_{j})]^{d_{j-1}} [L(x_{j+2})].
\end{align*}
\end{gather*}
Since $x_{j+1} \nmid x[\alpha_{i,j-1}]  f'^{d_{j-1}}_{j} f''^{d_{j-1}}_{j} x_{j+2}$, we have $x_{j+1} x[\alpha_{i,{j+1}}]= x[\alpha_{i,j}] f'_{j+1} f''_{j+1}$ by Proposition \ref{lemma:equations satisfy by q-character}.  Hence
\begin{align*}
x[\alpha_{i,{j+1}}]=\frac{x[\alpha_{i,j}] f'_{j+1} f''_{j+1} }{x_{j+1}}
&=\frac{(\omega(i,j_{\bullet}+1))_r (Y_{j+1,\xi(j+1)+1})_{r-1}(Y_{j+1,\xi(j+1)-1})_{r+1}}{(Y_{j+1,\xi(j+2)})_r}\\
&= ( \omega(i,j_{\bullet}+1))_r (Y_{j+1,\xi(j+1)+1})_{r} \\
&=( \omega(i,j+1))_r \quad(\text{by}~\text{Lemma}~\ref{equation of omega2}).
\end{align*}
By the principle of induction, our assertion is true. The proof is completed.
\end{proof}
			
Now we are in the place for one of our main theorems.
			
\begin{theorem}\label{theorem:real and prime}
The equivalence classes of $\HL$-modules are cluster variables in $\seed$. In particular, $\HL$-modules are real prime simple modules.
\end{theorem}
\begin{proof}
Let $L((Y_{i_1,a_1} Y_{i_2,a_2}\dots Y_{i_k,a_k})_r)$ be an $\HL$-module. Then $(Y_{i_1,a_1} Y_{i_2,a_2}\dots Y_{i_k,a_k})_r$ is of the form (\ref{Monomial of HLr}).  Assume that $k=1$. Then there is a height function $\xi$ subject to $\xi(i_1)=\xi(i_1+2)$ for a large enough $n$, and
\[
\xi(i_1)=a_1\mp1~\text{ if }~\xi(i_1)=\xi(i_1+1)\pm1.
\]
Note that $\xi(i_1)=\xi(i_1+2)$ implies that $i_1=(i_1)_\diamond$.  By Lemma \ref{Lem:Hight weight monomial}, we have  $x[\alpha_{i_1,i_1}]=(Y_{i_1,a_1})_r$.

Assume that $k = 2$ and $i_{2}-i_{1}\geq 1$. Then there is a height function $\xi$ subject to
\[
\begin{cases}
\xi(i_1)>\dots>\xi(i_2-1)>\xi(i_2)<\xi(i_2+1)>\xi(i_2+2)\quad \quad &\text{if}~a_1>a_2,\\
\xi(i_1)<\dots<\xi(i_2-1)<\xi(i_2)>\xi(i_2+1)<\xi(i_2+2) \quad\quad &\text{if}~a_1<a_2.
\end{cases}
\]
Note that $\xi(i_1)>\dots>\xi(i_2-1)>\xi(i_2)<\xi(i_2+1)>\xi(i_2+2)$ (respectively, $\xi(i_1)<\dots<\xi(i_2-1)<\xi(i_2)>\xi(i_2+1)<\xi(i_2+2)$) implies that $(i_1)_\diamond=i_{2}-1$, $i_2=(i_2)_\diamond$.
By Lemma \ref{Lem:Hight weight monomial}, we have  $x[\alpha_{i_1,i_2}]=(\omega(i_1,\overline{i_2}))_r=(\omega(i_1,i_2))_r=(Y_{i_1,a_1}Y_{i_2,a_2})_r$.

Assume that $k > 2$ and $i_{k}-i_{k-1}\geq 1$. Then there is a height function $\xi$ subject to
\[
\xi(i_s-1)=\xi(i_s+1)~\text{for}~2\leq s\leq k-1,
\]
and
\begin{gather*}
\begin{align*}
\begin{cases}
\xi(i_1)>\xi(i_2),\xi(i_{k-1})>\dots>\xi(i_k-1)>\xi(i_k)<\xi(i_k+1)>\xi(i_k+2)\quad &\text{if}~a_1>a_2,a_{k-1}>a_k,\\
\xi(i_1)<\xi(i_2),\xi(i_{k-1})>\dots>\xi(i_k-1)>\xi(i_k)<\xi(i_k+1)>\xi(i_k+2)\quad &\text{if}~a_1<a_2,a_{k-1}>a_k,\\
\xi(i_1)>\xi(i_2),\xi(i_{k-1})<\dots<\xi(i_k-1)<\xi(i_k)>\xi(i_k+1)<\xi(i_k+2)\quad &\text{if}~a_1>a_2,a_{k-1}<a_k,\\
\xi(i_1)<\xi(i_2),\xi(i_{k-1})<\dots<\xi(i_k-1)<\xi(i_k)>\xi(i_k+1)<\xi(i_k+2)\quad &\text{if}~a_1<a_2,a_{k-1}<a_k.\\
\end{cases}
\end{align*}
\end{gather*}
Note that $\xi(i_{k-1})>\dots>\xi(i_k-1)>\xi(i_k)<\xi(i_k+1)>\xi(i_k+2)$ (respectively, $\xi(i_{k-1})<\dots<\xi(i_k-1)<\xi(i_k)>\xi(i_k+1)<\xi(i_k+2)$) implies that $(i_{k-1})_\diamond=i_k-1$, $i_k=(i_k)_\diamond$.
By Lemma \ref{Lem:Hight weight monomial}, we have $x[\alpha_{i_1,i_k}]=(\omega(i_1,\overline{i_k}))_r=(\omega(i_1,i_k))_r=(Y_{i_1,a_1}\cdots Y_{i_k,a_k})_r$.

Therefore, the equivalence classes of $\HL$-modules are cluster variables in $\seed$, and so $\HL$-modules are real prime simple modules by Lemma \ref{theorem of real prime}.
\end{proof}

Given a height function $\xi$, the associated quiver $Q_\xi$ is defined in Section \ref{section:definition of cluster algebra}. If we delete the vertices $\{1',\dots,n'\}$ and the arrows incident to them, then the obtained quiver is opposite to the one defined in \cite{BC19}. Moreover, in $Q_\xi$, for any $i\in I$, the arrows incident to $\mathbf{f}'_i$ and the arrows incident to $\mathbf{f}''_i$ are symmetric with respect to the principal quiver. Replace $f_j$ by $\mathbf{f}'_j \mathbf{f}''_j$ in \cite[Proposition 2.5]{BC19}, we get a $q$-character formula of $\HL$-modules.

In \cite{DLL2021}, the authors gave a recursive $q$-character formula of HL-modules by induction on the length of the highest $l$-weight monomials of HL-modules. Likewise, replace $f_j$ by $\mathbf{f}'_j \mathbf{f}''_j$, we obtain the following proposition.

\begin{proposition}\label{prop:equation system}
Let $\xi$ be a height function such that for $1\leq i < j+1 \leq n$, $j_{\diamond}=j$, $(j+1)_{\diamond}=j+1$. Then
\begin{gather}\label{equation system}
\begin{aligned}[]
[L(x[\alpha_{i,j+1}])] = [L(x[\alpha_{i,j}])] [L(x[\alpha_{j+1,j+1}])] - [L(H_{i,\max\{i-1,j_{\bullet}-1\}})]^{1-\delta_{i,j_{\bullet}}} [L(f'_{\max\{i,j_{\bullet}\}})]^{\min\{1,(1-\delta_{j_{\bullet},i_{\bullet}})d_{j_{\bullet}-1}+\delta_{j_{\bullet},i}\}} \\
[L(f''_{\max\{i,j_{\bullet}\}})]^{\min\{1,(1-\delta_{j_{\bullet},i_{\bullet}})d_{j_{\bullet}-1}+\delta_{j_{\bullet},i}\}} [L(x[-\alpha_{j+2}])],
\end{aligned}
\end{gather}
where $H_{i,\max\{i-1,j_{\bullet}-1\}}=x[\alpha_{i,\max\{i-1,j_{\bullet}-1\}}]$ if $j_{\bullet}>i$, otherwise $H_{i,\max\{i-1,j_{\bullet}-1\}}=x[-\alpha_{i-1}]$. More
precisely,
\begin{equation}\label{equation: three equations}
\begin{cases}
[L(x[\alpha_{i,j+1}])]=[L(x[\alpha_{i,j}])] [L(x[\alpha_{j+1,j+1}])] - [L(x_{i-1})] [L(x_{j+2})] \quad  & \text{if}~j_{\bullet}=i_{\bullet},\\
[L(x[\alpha_{i,j+1}])] = [L(x[\alpha_{i,j}])] [L(x[\alpha_{j+1,j+1}])]- [L(f'_i)] [L(f''_i)] [L(x_{j+2})] \quad &\text{if}~j_{\bullet}=i,\\
[L(x[\alpha_{i,j+1}])] = [L(x[\alpha_{i,j}])] [L(x[\alpha_{j+1,j+1}])] - [L(x[\alpha_{i,j_{\bullet-1}}])] [L(f'_{j_{\bullet}})] [L(f''_{j_{\bullet}})] [L(x_{j+2})]\quad &\text{if}~j_{\bullet}> i~\text{and}~j_{\bullet}-1=(j_{\bullet}-1)_{\diamond},\\
[L(x[\alpha_{i,j+1}])] = [L(x[\alpha_{i,j}])] [L(x[\alpha_{j+1,j+1}])] - [L(x[\alpha_{i,j_{\bullet-1}}])] [L(x_{j+2})]\quad &\text{if}~j_{\bullet}> i~\text{and}~j_{\bullet}-1\neq(j_{\bullet}-1)_{\diamond}.
\end{cases}
\end{equation}
\end{proposition}

Given an $\HL$-module, we can always construct a height function $\xi$ subject to $\xi(j) = \xi(j+2)$ and $\xi(j+1)=\xi(j+3)$, as shown in the proof of Theorem \ref{theorem:real and prime}, such that $L(x[\alpha_{i,j+1}])$ is the $\HL$-module. Note that conditions $\xi(j) = \xi(j+2)$ and $\xi(j+1)=\xi(j+3)$ imply $j_{\diamond}=j$ and $(j+1)_{\diamond}=j+1$. Then the formula (\ref{equation system}) provides a recursive $q$-character formula for $\HL$-modules.
				
By Lemma \ref{Lem:Hight weight monomial}, the formula (\ref{equation: three equations}) is equivalent to the following theorem.

\begin{theorem}
Let $L((Y_{i_1,a_i}\cdots Y_{i_k,a_k})_r)$ be an $\HL$-module. 			
If $k=2$ and $a_1=a_2\pm (i_2-i_1+2)$, then
\[
[L((Y_{i_1,a_1} Y_{i_2,a_2})_r)] = [L((Y_{i_1,a_1})_r)][L((Y_{i_2,a_2})_r)] -[L((Y_{i_1-1,a_1\mp1})_r)] [L((Y_{i_2+1,a_2\pm1})_r)].
\]
If $k=3$, $i_2-i_1=1$ and $a_2=a_3\pm (i_3-i_2+2)$, then
\begin{gather*}
\begin{align*}
[L((Y_{i_1,a_1}Y_{i_2,a_2}Y_{i_3,a_3})_r)] = [L((Y_{i_1,a_1}Y_{i_2,a_2})_r)] [L((Y_{i_3,a_3})_r)] - [L((Y_{i_1,a_1})_{r\pm1})] [L((Y_{i_1,a_1\pm2})_{r\mp1})] [L((Y_{i_3+1,a_3\pm1})_r)].
\end{align*}
\end{gather*}
If $k>3$, $i_{k-1}-i_{k-2}=1$ and $a_{k-1}=a_{k}\pm(i_{k}-i_{k-1}+2)$, then
\begin{gather*}
\begin{align*}
[L((Y_{i_1,a_1}\cdots Y_{i_{k},a_{k}})_r)] =    [L((Y_{i_1,a_1}\cdots Y_{i_{k-1},a_{k-1}})_r)] [L((Y_{i_{k},a_{k}})_r)] -[L((Y_{i_1,a_1}\cdots Y_{i_{k-3},a_{k-3}})_r)] [L((Y_{i_{k-2},a_{k-2}})_{r\pm1})]\\
 [L((Y_{i_{k-2},a_{k-2}\pm2})_{r\mp1})] [L((Y_{i_{k}+1,a_{k}\pm1})_r)].
\end{align*}
\end{gather*}
If $k\geq3$, $i_{k-1}-i_{k-2}>1$ and $a_{k-1}=a_{k}\pm(i_{k}-i_{k-1}+2)$, then
\begin{gather*}
\begin{align*}
[L((Y_{i_1,a_1}\cdots Y_{i_{k},a_{k}})_r)] =  [L((Y_{i_1,a_1}\cdots Y_{i_{k-1},a_{k-1}})_r)] [L((Y_{i_{k},a_{k}})_r)] -[L((Y_{i_1,a_1}\cdots Y_{i_{k-2},a_{k-2}}Y_{i_{k-1}-1,a_{k-1}\mp1})_r)]  \\
[L( (Y_{i_{k}+1,a_{k}\pm1})_r)].
\end{align*}
\end{gather*}
\end{theorem}		

\begin{proof}
It follows from Lemma \ref{Lem:Hight weight monomial} that the highest $l$-weight monomial of each $\HL$-module $L(x[\alpha_{i, j}])$ in the formula (\ref{equation: three equations}) is the desired expression. 
\end{proof}

\section{Generalized HL-modules}\label{section of generalizedHL}
In this section, we further generalize the concept of $\HL$-modules and introduce the notion of generalized HL-modules, we will prove that all generalized HL-modules are real and prime.

\subsection{Definition of generalized HL-modules}
\begin{definition}\label{definition of generalizedHL}
A generalized HL-module is a simple $U_q(\widehat{\mathfrak{g}})$-module with the highest $l$-weight monomial
\[
(Y_{i_1,a_1})_{r,r_1}(Y_{i_2,a_2})_{r,r_2}\dots (Y_{i_k,a_k})_{r,r_k},
\]
where $ (Y_{i_j,a_j})_{r,r_j} =\begin{cases}
(Y_{i_j,a_j})_{r+r_j} & \text{if}~r_j\leq0,  \\
(Y_{i_j, a_j-2r_j})_{r+r_j} & \text{if}~r_j\geq 0,
\end{cases}$ 
$k,r\in\mathbb{Z}_{\geq 1}$, $i_j\in I$, $a_{j}\in \mathbb{Z}$, $r_j\in \mathbb{Z}_{\geq -r}$ for $j=1,2,\ldots,k$, and
\begin{enumerate}
\item[(1)] $i_1<i_2<\cdots<i_k$,
\item[(2)] $(a_{j}-a_{j-1})(a_{j+1}-a_{j})<0$ for $2\leq j\leq k-1$,
\item[(3)] $\lvert a_{j}-a_{j-1}\rvert=i_{j}-i_{j-1}+2$ for $2\leq j\leq k$,
\item[(4)] $r_1=0$ if $a_{1}>a_{2}$, $r_j=0$ if $2\leq j\leq k-1$, $a_{j-1}<a_j$ and $a_j>a_{j+1}$, $r_k=0$ if $a_{k-1}<a_{k}$.
\end{enumerate}	 
\end{definition}

It follows from the definition of generalized HL-modules that $r_1 = r_3 = r_5 = \cdots = 0$ if $a_1 > a_2$, otherwise $r_2 = r_4 = r_6 = \cdots = 0$.

We give some examples as follows.
\begin{example} 
Some examples of generalized HL-modules are shown in the below.
 \begin{gather}
\begin{align*}
& L( (Y_{1,-7})_{3,2}  (Y_{2,-4})_{3,0} (Y_{3,-7})_{3,1}) = L(Y_{1,-11} Y_{1,-9} Y_{1,-7}  Y_{1,-5}  Y_{1,-3}  Y_{2,-4} Y_{2,-2}  Y_{2,0}  Y_{3,-9} Y_{3,-7}  Y_{3,-5}  Y_{3,-3}), \\
& L((Y_{1,-3})_{2,0} (Y_{2,-6})_{2,-1}  (Y_{3,-3})_{2,0}  (Y_{4,-6})_{2,0}) = L(Y_{1,-3} Y_{1,-1} Y_{2,-6}  Y_{3,-3}  Y_{3,-1} Y_{4,-6}  Y_{4,-4}), \\
& L( (Y_{2,-6})_{2,-1}  (Y_{4,-2})_{2,0}  (Y_{5,-5})_{2,1} (Y_{6,-2})_{2,0})=L(Y_{2,-6} Y_{4,-2} Y_{4,0} Y_{5,-7} Y_{5,-5} Y_{5,-3} Y_{6,-2} Y_{6,0}), \\
& L( (Y_{2,-6})_{2,1} (Y_{4,-2})_{2,0} (Y_{6,-6})_{2,-1} (Y_{5,-3})_{2,0}) = L(Y_{2,-8} Y_{2,-6}  Y_{2,-4}  Y_{4,-2}  Y_{4,0}  Y_{6,-6} Y_{5,-3}  Y_{5,-1}).
\end{align*}
\end{gather}
\end{example}

\subsection{Equivalence classes of generalized HL-modules as cluster variables}\label{equivalence classes of generalized HL-modules as cluster variables}

Starting from the initial seed of $K_0(\mathscr{C}_{\ell})$ defined in Section \ref{K0Ccluster algebra}, each mutation at a cluster variable yields a new cluster variable, and the new cluster variable corresponds to a real prime simple module of $\mathscr{C}_{\ell}$ by Lemma \ref{theorem of real prime}. We can compute the highest $l$-weight monomial of the new real prime simple module by Proposition \ref{lemma:equations satisfy by q-character}. In fact, we find a mutation sequence such that the equivalence classes of generalized HL-modules are cluster variables in $K_0(\mathscr{C}_{\ell})$.

In the following, we give a mutation sequence of an arbitrary generalized HL-module 
\[
L((Y_{i_1,a_1})_{r,r_1}(Y_{i_2,a_2})_{r,r_2}\dots (Y_{i_k,a_k})_{r,r_k})
\]
in $K_0(\mathscr{C}_{\ell})$. It follows from Definition \ref{definition of GHL} that $L((Y_{i_1,a_1})_{r,0}(Y_{i_2,a_2})_{r,0}\dots (Y_{i_k,a_k})_{r,0})$ is an $\HL$-module. As shown in the proof of Theorem \ref{theorem:real and prime}, we can always construct a height function $\xi$ subject to $\xi(i_{1}-1)\neq\xi(i_{1}+1)$, $\xi(i_{k}-1)=\xi(i_{k}+1)$, and $\xi(i_{k})=\xi(i_{k}+2)$, such that $L(x[\alpha_{i_{1}, i_{k}}])$ is the $\HL$-module $L((Y_{i_1,a_1})_{r,0}(Y_{i_2,a_2})_{r,0}\dots (Y_{i_k,a_k})_{r,0})$.

Let $\mathscr{S}_m(\seed)$ be the seed defined in (\ref{the new obtained seed before HLr}). The next step is to perform the following mutations in $\mathscr{S}_m(\seed)$. If there exists a vertex $(u,v)$, with $u\in I$, $v\in [1,r-2]\cup [r+2,\ell]$, such that the arrows incident to it are shown in Figure \ref{arrows incident to (u,v)}, then we mutate it. If, after the mutation, there still exists a vertex $(u,v)$ in the current quiver such that the arrows incident to it are shown in Figure \ref{arrows incident to (u,v)}, then we continue mutating it. Repeating this process until no vertex in the first $s$ rows, where $s$ is sufficiently large, such that the arrows incident to it are shown in Figure \ref{arrows incident to (u,v)}. The new seed is denoted by $\mathscr{S}_m(\seed)'$.

\begin{figure}[H]
\centerline{
{\tiny
\xymatrix{
&(u,v-1)\ar[d]&\\
(u-1,v)  & (u,v) \ar[l] \ar[r]& (u+1,v) \\
&(u,v+1)\ar[u]&
}}
}
\caption{Arrows incident to $(u,v)$. Here we ignore vertices and arrows incident to them if the vertices in this quiver are not in $I\times [1,\ell+1]$.} \label{arrows incident to (u,v)}
\end{figure}
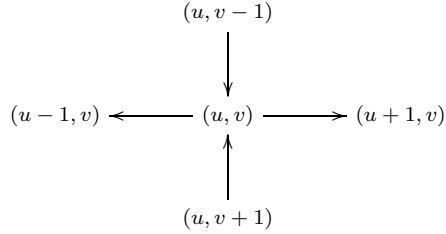

For $u\in I$, $t_u\in\mathbb{Z}$, let $S^{t_u}_{\xi}$ be an ordered sequence of vertices,  defined by
\[
S^{t_u}_{\xi}:=\begin{cases}
((u,{r+t_{u}}), (u, {r+t_{u}-1}),\dots,(u, {r})) & \text{if}~\xi(u)=\xi(u+1)-1, t_u\geq 0,\\
((u,{r+t_{u}}), (u, {r+t_{u}+1}),\dots,(u, {r})) & \text{if}~\xi(u)=\xi(u+1)-1, t_u\leq 0,\\
((u, {r}),(u, {r+1}),\dots,(u, {r+t_{u}})) & \text{if}~\xi(u)=\xi(u+1)+1, t_u\geq 0,\\
((u, {r}),(u, {r-1}),\dots, (u, {r+t_{u}})) & \text{if}~\xi(u)=\xi(u+1)+1, t_u\leq 0.
\end{cases}
\] 			
Define $(S^{t_{i_{1}}}_{\xi}, S^{t_{i_{1}+1}}_{\xi}, \dots, S^{t_{i_{k}}}_{\xi})$ to be an ordered sequence of vertices subject to the following conditions:
\begin{itemize}
\item[(1)] $t_{i_1}=r_1$, $t_{i_m}=r_{m-1}+r_{m}$ for $2\leq m\leq k$, 
\item[(2)] $t_{p}=r_{m-1}+r_{m}$ for $i_{m-1}<p<i_{m}$, where $2\leq m\leq k$.
\end{itemize}

If $a_{k-1} < a_k$, or $a_{k-1} > a_k$ and $r_k = 0$, or $n=i_k$ ($n$ is the rank of $\mathfrak{g}$), let
\[
\mathscr{S}':= (S^{t_{i_{1}}}_{\xi}, S^{t_{i_{1}+1}}_{\xi}, \dots, S^{t_{i_{k}}}_{\xi}); 
\]
if  $i_k < n$, $a_{k-1} > a_k$, and $r_k > 0$, let
\[
\mathscr{S}':= (S^{t_{i_{1}}}_{\xi}, S^{t_{i_{1}+1}}_{\xi}, \dots, S^{t_{i_{k}}}_{\xi},(i_{k}+1,r), (i_{k}+1,r+1), \dots, (i_{k}+1,r+r_k)); 
\]
and if $i_k < n$, $a_{k-1} > a_k$, and $r_k < 0$, let 
\[
\mathscr{S}':= (S^{t_{i_{1}}}_{\xi}, S^{t_{i_{1}+1}}_{\xi}, \dots, S^{t_{i_{k}}}_{\xi},(i_{k}+1,r), (i_{k}+1,r-1), \dots, (i_{k}+1,r+r_k)). 
\]

We mutate sequence $\mathscr{S}'$, meaning that we first mutate the vertices of $S^{t_{i_1}}_{\xi}$ in order, and then mutate the vertices of $S^{t_{i_{1}+1}}_{\xi}$ in order, and so on.

\begin{theorem} \label{generalized HL-modules are cluster variables}
The equivalence classes of generalized HL-modules are cluster variables in $K_0(\mathscr{C}_{\ell})$. In particular, generalized HL-modules are real and prime. 
\end{theorem}
 
The proof of the theorem above will be given in Section \ref{Proof of Theorem on generalized HL-modules}. 

We have the following corollary.
\begin{corollary}
Both Conjecture \ref{conjecture of bijection} and Conjecture \ref{conjecture of geometric q-character formula} are true for generalized HL-modules.
\end{corollary}
\begin{proof}
It follows from Theorem \ref{generalized HL-modules are cluster variables} that Conjecture \ref{conjecture of bijection} is true for generalized HL-modules. Conjecture \ref{conjecture of geometric q-character formula} is true for real simple modules corresponding to cluster monomials, see \cite[Section 5.5.4]{HL21}. Therefore, by Theorem \ref{generalized HL-modules are cluster variables} again, Conjecture \ref{conjecture of geometric q-character formula} is true for generalized HL-modules.
\end{proof}

\subsection{Proof of Theorem \ref{generalized HL-modules are cluster variables}} \label{Proof of Theorem on generalized HL-modules}
In what follows, we prove that the required generalized HL-module is obtained by performing the mutation sequence $\mathscr{S}'$ in the seed $\mathscr{S}_m(\seed)'$.

We assume that $Q'_\xi$ is the quiver in the seed $\mathscr{S}_m(\seed)'$. From the definition of sequence $\mathscr{S}'$ in Section \ref{equivalence classes of generalized HL-modules as cluster variables}, it follows that each vertex in $\mathscr{S}'$ appears only once. 
Let $(j,r+s)$ be a vertex in $\mathscr{S}'$, where $i_1 \leq j\leq i_k+1$, $t_j \geq 0$ (respectively, $t_j \leq 0$), $s\in [0,t_j]$ (respectively, $s\in[t_j,0]$). Apply the mutation sequence $\mathscr{S}'$ to $Q'_\xi$, we denote by $Q_\xi^{s}[i_1,j]$ the quiver, which is obtained by mutating the first vertex until the vertex before $(j, r+s)$ in order.

Let		
\[
p^x(t)=\begin{cases}
1  \quad &\text{if}~ t \geq x,\\
0  \quad &\text{if}~ t < x, 
\end{cases}
\qquad q^x(t)=\begin{cases}
1  \quad &\text{if}~ t \leq x,\\
0  \quad &\text{if}~ t > x.
\end{cases}
\]

The arrows incident to the vertex $(j, r-s)$ in $Q_\xi^{-s}[i_1,j]$ can be obtained from the arrows incident to $(j, r+s)$ in $Q^{s}_{\xi}[i_1,j]$ by replacing each vertex $(j, r+a)$ with $(j, r-a)$, the function $p^{x}(t)$ (respectively, $q^{x}(t)$) with $q^{x}(t)$ (respectively, $p^{x}(t)$), and deleting the vertices that are not in $Q'_\xi$, as well as the arrows incident to them. For our purpose, it is enough to give all arrows incident to $(j, r+s)$ in $Q^{s}_{\xi}[i_1,j]$.

With the help of Keller's quiver mutation applet \cite{Kel}, all arrows incident to $(j, r+s)$ in $Q^{s}_{\xi}[i_1,j]$ are given by the following 10 cases. By the same argument as in Lemma \ref{arrows accident with j}, we prove Case 10 in Appendix \ref{appendix}, and the other cases can be similarly proved by induction.

{\bf Case 1.} In the case where $\xi(i_1)=\xi(i_1+1)-1$, $t_{i_1}=0$, arrows incident to $(i_1,r)$ in $Q_\xi^{0}[i_1,i_1]$ are shown in Figure \ref{arrows incident with 0}, and reverse the orientations of arrows if $\xi(i_1)=\xi(i_1+1)+1$, $t_{i_1}=0$.
\begin{figure}
\centering
\resizebox{.55\textwidth}{.8\height}{
\xymatrix{
 &(i_1, r-1) \ar[d] &(i_1+1, r-1)    \\
 (i_1-1, r)  & (i_1, r) \ar[l]   \ar@/^/[r]^{d_{i_1}} \ar[ru]^{1-d_{i_1}}
 \ar[rd]_{1-d_{i_1}}
&(i_1+1, r) \ar@/^/[l]^{1-d_{i_1}} \\
 &(i_1, r+1)   \ar[u]& (i_1+1, r+1)
}
}
\caption{Arrows incident to $(i_1,r)$ in $Q_\xi^{0}[i_1,i_1]$ (see  Case 1).} \label{arrows incident with 0}
\end{figure}
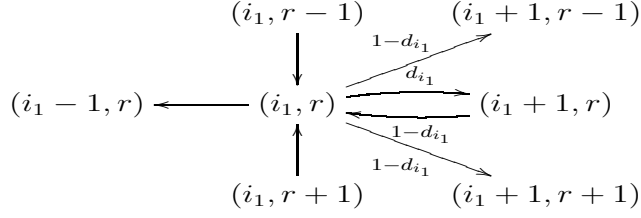

{\bf Case 2.} In the case where $\xi(i_1)=\xi(i_1+1)-1$, $t_{i_1} \geq 1$, arrows incident to $(i_1,r+t_{i_1})$ in $Q_\xi^{t_{i_{1}}}[i_1,i_1]$ are shown in Figure \ref{arrows incident with 1}.
\begin{figure}
\centering
\resizebox{.55\textwidth}{.8\height}{
\xymatrix{
(i_1-1, r+t_{i_1}-1) \ar[rd] &(i_1, r+t_{i_1}-1)  &   \\
(i_1-1, r+t_{i_1})  & (i_1, r+t_{i_1}) \ar[l]    \ar[rd]\ar[u] &(i_1+1, r+t_{i_1}) \ar[l] \\
 &(i_1, r+t_{i_1}+1)   \ar[u]& (i_1+1, r+t_{i_1}+1)
}
}
\caption{Arrows incident to $(i_1,r+t_{i_1})$ in $Q_\xi^{t_{i_{1}}}[i_1,i_1]$ (see  Case 2).} \label{arrows incident with 1}
\end{figure}
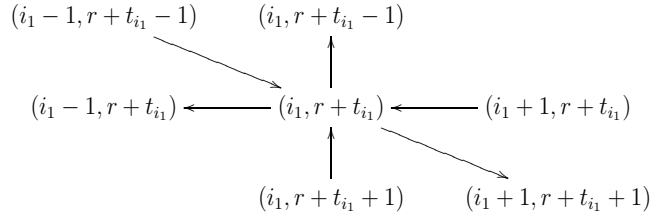

{\bf Case 3.}  In the case where $t_{i_1}>s \geq 1$, $\xi(i_1)=\xi(i_1+1)-1$, arrows incident to $(i_1,r)$ in $Q_\xi^{0}[i_1,i_1]$ are shown in the left picture of Figure \ref{arrows incident with 2}, and arrows incident to $(i_1,r+s)$ in $Q_{\xi}^{s}[i_1,i_1]$ are shown in the right picture of Figure \ref{arrows incident with 2}.
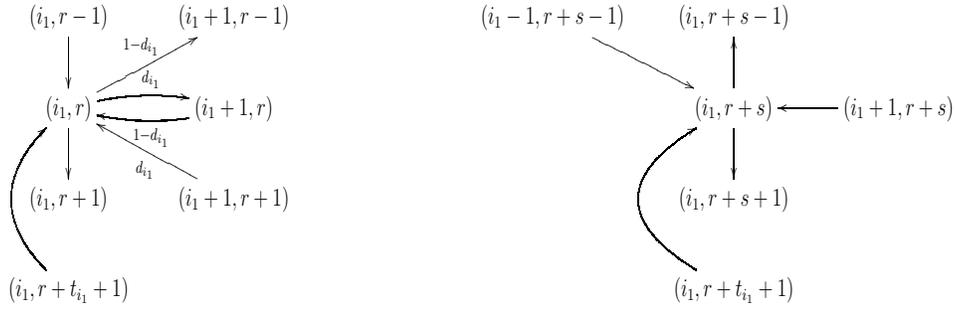
\begin{figure}[H]
\centering
\resizebox{.8\textwidth}{.8\height}{
\begin{minipage}[c]{0.5\textwidth}
\xymatrix{
(i_1, r-1) \ar[d]  &(i_1+1, r-1)  \\
 (i_1, r) \ar[ur]^{1-d_{i_1}}  \ar@/^/[r]^{d_{i_1}} \ar[d]&(i_1+1, r) \ar@/^/[l]^{1-d_{i_1}} \\
(i_1, r+1)   & (i_1+1, r+1) \ar[lu]^{d_{i_1}}\\
(i_1, r+t_{i_1}+1) \ar@/^3pc/[uu]  &
 }
\end{minipage}
\qquad \qquad
\begin{minipage}[c]{0.5\textwidth}
\xymatrix{
& &(i_1-1, r+s-1)  \ar[rd]  &(i_1, r+s-1)   & \\
& &  &(i_1, r+s)   \ar[u] \ar[d]& (i_1+1, r+s)  \ar[l]\\
&   & & (i_1, r+s+1)   & \\
&   & & (i_1, r+t_{i_1}+1)  \ar@/^5pc/[uu]  & } 
\end{minipage}}
\caption{Arrows incident to $(i_1,r)$ in $Q_\xi^{0}[i_1,i_1]$ (left), and arrows incident to $(i_1,r+s)$ in $Q_{\xi}^{s}[i_1,i_1]$ (right) (see Case 3).} \label{arrows incident with 2}
\end{figure}

{\bf Case 4.} In the case where $j> i_1$, $t_j=0$ and $\xi(j)=\xi(j+1)-1$, arrows incident to $(j, r)$ in  $Q_\xi^{0}[i_1, j]$ are shown in Figure \ref{arrows incident with 3}, where $a_1=(1-\delta_{{i_1}_{\bullet},j_{\bullet}}-\delta_{{i_1}, j_{\bullet}})d_{j_{\bullet}-1}p^{0}(r_{j_\bullet})+\delta_{{i_1}, j_{\bullet}}$, $a_2=(1-\delta_{{i_1}_{\bullet},j_{\bullet}}-\delta_{{i_1}, j_{\bullet}})d_{j_{\bullet}-1} q^{0}(r_{j_\bullet})+\delta_{{i_1},j_{\bullet}}$.

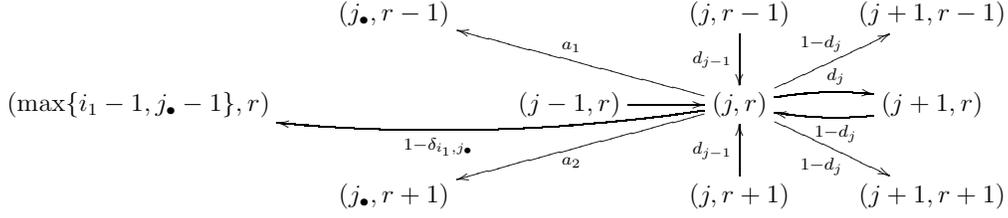
\begin{figure} [H]
\centering
\resizebox{.85\textwidth}{.8\height}{
\xymatrix{
& (j_{\bullet}, r-1) &  &(j, r-1) \ar[d]_{d_{j-1}}  &(j+1, r-1)  \\
(\max\{i_1-1,j_{\bullet}-1\}, r) & & (j-1, r) \ar[r]& (j, r) \ar[ull]_{a_{1}} \ar[dll]^{a_{2}} \ar@/^1pc/[lll]^{1-\delta_{i_1, j_{\bullet}}} \ar[ur]^{1-d_j} \ar[dr]_{1-d_j} \ar@/^/[r]^{d_{j}} & (j+1, r) \ar@/^/[l]^{1-d_{j}} \\
&(j_{\bullet}, r+1)  &  &(j, r+1) \ar[u]^{d_{j-1}}  & (j+1, r+1)
}}
\caption{Arrows incident to $(j, r)$ in  $Q_\xi^{0}[i_1, j]$ (see Case 4).}\label{arrows incident with 3}
\end{figure}

{\bf Case 5.} In the case where $j > i_1$, $t_j \geq 1$ and $\xi(j)=\xi(j+1)-1$, arrows incident to $(j, r)$ in  $Q_\xi^{0}[i_1,j]$ are shown in Figure \ref{arrows incident with 4}, where $a=(1-\delta_{{i_1}_{\bullet},j_{\bullet}}-\delta_{{i_1},j_{\bullet}})d_{j_{\bullet}-1}p^{0}(r_{j_\bullet})+\delta_{{i_1},j_{\bullet}}$, $b=(1-\delta_{{i_1}_\bullet, j_{\bullet}}-\delta_{{i_1},j_{\bullet}})q^{0}(r_{j_\bullet})$.
\begin{figure}[H]
\centering
\resizebox{.85\textwidth}{.8\height}{
\xymatrix{
&(\max\{i_1, j_{\bullet}\}, r-1) &  &(j, r-1) \ar[d]_{d_{j-1}}  &(j+1, r-1)  \\
(j_{\bullet}-1, r) & & (j-1, r)  \ar[r]& (j, r)  \ar[ull]_{a}  \ar@/^1pc/[lll]^{b} \ar[ur]^{1-d_j} \ar@/^/[r]^{d_{j}} \ar[d] & (j+1, r)  \ar@/^/[l]^{1-d_{j}} \\
&   &  & (j, r+1) & (j+1, r+1) \ar[ul]^{d_j}  \\
&  &  &(j, t_j+r+1) \ar@/^3pc/[uu]^{d_{j-1}} & 
}}
\caption{Arrows incident to $(j, r)$ in  $Q_\xi^{0}[i_1,j]$ (see  Case 5).}\label{arrows incident with 4}
\end{figure}
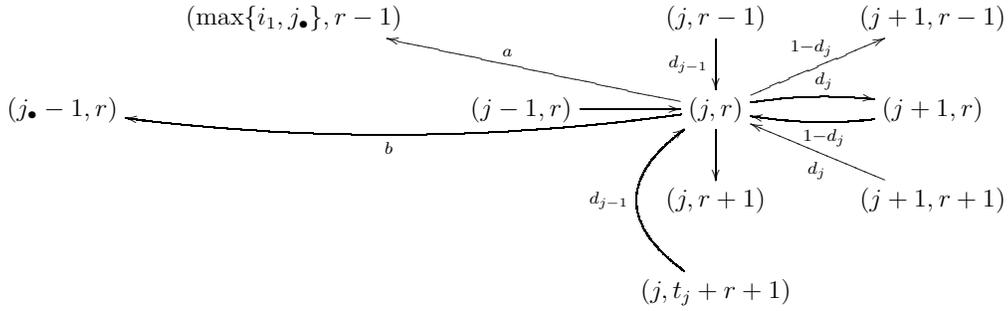

{\bf Case 6.} In the case where $j>i_1$, $s \geq 1$ and $\xi(j)=\xi(j+1)-1$, $t_j>s$, arrows incident to $(j, r+s)$ in  $Q_\xi^{s}[i_1,j]$ are shown in Figure \ref{arrows incident with 6}, where $a=d_{j-1} q^{t_{j-1}}(s) (1-\delta_{1,s})$, $b=d_{j-1} q^{t_{j-1}}(s)\delta_{1,s}$, $c=d_{j-1}p^{t_{j-1}}(s-1)+(1-d_{j-1})$.
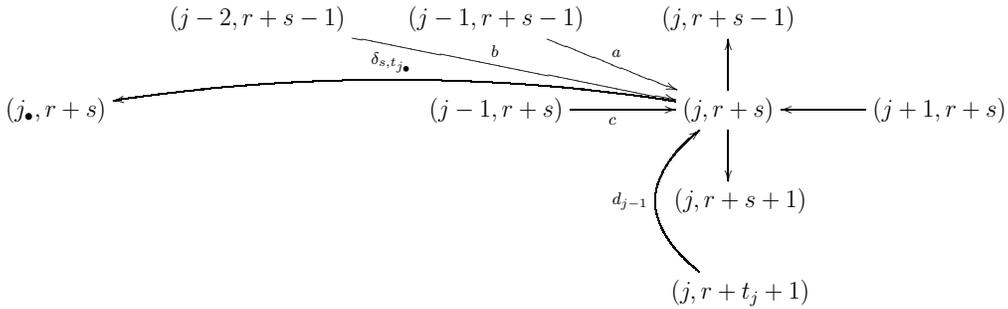
\begin{figure} [H]
\centering
\resizebox{.85\textwidth}{.8\height}{
\xymatrix{
&(j-2, r+s-1) \ar[rrd]^{b}& (j-1, r+s-1) \ar[rd]^{a} &(j, r+s-1)   &\\
(j_\bullet, r+s) & & (j-1, r+s) \ar[r]_{c}& (j, r+s) \ar[u]  \ar[d]  \ar@/^-1.2pc/[lll]_{\delta_{s, t_{j_\bullet}}}& (j+1, r+s) \ar[l]\\
&&  &\quad (j, r+s+1)   & \\
&&  &\quad (j, r+t_j+1) \ar@/^3pc/[uu]^{d_{j-1}}&  
}}
\caption{Arrows incident to $(j, r+s)$ in  $Q_\xi^{s}[i_1,j]$ (see Case 6).}\label{arrows incident with 6}					
\end{figure}

{\bf Case 7.} In the case where $j> i_1$, $s \geq 1$ and $\xi(j)=\xi(j+1)-1$, $t_j=s$, arrows incident to $(j, r+s)$ in  $Q_\xi^{s}[i_1,j]$ are shown in Figure \ref{arrows incident with 5}, where $a=d_{j-1}q^{t_{j-1}}(s)(1-\delta_{1,s})$, $b=d_{j-1} q^{t_{j-1}}(s)\delta_{1,s}$, $c=d_{j-1}p^{t_{j-1}}(s-1)+(1-d_{j-1})$, $d=(1-\delta_{i_1,j_\bullet}-\delta_{{i_1}_\bullet, j_\bullet})p^{t_{j_\bullet}}(s)  d_{j_\bullet-1}+\delta_{i_1,j_\bullet}$.
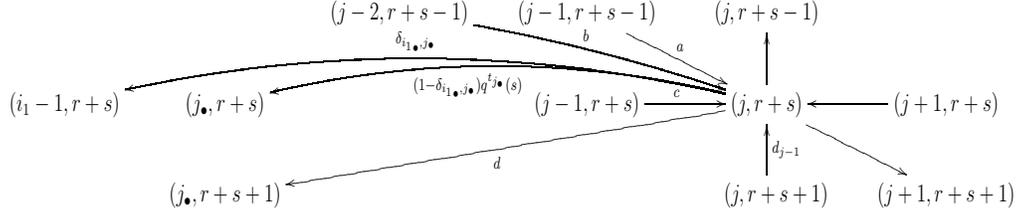
\begin{figure}[H] 
\resizebox{.85\textwidth}{.8\height}{
\begin{minipage}[c]{0.9\textwidth}
\xymatrix{
&&(j-2, r+s-1) \ar@/^0.4pc/[rrd]^{b} & (j-1, r+s-1) \ar[rd]^{a} &(j, r+s-1) &\\
(i_1-1, r+s) &(j_\bullet, r+s) & & (j-1, r+s) \ar[r]^{c}  & (j, r+s)  \ar[u]  \ar[dr] \ar@/^-1.5pc/[lll]^{(1-\delta_{{i_1}_\bullet, j_\bullet})q^{t_{j_\bullet}}(s)\quad\quad\quad} \ar@/^-1.8pc/[llll]_{\delta_{{i_1}_\bullet, j_\bullet}} \ar[dlll]^{d}& (j+1, r+s) \ar[l]\\
&(j_\bullet, r+s+1) 	&& &\quad (j, r+s+1)  \ar[u]_{d_{j-1}}  & (j+1, r+s+1)  
}
\end{minipage}}
\caption{Arrows incident to $(j, r+s)$ in  $Q_\xi^{s}[i_1,j]$ (see Case 7).}\label{arrows incident with 5}
\end{figure}

{\bf Case 8.} In the case where $j>i_1$, $t_j\geq 0$ and $\xi(j)=\xi(j+1)+1$, arrows incident to $(j, r)$ in  $Q_\xi^{0}[i_1,j]$ are shown in Figure \ref{arrows incident with 7}, where $a=d_{j-1}p^{0}(t_{j-1})$, $b=d_{j-1}q^{0}(t_{j-1})$, $c=(\delta_{i_1, j_\bullet}+(1-\delta_{{i_1}_\bullet, j_\bullet}-\delta_{i_1,j_\bullet})d_{j_\bullet-1})q^{-1}(t_{j_\bullet})$, $d=(\delta_{i_1,j_\bullet}+(1-\delta_{{i_1}_\bullet, j_\bullet}-\delta_{i_1,j_\bullet})d_{j_\bullet-1})p^0(t_{j_\bullet})$, $e=(\delta_{i_1,j_\bullet}+(1-\delta_{{i_1}_\bullet, j_\bullet}-\delta_{i_1,j_\bullet})d_{j_\bullet-1})q^0( t_{j_\bullet})$, $f=(\delta_{i_1,j_\bullet}+(1-\delta_{{i_1}_\bullet, j_\bullet}-\delta_{i_1,j_\bullet})d_{j_\bullet-1})p^1( t_{j_\bullet}) $.
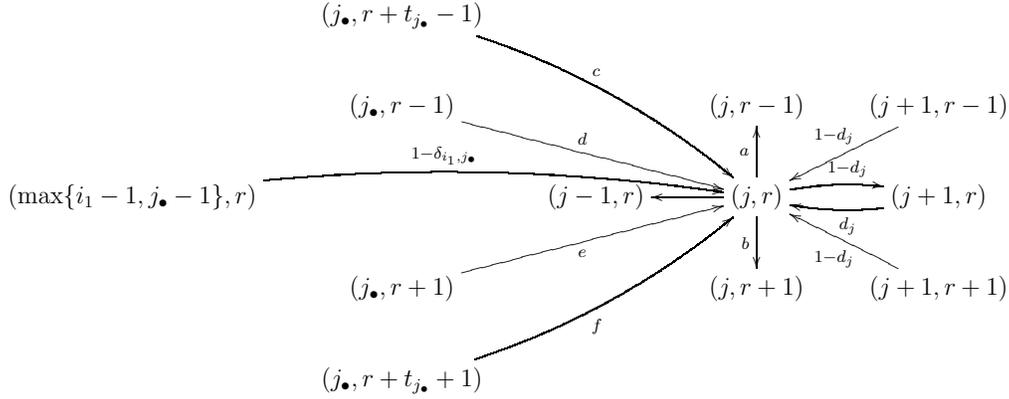
\begin{figure}[H]
\resizebox{.85\textwidth}{.8\height}{
\xymatrix{
&(j_\bullet, r+t_{j_\bullet}-1) \ar@/^1pc/[ddrr]^{c}&   &  \\
&(j_\bullet, r-1) \ar[drr]^{d}  &  &(j, r-1) &(j+1, r-1) \ar[dl]_{1-d_{j}} \\
(\max\{i_1 -1,j_{\bullet}-1\}, r) \ar@/^1pc/[rrr]^{1-\delta_{i_1, j_\bullet}}& & (j-1, r) & (j, r) \ar[u]^{a} \ar[l]\ar@/^/[r]^{1-d_{j}}\ar[d]_{b} & (j+1, r)\ar@/^/[l]^{d_{j}} \\ 
&(j_\bullet, r+1)\ar[urr]_{e}  &  & (j, r+1)  & (j+1, r+1)\ar[ul]^{1-d_{j}}\\
&(j_\bullet, r+t_{j_\bullet}+1) \ar@/_1pc/[uurr]_{f} &  & & 		
}}
\caption{Arrows incident to $(j, r)$ in  $Q_\xi^{0}[i_1,j]$ (see Case 8).}\label{arrows incident with 7}					
\end{figure}

{\bf Case 9.} In the case where $j>i_1$, $t_j \geq1$ and $\xi(j)=\xi(j+1)+1$, arrows incident to $(j, r+1)$ in  $Q_\xi^{1}[i_1,j]$ are shown in Figure \ref{arrows incident with 8}, where $X=\max\{x \mid i_1\leq x<j, \xi(x)=\xi(x+1)+1\}$, $Y=X-1$,  $a=(1-\delta_{i_1,X})d_{X-1}q^0(t_Y),$ $b=d_{X-1}d_{j-1}(\delta_{1,t_X}+p^{2}(t_X)\delta_{1,t_j}) +d_{j-1}(1-\delta_{i_1,X})(1-d_{X-1})\delta_{1,t_j}$, $c=(1-d_{j-1})p^{1}(t_{j-1})$, $d=d_{j-1}p^2(t_{j-1})+(1-d_{j-1})((1-d_{i_1})\delta_{j-1, i_1}+(1-\delta_{j-1,i_1})d_{j-2}q^{1}(t_{j-1}))$, $e=(1-\delta_{X,i_1})d_{X-1}d_{j-1}\delta_{1,t_j}q^{1}(t_X)+\delta_{X,i_1}d_{j-1}\delta_{1,t_j}$,  $f=(1-d_{j-1})+d_{j-1}p^{2}(t_{j-1})$, $g=d_{j-1}\delta_{{i_1}_\bullet,(j-1)_\bullet}\delta_{t_j,1}$.
\begin{figure}[H] 
\resizebox{.85\textwidth}{.8\height}{
\begin{minipage}[c]{0.8\textwidth}
\xymatrix{
& (X, r-1) &  & &   &  \\
(Y,r)&  & (j-1, r)\ar[drr] & &  &   \\
(i_1 -1, r+1)& (X, r+1)  & (j-1, r+1) &&  (j, r+1)  \ar@/^-2pc/[llluu]_{\qquad(1-\delta_{i_1,X}) d_{X-1} p^{0}(t_{X})+\delta_{i_1,X}}  \ar@/^-1pc/[llllu]^{a }  \ar@/^-1pc/[lll]^{\quad  b }   \ar@/^-1.3pc/[llll]_{g\quad\quad } \ar[ll]^{c}\ar[r] \ar[lld]^{d} \ar[dlll]^{e} & (j+1, r+1)  \\
& (X, r+2) &  (j-1, r+2) & & (j, r+2) \ar[u]^{f}&  
}
\end{minipage}}
\caption{Arrows incident to $(j, r+1)$ in  $Q_\xi^{1}[i_1,j]$ (see Case 9).}\label{arrows incident with 8}					
\end{figure}
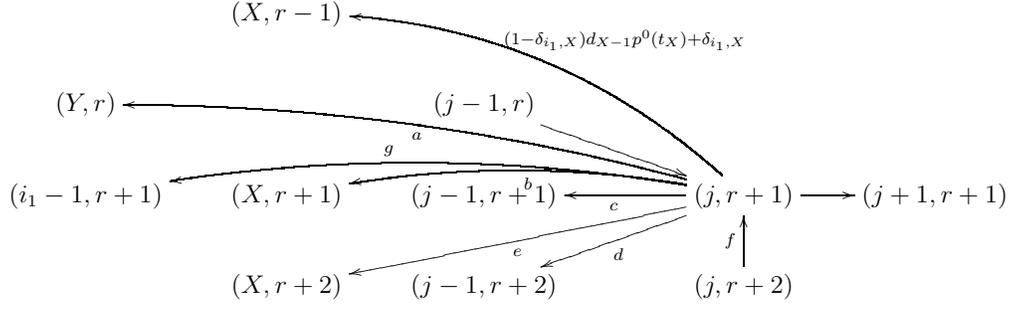

{\bf Case 10.} In the case where $j > i_1$, $t_j \geq 2$, $t_j\geq s \geq2$ and $\xi(j)=\xi(j+1)+1$, arrows incident to $(j, r+s)$ in  $Q_\xi^{s}[i_1, j]$ are shown in Figure \ref{arrows incident with 9}, where $X=\max\{x \mid i_1 \leq  x<j, \xi(x)=\xi(x+1)+1\}$, $Y=X-1$,  $a=(1-d_{j-1})d_{j-2}p^{1}(t_{j-1})p^{t_{j-1}}(s-1)+ d_{j-1}d_{X-1}p^{\min\{t_X, t_{j-1}\}}(s)p^{1}(t_X)$, $b=(1-d_{j-1})q^{t_{j-1}}(s)$,  $c=(1-\delta_{X,i_1})d_{X-1}d_{j-1}\delta_{s, t_{j}}p^{t_X}(t_{j})+\delta_{X,i_1}d_{j-1}\delta_{s, t_{j}},$ $d=d_{j-1}(1-\delta_{s, t_{j}})+(1-d_{j-1})((1-d_{i_1})\delta_{j-1,i_1}+(1-\delta_{j-1,i_1})d_{j-2}p^{t_{j-1}}(s))$, $e=d_{j-1}(1-\delta_{i_1,X})(1-d_{X-1})\delta_{s, t_j}$, $f=(1-d_{j-1})+d_{j-1}p^{s+1}(t_{j-1})$.
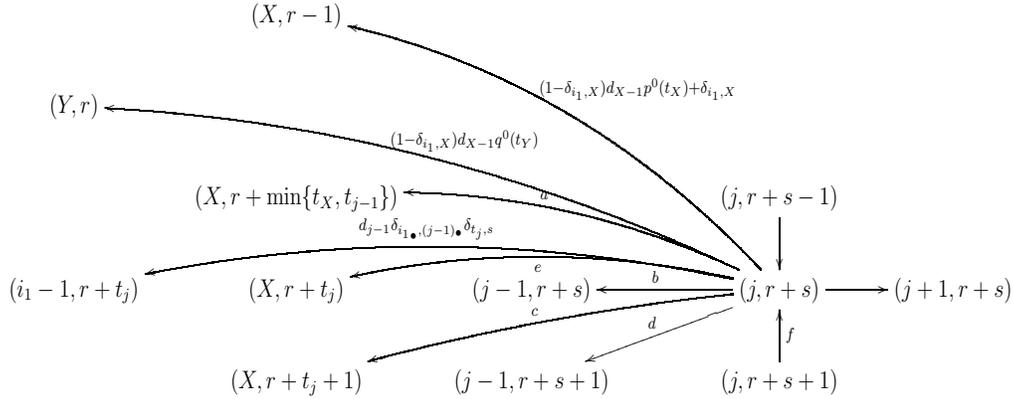
\begin{figure}[H] 
\resizebox{.85\textwidth}{.8\height}{
\xymatrix{
& (X, r-1) & &  &  &  \\
(Y, r)	&  &   &  & &   \\
&  (X, r+ \min\{t_X, t_{j-1}\}) & &  & (j, r+s-1)\ar[d] &  \\
(i_1-1, r+t_j)&  (X, r+t_j) & (j-1, r+s) & & (j, r+s) \ar@/^-1.3pc/[lll]^{e} \ar@/^-1.7pc/[llll]_{ d_{j-1}\delta_{{i_1}_\bullet,(j-1)_\bullet}\delta_{t_j,s}}  \ar@/^-1.5pc/[ulll]_{a}  \ar@/^-2pc/[llluuu]_{\qquad(1-\delta_{i_1,X}) d_{X-1} p^{0}(t_{X})+\delta_{i_1,X}}  \ar@/^-1.6pc/[lllluu]_{\quad (1-\delta_{i_1,X})d_{X-1}q^0(t_Y) }  \ar[ll]_{b}\ar[r] \ar[lld]_{d}  \ar@/^-0.5pc/[llld]_{\quad }_{c} & (j+1, r+s)  \\
&(X, r+t_{j}+1) &  (j-1, r+s+1) & & (j, r+s+1) \ar[u] _{f}&  
}}
\caption{Arrows incident to $(j, r+s)$ in  $Q_\xi^{s}[i_1, j]$ (see Case 10).} \label{arrows incident with 9}					
\end{figure}

Let $L((Y_{i_1,a_1})_{r, r_1}(Y_{i_2,a_2})_{r, r_2}\dots (Y_{i_k,a_k})_{r, r_k})$ be a generalized HL-module. For $1\leq t \leq k$, let
\[
\widetilde{\omega}(i_1, i_t)=(Y_{i_1,a_1})_{r, r_1}(Y_{i_2,a_2})_{r, r_2} \cdots (Y_{i_t,a_t})_{r, r_t}.
\]

We need the following notation:
\begin{align*}
& j_\bullet+1= i_1~\text{ if} ~j_\bullet={i_1}_\bullet,\\
& \widehat{j}=\begin{cases}
i_1 & \text{if $j_\bullet=j-1=i_1$, or $j>i_1$ and $j_\bullet={i_1}_\bullet$}, \\
\max\{x \mid x<j, \xi(x-1)=\xi(x+1)\} & \text{otherwise},
\end{cases} \\
& \widetilde{\omega}(i_1, j)=1~\text{if}~ i_1>j.
\end{align*}

Recall that the initial seed of $K_0(\mathscr{C}_{\ell})$ is given in Section \ref{K0Ccluster algebra}. There is a mutation sequence bringing $K_0(\mathscr{C}_{\ell})$ to the seed $\mathscr{S}_m(\seed)'$, see Section \ref{equivalence classes of generalized HL-modules as cluster variables}.  In addition, all exchange relations are from $T$-system equations in this process of mutations. We denote by $[L(x_{j,s})]$ the cluster variable corresponding to the vertex $(j,r+s)$ in the seed $\mathscr{S}_m(\seed)'$. By Equation (\ref{equ:T-system equation}), we have 
\[
[L(x_{j, s})]=\begin{cases}[L((Y_{j,\xi(j+1)})_{r})]  \qquad &\text{if}~s=0,\\
[L((Y_{j,\xi(j)+1})_{r, s})]  \qquad~&\text{if}~s\neq 0.
\end{cases}
\]

We denote by $[L(x_{[i_1, j], s})]$ the cluster variable corresponding to the vertex $(j, r+s)$ in the seed obtained by applying the mutation sequence $\mathscr{S}'$ to the seed $\mathscr{S}_m(\seed)'$. For simplicity, we write $[x_{j, s}]$ and $[x_{[i_1, j],s}]$ for $[L(x_{j, s})]$ and $[L(x_{[i_1,j],s})]$ respectively. By Lemma \ref{theorem of real prime}, $L(x_{[i_1, j], s})$ is a simple module.

\begin{proposition}\label{pro:cluster variables}
The simple module $L(x_{[i_1,j],s})$ has the highest $l$-weight monomial as follows:

\begin{itemize}
\item[(1)] If $i_1\leq j\leq i_k$, $\xi(j)=\xi(j+1)-1$, $s=0$, we have
\[
x_{[i_1, j], s}=\begin{cases}
 \widetilde{\omega}(i_1, j_\bullet+1) ~&\text{if}~t_j=0, d_j=1,  \\
  \widetilde{\omega}(i_1, j_\bullet+1) (Y_{j+1,\xi(j+1)+1})_{r} ~&\text{if}~ d_j=0,  \\
 \widetilde{\omega}(i_1, j_{\bullet}+1) (Y_{j+1,\xi(j+1)+1})_{r,1} ~&\text{if}~t_j>0, d_j=1, \\	
\widetilde{\omega}(i_1, j_{\bullet}+1)(Y_{j+1,\xi(j+1)+1})_{r,-1} ~&\text{if}~t_j<0, d_j=1.
\end{cases}
\]

\item[(2)] If $i_1\leq j\leq i_k$, $\xi(j)=\xi(j+1)-1$, $t_j\geq s>0$, we have
\begin{gather}
\begin{align*}
x_{[i_1, j],s}=\begin{cases}
(Y_{i_1-1,\xi(i-1)+1})_{r, s-1}   (Y_{j_\bullet+1,\xi(j_\bullet+2)-2(t_{(j_\bullet+1)}+1)})_{t_{(j_\bullet+1)}-s+1}   (Y_{j+1,\xi(j+1)+1})_{r, s} ~&\text{if}~ {i_1}_\bullet=j_\bullet,  \\	
(Y_{i_1, \xi(i_1)+1})_{r, s}   (Y_{j_\bullet+1,\xi(j_\bullet+2)-2(t_{(j_\bullet+1)}+1)})_{t_{(j_\bullet+1)}-s+1}   (Y_{j+1,\xi(j+1)+1})_{r, s}  ~&\text{if}~ i_1 =j_\bullet,   \\	
 \widetilde{\omega}(i_1, \widehat{j_\bullet}) (Y_{j_\bullet,\xi(j_\bullet)-1})_{r, s-1}   (Y_{j_\bullet+1,\xi(j_\bullet+2)-2(t_{(j_\bullet+1)}+1)})_{t_{(j_\bullet+1)}-s+1}   (Y_{j+1,\xi(j+1)+1})_{r, s}  ~&\text{if}~ i_1 <j_\bullet,  d_{j_\bullet-1}=0,\\
(\widetilde{\omega}(i_1, \widehat{j_\bullet}))^{q^{t_{j_\bullet}}(s)}(Y_{j_\bullet,\xi(j_\bullet)+1})_{r, s}   (Y_{j_\bullet+1,\xi(j_\bullet+2)-2(t_{(j_\bullet+1)}+1)})_{t_{(j_\bullet+1)}-s+1}   (Y_{j+1,\xi(j+1)+1})_{r, s} ~&\text{if}~ i_1 <j_\bullet, d_{j_\bullet-1}=1.
\end{cases}
\end{align*}
\end{gather}

\item[(3)] If $i_1\leq j\leq i_k$, $\xi(j)=\xi(j+1)-1$, $ t_j\leq s<0$,  we have
\begin{gather}
\begin{align*}
x_{[i_1, j],s}=\begin{cases}
(Y_{i_1-1,\xi(i_1 -1)+1})_{r, t_{j}} (Y_{j,\xi(j+1)+2(r+t_j)})_{s-t_j+1} (Y_{j+1,\xi(j+1)+1})_{r, t_{j}-1}  ~&\text{if}~ {i_1}_\bullet=j_\bullet,\\
(Y_{j_\bullet,\xi(j_{\bullet})+1})_{r, t_{j}-1} (Y_{j,\xi(j+1)+2(r+t_j)})_{s-t_j+1} (Y_{j+1,\xi(j+1)+1})_{r, t_{j}-1} ~&\text{if}~ i_1=j_\bullet,\\
 \widetilde{\omega}(i_1, \widehat{j_\bullet})  (Y_{j_\bullet,\xi(j_{\bullet})-1})_{r, t_{j}} (Y_{j,\xi(j+1)+2(r+t_j)})_{s-t_j+1} (Y_{j+1,\xi(j+1)+1})_{r, t_{j}-1} ~&\text{if}~ i_1 <j_\bullet, d_{j_\bullet-1}=0,\\
( \widetilde{\omega}(i_1, \widehat{j_\bullet}))^{p^{t_{j_\bullet}}(s)} (Y_{j_\bullet,\xi(j_{\bullet})+1})_{r, t_{j}-1} (Y_{j,\xi(j+1)+2(r+t_j)})_{s-t_j+1} (Y_{j+1,\xi(j+1)+1})_{r, t_{j}-1} ~&\text{if}~  i_1 <j_\bullet, d_{j_\bullet-1}=1.
\end{cases}
\end{align*}
\end{gather}					 

\item[(4)] If $i_1\leq j\leq i_k+1$, $\xi(j)=\xi(j+1)+1$, we have 
\[
x_{[i_1, j],s}=\begin{cases}
\widetilde{\omega}(i_1, j_\bullet+1)~&\text{if}~s=0, d_j=1,\\
\widetilde{\omega}(i_1, j_\bullet+1) (Y_{j+1, \xi(j+1)-1})_r~&\text{if}~s=0, d_j=0, \\
 \widetilde{\omega}(i_1, \widehat{j})(Y_{j,\xi(j)-1})_{r, s} ~&\text{if}~s\neq 0, d_{j-1}=0,\\
 \widetilde{\omega}(i_1, \widehat{j}) ~&\text{if}~ s=t_{j-1}, t_{j-1} \neq 0, d_{j-1}=1,\\
 \widetilde{\omega}(i_1, \widehat{j})(Y_{j,\xi(j)+1})_{r, s+1} ~&\text{if}~0<s<t_{j-1},  d_{j-1}=1,\\
 \widetilde{\omega}(i_1, \widehat{j})(Y_{j,\xi(j)+1})_{ r, s-1 } ~&\text{if}~ t_{j-1}<s<0, d_{j-1}=1.
\end{cases}
\]
\end{itemize}
\end{proposition}

\begin{proof}
For $i\in I, j\in\mathbb{Z}, r\in\mathbb{Z}_{\geq1}$, we denote $\prod^{r-1}_{t=0}A_{i, j+2t}$ by $(A_{i, j})_r$.

We prove Proposition \ref{pro:cluster variables} by induction on the length $\boldsymbol{l}$ of mutation sequence. In the following, we only prove the case of $s\geq 0$, since each exchange relation in the case of $s< 0$ can be obtained from one in the case of $s \geq 0$, and the remaining proof can be done as the case of $s \geq 0$.  

For $\boldsymbol{l}=1$, we have two cases to consider.

{\bf Case 1.} In the case where $\xi(i_1)=\xi(i_1+1)+1$, or $\xi(i_1)=\xi(i_1+1)-1$, $t_{i_1}=0$, it follows from Figure \ref{arrows incident with 0} that
\[
[x_{i_1, 0}] [x_{[i_1, i_1],0}]= [x_{i_1,-1}] [x_{i_1,1}] ([x_{i_1+1,0}])^{1-d_{i_1}}+[x_{i_1-1,0}] ([x_{i_1+1,1}][x_{i_1+1,-1}])^{1-d_{i_1}} [x_{i_1+1,0}]^{d_{i_1}}.
\]
Since $ x_{i_1, 0} \nmid  x_{i_1-1, 0} (x_{i_1+1,1} x_{i_1+1,-1})^{1-d_{i_1}} (x_{i_1+1,0})^{d_{i_1}}$, we have 
\[
x_{i_1,0} x_{[i_1, i_1],0} = x_{i_1, -1} x_{i_1, 1}(x_{i_1+1,0})^{1-d_{i_1}}
\]
 by Proposition \ref{lemma:equations satisfy by q-character}. Hence
\begin{gather}
\begin{align*}
 x_{[i_1, i_1],0} &= \frac{x_{i_1,-1} x_{i_1,1} (x_{i_1+1,0})^{1-d_{i_1}}}{x_{i_1,0}} = \frac{(Y_{i_1,\xi(i_1)+1})_{r,-1} (Y_{i_1,\xi(i_1)+1})_{r,1} \big((Y_{i_1+1,\xi(i_1+2)})_r \big)^{1-d_{i_1}}}{(Y_{i_1,\xi(i_1+1)})_r}\\
 &=\begin{cases}
 (Y_{i_1,\xi(i_1)-1})_r \big((Y_{i_1+1,\xi(i_1+2)})_r\big)^{1-d_{i_1}}~&\text{if}~\xi(i_1)=\xi(i_1+1)-1,\\
(Y_{i_1,\xi(i_1)+1})_r \big((Y_{i_1+1,\xi(i_1+2)})_r\big)^{1-d_{i_1}}~&\text{if}~\xi(i_1)=\xi(i_1+1)+1.
 \end{cases}
\end{align*}
\end{gather}

{\bf Case 2.} In the case where $\xi(i_1)=\xi(i_1+1)-1$, $t_{i_1}>0$, it follows from Figure \ref{arrows incident with 1} that
\[
[x_{i_1,t_{i_1}}] [x_{[i_1, i_1],t_{i_1}}] = [x_{i_1-1,t_{i_1}-1}] [x_{i_1,t_{i_1}+1}] [x_{i_1+1,t_{i_1}}] + [x_{i_1-1,t_{i_1}}]  [x_{i_1,t_{i_1}-1}] [x_{i_1+1,t_{i_1}+1}].
\]
Since $ x_{i_1,t_{i_1}} \nmid  x_{i_1-1,t_{i_1}} x_{i_1,t_{i_1}-1} x_{i_1+1,t_{i_1}+1}$, we have $ x_{i_1,t_{i_1}} x_{[i_1, i_1],t_{i_1}} = x_{i_1-1,t_{i_1}-1} x_{i_1, t_{i_1}+1} x_{i_1+1,t_{i_1}}$ by Proposition \ref{lemma:equations satisfy by q-character}. Hence
\begin{gather}
\begin{align*}
 x_{[i_1, i_1],t_{i_1}} &= \frac{x_{i_1-1,t_{i_1}-1} x_{i_1,t_{i_1}+1} x_{i_1+1,t_{i_1}}} {x_{i_1,t_{i_1}}} = \frac{(Y_{i_1-1,\xi(i_1-1)+1})_{r, t_{i_1}-1} (Y_{i_1,\xi(i_1)+1})_{r, t_{i_1}+1} (Y_{i_1+1,\xi(i_1+1)+1})_{r, t_{i_1}}} {(Y_{i_1,\xi(i_1)+1})_{r, t_{i_1}}}  \\
&=(Y_{i_1-1,\xi(i_1-1)+1})_{r, t_{i_1}-1} Y_{i_1,\xi(i_1)+1-2(t_{i_1}+1)} (Y_{i_1+1,\xi(i_1+1)+1})_{r, t_{i_1}}.
\end{align*}
\end{gather}

Suppose that our result holds for $\boldsymbol{l}\geq m$, $m\in \mathbb{Z}_{\geq 1}$. By induction, we need to prove it for $\boldsymbol{l}=m+1$.

For $\boldsymbol{l}=m+1$, applying the first $m$ vertices in $\mathscr{S}'$ to the seed $\mathscr{S}_m(\seed)'$, we get a quiver and in this quiver, all arrows incident to the $(m+1)$th vertex in $\mathscr{S}'$ must be one of diagrams shown in Figures \ref{arrows incident with 2}--\ref{arrows incident with 9}. We next prove Proposition \ref{pro:cluster variables} by cases.

{\bf Case 1.} In the case where $\xi(i_1)=\xi(i_1+1)-1$, $t_{i_1}>0$, $j=i_1$, $s=0$, it follows from Figure \ref{arrows incident with 2} (left) that
\[
[x_{i_1, 0}] [x_{[i_1,i_1],0}] = [x_{i_1, -1}] [x_{i_1, t_{i_1}+1}] [x_{i_1+1, 1}]^{d_{i_1}} [x_{i_1+1,0}]^{1-d_{i_1}} + [x_{[i_1, i_1],1}] [x_{i_1+1,-1}]^{1-d_{i_1}} [x_{i_1+1,0}]^{d_{i_1}}.
\]
Since $x_{i_1,0} \nmid  x_{[i_1, i_1],1} (x_{i_1+1,-1})^{1-d_{i_1}} (x_{i_1+1,0})^{d_{i_1}}$ by induction, we have 
\[
 x_{i_1,0} x_{[i_1, i_1],0} =   x_{i_1,-1} x_{i_1, t_{i_1}+1} (x_{i_1+1,1})^{d_{i_1}}  (x_{i_1+1,0})^{1-d_{i_1}}
 \]
by Proposition \ref{lemma:equations satisfy by q-character}. Hence
\begin{gather}
\begin{align*}
x_{[i_1,i_1],0} &=\frac{   x_{i_1,-1} x_{i_1,t_{i_1}+1} (x_{i_1+1,1})^{d_{i_1}}  (x_{i_1+1,0})^{1-d_{i_1}} }{x_{i_1,0} }
\\
&=\frac{  (Y_{i_1,\xi(i_1)+1})_{r,-1} (Y_{i_1, \xi(i_1)+1})_{r, t_{i_1}+1}  ((Y_{i_1+1,\xi(i_1+1)+1})_{r,1})^{d_{i_1}} ((Y_{i_1+1,\xi(i_1+2)})_{r})^{1-d_{i_1} }}{(Y_{i_1,\xi(i_1+1)})_{r} }
\\
&= (Y_{i_1,\xi(i_1)-1-2t_{i_1}})_{r+t_{i_1}} ((Y_{i_1+1,\xi(i_1+1)+1})_{r,1})^{d_{i_1}} ((Y_{i_1+1,\xi(i_1+2)})_{r})^{1-d_{i_1} }.
\end{align*}
\end{gather}

{\bf Case 2.} In the case where $\xi(i_1)=\xi(i_1+1)-1$, $j=i_1$, $t_{i_1} > s >0$, it follows from Figure \ref{arrows incident with 2} (right) that
\[
[x_{i_1, s }] [x_{[i_1, i_1], s }] = [x_{i_1-1, s -1}] [x_{i_1, t_{i_1}+1}] [x_{i_1+1, s }] + [x_{i_1, s -1}] [x_{[i_1, i_1], s +1}].
\]
Since $x_{i_1, s } \nmid x_{i_1, s -1} x_{[i_1, i_1], s+1 }$ by induction, we have $x_{i_1, s } x_{[i_1, i_1], s } = x_{i_1-1, s -1} x_{i_1, t_{i_1}+1} x_{i_1+1, s }$ by Proposition \ref{lemma:equations satisfy by q-character}. Hence
\begin{gather}
\begin{align*}
x_{[i_1, i_1], s } &= \frac{x_{i_1-1, s -1} x_{i_1, t_{i_1}+1}  x_{i_1+1, s } }{x_{i_1, s } } =\frac{(Y_{i_1-1,\xi(i_1-1)+1})_{r, s -1}  (Y_{i_1,\xi(i_1)+1})_{r, t_{i_1}+1} (Y_{i_1+1,\xi(i_1+1)+1})_{r, s } }{(Y_{i_1,\xi(i_1)+1})_{r, s }}\\
&= (Y_{i_1-1,\xi(i_1-1)+1})_{r, s -1} (Y_{i_1,\xi(i_1)+1-2(t_{i_1}+1)})_{t_{i_1}- s +1} (Y_{i_1+1,\xi(i_1+1)+1})_{r, s }.
 \end{align*}
 \end{gather}

{\bf Case 3.} In the case where $j>i_1$, $\xi(j)=\xi(j+1)-1$, $s=t_j=0$, it follows from Figure \ref{arrows incident with 3} that
\begin{align*}
[x_{j,0}] [x_{[i_1, j],0}]=& [x_{[i_1, j-1],0}] ([x_{j,-1}] [x_{j,1}])^{d_{j-1}} [x_{j+1,0}]^{1-d_j} + 
 \bigg( [x_{i_1-1,0}]^{\delta_{{i_1}_\bullet, j_\bullet}} ([x_{i_1,1}] [x_{i_1,-1}])^{\delta_{i_1, j_\bullet}} + \\ 
&(1-\delta_{{i_1}_\bullet, j_\bullet}- \delta_{i_1, j_\bullet}) [x_{[i_1, j_\bullet-1],0}]  \big( [x_{j_\bullet,1}]^{q^0(t_{j_\bullet})} [x_{j_\bullet,-1}]^{p^0(t_{j_\bullet})} \big)^{d_{j_\bullet-1}} \bigg)\\
&([x_{j+1,-1}] [x_{j+1,1}])^{1-d_j} [x_{j+1,0}]^{d_j}.
\end{align*}
Since
\begin{align*}
x_{j,0} \nmid  &
 \bigg( (x_{i_1-1,0})^{\delta_{{i_1}_\bullet, j_\bullet}} ( x_{i_1,1}  x_{i_1,-1} )^{\delta_{i_1, j_\bullet}} + 
(1-\delta_{{i_1}_\bullet, j_\bullet}- \delta_{i_1, j_\bullet})  x_{[i_1, j_\bullet-1],0}  \\ 
&\big( ( x_{j_\bullet,1})^{q^0(t_{j_\bullet})} (x_{j_\bullet,-1})^{p^0(t_{j_\bullet})} \big)^{d_{j_\bullet-1}} \bigg)( x_{j+1,-1} x_{j+1,1} )^{1-d_j} (x_{j+1,0})^{d_j}
\end{align*}
by induction, we have
\[
 x_{j,0} x_{[i_1, j],0} = x_{[i_1, j-1],0}  ( x_{j,-1} x_{j,1} )^{d_{j-1}}  (x_{j+1,0})^{1-d_j} 
\]
by Proposition \ref{lemma:equations satisfy by q-character}. Hence 
\begin{gather}
\begin{align*}
 x_{[i_1,j],0} &= \frac{ x_{[i_1,j-1],0} (x_{j,-1} x_{j,1})^{d_{j-1}} (x_{j+1,0})^{1-d_j}} { x_{j,0}}\\
&=\begin{cases}
 \frac{ \widetilde{\omega}(i_1,(j-1)_\bullet+1)  (Y_{j ,\xi(j )+1})_{r,-1} (Y_{j ,\xi(j )+1})_{r,1}  ((Y_{j+1,\xi(j+2)})_{r})^{1-d_{j }}}{ (Y_{j ,\xi(j+1)})_{r}}=\widetilde{\omega}(i_1, j ) ((Y_{j+1,\xi(j+2)})_{r})^{1-d_{j }} ~&\text{if}~d_{j-1}=1,\\
 \frac{ \widetilde{\omega}(i_1, \widehat{j} )  (Y_{j,\xi(j+1)})_{r}((Y_{j+1,\xi(j+2)})_{r})^{1-d_{j}}}{ (Y_{j,\xi(j+1)})_{r}}=\widetilde{\omega}(i_1, \widehat{j}) ((Y_{j+1,\xi(j+2)})_{r})^{1-d_{j}}~&\text{if}~d_{j-1}=0.
 \end{cases}
\end{align*}	
\end{gather}

{\bf Case 4.} In the case where $j>i_1$, $\xi(j)=\xi(j+1)-1$, $t_j>0$, it follows from Figure \ref{arrows incident with 4} that
\begin{gather}
\begin{align*}
[x_{j,0}] [x_{[i_1, j],0}]=& [x_{[i_1, j-1],0}] [x_{j,-1}]^{d_{j-1}} [x_{j, t_j+1}]^{d_{j-1}} [x_{j+1,0}]^{1-d_j} [x_{j+1,1}]^{d_j}   + \bigg(  [x_{i_1,-1} ]^{\delta_{i_1, j_\bullet}} + \\ 
&(1-\delta_{{i_1}_\bullet, j_\bullet}- \delta_{i_1,j_\bullet}) [x_{[i_1, j_\bullet-1],0}]^{q^{0}(t_{j_\bullet})}    [x_{j_\bullet,-1}]^{p^0(t_{j_\bullet})d_{j_\bullet-1}} \bigg) [x_{[i_1, j],1} ][x_{j+1,-1}]^{1-d_j} [x_{j+1,0}]^{d_j}.
\end{align*}
\end{gather}
Since 
\begin{gather}
\begin{align*}
x_{j,0}  \nmid \bigg(  (x_{i_1, -1} )^{\delta_{i_1, j_\bullet}} + 
(1-\delta_{{i_1}_\bullet, j_\bullet}- \delta_{i_1, j_\bullet}) (x_{[i_1, j_\bullet-1],0})^{q^{0}(t_{j_\bullet})}    (x_{j_\bullet,-1})^{p^0(t_{j_\bullet})d_{j_\bullet-1}} \bigg) 
x_{[i_1, j],1} (x_{j+1,-1} )^{1-d_j} (x_{j+1,0})^{d_j}
\end{align*}
\end{gather}
by induction, we have
\[
 x_{j,0} x_{[i_1, j],0} =  x_{[i_1, j-1],0} (x_{j,-1} )^{d_{j-1}} (x_{j, t_j+1})^{d_{j-1}} (x_{j+1,0})^{1-d_j} (x_{j+1,1})^{d_j}
\]
by Proposition \ref{lemma:equations satisfy by q-character}. Hence 

\begin{gather}
\begin{align*}
 x_{[i_1, j],0}&= \frac{x_{[i_1, j-1],0} (x_{j,-1} )^{d_{j-1}} (x_{j, t_j+1})^{d_{j-1}} (x_{j+1,0})^{1-d_j} (x_{j+1,1})^{d_j} }{x_{j,0}}\\
& =  \begin{cases}
   \frac{\widetilde{\omega}(i_1, (j-1)_{\bullet}+1) (Y_{j, \xi(j)+1})_{r,-1} (Y_{j,\xi(j)+1})_{r, t_{j}+1}  ((Y_{j+1,\xi(j+2)})_{r})^{ 1-d_{j} }  ((Y_{j+1,\xi(j+1)+1})_{r,1})^{ d_{j} }} {(Y_{j,\xi(j+1)})_{r} }~&\text{if}~ d_{j-1}=1,\\
 \frac{\widetilde{\omega}(i_1, (j-1)_{\bullet}+1) (Y_{j,\xi(j)+1})_r ((Y_{j+1,\xi(j+2)})_{r})^{1-d_{j }}  ((Y_{j+1,\xi(j+1)+1})_{r,1})^{d_{j }}} {(Y_{j ,\xi(j+1)})_{r} }
 ~&\text{if}~ d_{j-1}=0,
 \end{cases}\\
&=  \begin{cases}
\widetilde{\omega}(i_1, j ) ((Y_{j+1,\xi(j+2)})_{r})^{ 1-d_{j } }  ((Y_{j+1,\xi(j+1)+1})_{r,1})^{ d_{j } }  ~&\text{if}~ d_{j-1}=1,\\
  \widetilde{\omega}(i_1, (j-1)_{\bullet}+1)  ((Y_{j+1,\xi(j+2)})_{r})^{1-d_{j }}  ((Y_{j+1,\xi(j+1)+1})_{r,1})^{d_{j }}  ~&\text{if}~ d_{j-1}=0.
   \end{cases}
\end{align*}
 \end{gather}
  
{\bf Case 5.} In the case where $j>i_1$, $\xi(j)=\xi(j+1)-1$, $t_j>s>0$, it follows from Figure \ref{arrows incident with 6} that
\begin{gather}			
\begin{align*}
[x_{j, s }] [x_{[i_1, j], s }] = &[x_{[i_1, j_\bullet], s }]^{\delta_{ s ,t_{j_\bullet}}}  [x_{j, s -1}]  [x_{[i_1, j], s +1}] +  [x_{[i_1, j-2],0}]^{ q^{t_{j-1}}( s ) d_{j-1}\delta_{1, s }} [x_{[i_1, j-1], s -1}]^{d_{j-1}  q^{t_{j-1}}( s )  (1-\delta_{1, s })}\\
&  [x_{j-1, s }]^{d_{j-1}  p^{t_{j-1}}( s -1) } [x_{[i_1, j-1], s }]^{1-d_{j-1}} [x_{j, t_j+1}]^{d_{j-1}} [x_{j+1, s }].
\end{align*}
\end{gather}
Since $x_{j, s } \nmid (x_{[i_1, j_\bullet], s })^{\delta_{ s ,t_{j_\bullet}}} x_{j, s -1}  x_{[i_1, j], s +1}$ by induction, we have
\begin{align*}
 x_{j, s } x_{[i_1, j], s }=&(x_{[i_1, j-2],0})^{ q^{t_{j-1}}( s ) d_{j-1}\delta_{1, s }} 
(x_{[i_1, j-1], s -1})^{d_{j-1}  q^{t_{j-1}}( s )  (1-\delta_{1, s })} (x_{j-1, s })^{d_{j-1}  p^{t_{j-1}}( s -1) }\\
& (x_{[i_1, j-1], s })^{1-d_{j-1}} (x_{j, t_j+1})^{d_{j-1}} x_{j+1, s }
\end{align*}
by Proposition \ref{lemma:equations satisfy by q-character}. Hence 
\begin{gather}
\begin{align*}
x_{[i_1, j], s } & =  \begin{cases}
\displaystyle  \frac{ x_{j-1, s }   x_{j, t_j+1} x_{j+1, s } } {x_{j, s } }  ~&\text{if}~ d_{j-1}=1, t_{j-1}<s,\\
\displaystyle \frac { x_{[i_1, j-2],0}  x_{j, t_j+1}  x_{j+1, s } } { x_{j, s } }  ~&\text{if}~ d_{j-1}=1,s=1,t_{j-1}\geq 1,\\
\displaystyle\frac{ x_{[i_1, j-1], s -1}  x_{j, t_j+1}  x_{j+1, s }} {x_{j, s } }   ~&\text{if}~ d_{j-1}=1,t_{j-1}\geq s>1,\\
\displaystyle\frac{  x_{[i_1, j-1], s }  x_{j+1, s }} {x_{j, s } }~&\text{if}~ d_{j-1}=0,
\end{cases}\\
&=  \begin{cases}
\displaystyle\frac{(Y_{j-1,\xi(j-1)+1})_{r, s} (Y_{j,\xi(j)+1})_{r, t_{j}+1} (Y_{j+1,\xi(j+1)+1})_{r, s}} {(Y_{j ,\xi(j )+1})_{r, s}}
& \text{if}~ d_{j-1}=1, t_{j-1}< s,\\
\displaystyle\frac{ x_{[i_1, j-2],0} (Y_{j,\xi(j)+1})_{r, t_{j}+1} (Y_{j+1,\xi(j+1)+1})_{r, s}} {(Y_{j,\xi(j)+1})_{r, s}}    ~&\text{if}~ d_{j-1}=1,s=1,t_{j-1}\geq 1,\\
\displaystyle \frac{  x_{[i_1, j-1], s -1} (Y_{j,\xi(j)+1})_{r, t_{j}+1} (Y_{j+1,\xi(j+1)+1})_{r, s }} {(Y_{j,\xi(j)+1})_{r, s }}   ~&\text{if}~ d_{j-1}=1,t_{j-1}\geq s>1,\\
\displaystyle\frac{ x_{[i_1, j-1], s} (Y_{j+1,\xi(j+1)+1})_{r, s } } {(Y_{j,\xi(j)+1})_{r, s }}  ~&\text{if}~ d_{j-1}=0,
\end{cases}
\end{align*}
\end{gather}

\begin{gather}
\begin{align*}
& =  \begin{cases}
(Y_{j-1,\xi(j-1)+1})_{r, s} (Y_{j,\xi(j)+1-2(t_{j}+1)})_{t_{j}- s +1} (Y_{j+1,\xi(j+1)+1})_{r, s  } \hspace*{\fill} ~\text{if}~ d_{j-1}=1, t_{j-1}< s, \\
\widetilde{\omega}(i_1,\widehat{j-1}) (Y_{j-1, \xi(j-1)-1})_{r, s -1}  (Y_{j, \xi(j)+1-2(t_{j}+1)})_{ t_{j}- s +1} (Y_{j+1,\xi(j+1)+1})_{r, s} \\
\hspace*{\fill} ~\text{if}~d_{j-1}=1,t_{j-1}\geq s\geq1,d_{j-2}=0, \\
\widetilde{\omega}(i_1,\widehat{j-1}) (Y_{j-1,\xi(j-1)+1 })_{r, s }  (Y_{j,\xi(j)+1-2(t_{j}+1)})_{ t_{j}- s +1} (Y_{j+1,\xi(j+1)+1})_{r, s }\\
\hspace*{\fill}  ~\text{if}~d_{j-1}=1,t_{j-1}\geq s\geq1,d_{j-2}=1,  \\
\tilde{\omega}(i_1,\widehat{(j-1)_\bullet})^{q^{t_{(j-1)_\bullet}}( s )} (Y_{(j-1)_{\bullet},\xi((j-1)_{\bullet})+1})_{r, s } (Y_{(j-1)_{\bullet}+1,\xi((j-1)_{\bullet}+2)-2(t_{(j-1)_{\bullet}+1}+1)})_{t_{(j-1)_{\bullet}+1}-s +1}\\
\hspace*{\fill}  (Y_{j+1,\xi(j+1)+1})_{r, s } ~\quad \text{if}~d_{j-1}=0, (j-1)_\bullet>i_1, d_{(j-1)_{\bullet}-1}=1,~\text{or}~d_{j-1}=0, (j-1)_\bullet= i_1, \\
\tilde{\omega}(i_1,\widehat{(j-1)_\bullet})  (Y_{(j-1)_{\bullet},\xi((j-1)_{\bullet})+1})_{r, s -1} (Y_{(j-1)_{\bullet}+1,\xi((j-1)_{\bullet}+2)-2(t_{(j-1)_{\bullet}+1}+1)})_{t_{(j-1)_{\bullet}+1}-s+1} \\
\hspace*{\fill}  (Y_{j+1,\xi(j+1)+1})_{r, s } ~\quad \text{if}~d_{j-1}=0, (j-1)_\bullet>i_1, d_{(j-1)_{\bullet}-1}=0,~\text{or}~d_{j-1}=0, (j-1)_\bullet={i_1}_\bullet.\\
\end{cases}
\end{align*}
\end{gather}

{\bf Case 6.} In the case where $j>i_1$, $\xi(j)=\xi(j+1)-1$, $t_j=s>0$, it follows from Figure \ref{arrows incident with 5} that
\begin{gather}			
\begin{align*}
[x_{j, s }] [x_{[i_1, j], s }]=& [x_{i_1-1,  s }]^{\delta_{{i_1}_{\bullet},j_{\bullet}}} [x_{j_\bullet, s +1}]^{\delta_{i_1,j_\bullet}+(1-\delta_{{i_1}_\bullet, j_\bullet}-\delta_{i_1,j_\bullet}) d_{j_\bullet-1} p^{t_{j_\bullet}}( s )} [x_{[i_1,j_\bullet], s }]^{(1-\delta_{{i_1}_\bullet, j_\bullet}) q^{t_{j_\bullet}}( s )}  [x_{j, s -1}]  [x_{j+1, s +1}]  +\\
&[x_{[i_1,j-2],0}]^{d_{j-1}  q^{t_{j-1}}( s ) \delta_{1, s }}  [x_{[i_1,j-1], s -1}]^{d_{j-1} q^{t_{j-1}}( s )  (1-\delta_{1, s })} 	[x_{j-1, s }]^{d_{j-1} p^{t_{j-1}}( s-1 )  } [x_{[i_1,j-1], s }]^{1-d_{j-1}} 	[x_{j, s +1}]^{d_{j-1}} [x_{j+1, s }].
\end{align*}
\end{gather}
Since 
\begin{gather}			
\begin{align*}
x_{j, s } \nmid (x_{i_1-1,  s })^{\delta_{{i_1}_{\bullet}, j_{\bullet}}} (x_{j_\bullet, s +1})^{\delta_{i_1,j_\bullet}+(1-\delta_{{i_1}_\bullet, j_\bullet}-\delta_{i_1,j_\bullet}) d_{j_\bullet-1} p^{t_{j_\bullet}}( s )} (x_{[i_1,j_\bullet], s })^{(1-\delta_{{i_1}_\bullet, j_\bullet}) q^{t_{j_\bullet}}( s )}  x_{j, s -1}  x_{j+1, s +1}
\end{align*}
\end{gather}
by induction, we have
\begin{align*}
x_{j, s } x_{[i_1,j], s }= & (x_{[i_1,j-2],0})^{d_{j-1}  q^{t_{j-1}}( s ) \delta_{1, s }}  (x_{[i_1,j-1], s -1})^{d_{j-1} q^{t_{j-1}}( s )  (1-\delta_{1, s })} 	(x_{j-1, s })^{d_{j-1} p^{t_{j-1}}( s-1 )  }\\
& (x_{[i_1,j-1], s })^{1-d_{j-1}} (x_{j, s +1})^{d_{j-1}} x_{j+1, s }
\end{align*}
by Proposition \ref{lemma:equations satisfy by q-character}. 
The rest of the proof is the same as the proof of Case 5.
 
{\bf Case 7.} In the case where $j>i_1$, $\xi(j)=\xi(j+1)+1$, $t_j \geq 0$, $s=0$, it follows from Figure \ref{arrows incident with 7} that		
 \begin{align*}
[x_{j,0}] [x_{[i_1,j],0}]=& [x_{i_1-1, 0}]^{\delta_{{i_1}_\bullet, j_\bullet}} 
[x_{[i_1,j_\bullet-1],0}]^{1-\delta_{{i_1}_\bullet, j_\bullet}-\delta_{i_1,j_\bullet}}  [x_{j_\bullet,-1}]^{( \delta_{i_1,j_\bullet}+(1-\delta_{i_1,j_\bullet}-\delta_{{i_1}_\bullet, j_\bullet}) d_{j_\bullet-1}) p^{0}(t_{j_\bullet})} \\
&[x_{j_\bullet,1}]^{( \delta_{i_1,j_\bullet}+(1-\delta_{i_1,j_\bullet}-\delta_{{i_1}_\bullet, j_\bullet}) d_{j_\bullet-1}) q^{0}(t_{j_\bullet})}  [x_{j_\bullet, t_{j_\bullet}+1}]^{( \delta_{i_1,j_\bullet}+(1-\delta_{i_1,j_\bullet}-\delta_{{i_1}_\bullet, j_\bullet}) d_{j_\bullet-1})p^{1}(t_{j_\bullet})} \\ 
  &[x_{j_\bullet, t_{j_\bullet}-1}]^{( \delta_{i_1,j_\bullet}+(1-\delta_{i_1,j_\bullet}-\delta_{{i_1}_\bullet, j_\bullet}) d_{j_\bullet-1})q^{-1}(t_{j_\bullet})} 	([x_{j+1,-1}] [x_{j+1,1}])^{1-d_j} [x_{j+1,0}]^{d_j} +\\
& [x_{[i_1,j-1],0}]  [x_{j,-1}]^{d_{j-1}p^{0}(t_{j-1})} [x_{j,1}]^{d_{j-1}q^{0}(t_{j-1})} [x_{j+1,0}]^{1-d_j}.		
\end{align*}
Since 
 \begin{align*}
x_{j,0} \nmid &(x_{i_1-1, 0})^{\delta_{{i_1}_\bullet, j_\bullet}} 
(x_{[i_1,j_\bullet-1],0})^{1-\delta_{{i_1}_\bullet, j_\bullet}-\delta_{i_1,j_\bullet}}  (x_{j_\bullet,-1})^{( \delta_{i_1,j_\bullet}+(1-\delta_{i_1,j_\bullet}-\delta_{{i_1}_\bullet, j_\bullet}) d_{j_\bullet-1}) p^{0}(t_{j_\bullet})} \\
&(x_{j_\bullet,1})^{( \delta_{i_1,j_\bullet}+(1-\delta_{i_1,j_\bullet}-\delta_{{i_1}_\bullet, j_\bullet}) d_{j_\bullet-1}) q^{0}(t_{j_\bullet})}  (x_{j_\bullet, t_{j_\bullet}+1})^{( \delta_{i_1,j_\bullet}+(1-\delta_{i_1,j_\bullet}-\delta_{{i_1}_\bullet, j_\bullet}) d_{j_\bullet-1})p^{1}(t_{j_\bullet})} \\   &(x_{j_\bullet, t_{j_\bullet}-1})^{( \delta_{i_1,j_\bullet}+(1-\delta_{i_1,j_\bullet}-\delta_{{i_1}_\bullet, j_\bullet}) d_{j_\bullet-1})q^{-1}(t_{j_\bullet})} 	(x_{j+1,-1} x_{j+1,1})^{1-d_j} (x_{j+1,0})^{d_j}  		
\end{align*}
by induction, we have
\[
x_{j,0} x_{[i_1,j],0}=  x_{[i_1,j-1],0}  (x_{j,-1})^{d_{j-1}p^{0}(t_{j-1})}  
(x_{j,1})^{d_{j-1}q^{0}(t_{j-1})} (x_{j+1,0})^{1-d_j}  	
\]
by Proposition \ref{lemma:equations satisfy by q-character}. Hence
\begin{align*}
 x_{[i_1,j],0}&=\frac{  x_{[i_1,j-1],0}  (x_{j,-1})^{d_{j-1}p^{0}(t_{j-1})}  
(x_{j,1})^{d_{j-1}q^{0}(t_{j-1})} (x_{j+1,0})^{1-d_j} } { x_{j,0} } 	\\
&= \begin{cases}
\frac{ x_{[i_1,j-1],0}  x_{j,-1}  (x_{j+1,0})^{1-d_j}} {x_{j,0}} ~&\text{if}~ d_{j-1}=1, t_{j-1}>0,\\
  \frac{ x_{[i_1,j-1],0}   x_{j,1}  (x_{j+1,0})^{1-d_j} } {x_{j,0}}~&\text{if}~ d_{j-1}=1, t_{j-1}<0,\\
  \frac{ x_{[i_1,j-1],0} x_{j,-1}  x_{j,1}  (x_{j+1,0})^{1-d_j} }{x_{j,0}}~&\text{if}~ d_{j-1}=1, t_{j-1}=0,\\
  \frac{ x_{[i_1,j-1],0}    (x_{j+1,0})^{1-d_j}} {x_{j,0}}~&\text{if}~ d_{j-1}=0,\\
   \end{cases}\\
&= \begin{cases}
   \frac{  \widetilde{\omega}(i_1,(j-1)_\bullet+1) (Y_{j ,\xi(j )+1})_{r,1} (Y_{j ,\xi(j )+1})_{r,-1} ((Y_{j+1,\xi(j+2)})_{r} )^{1-d_{j }}} {(Y_{j ,\xi(j+1)})_{r}}~&\text{if}~ d_{j-1}=1, t_{j-1}>0,\\
 \frac{  \widetilde{\omega}(i_1,(j-1)_\bullet+1) (Y_{j ,\xi(j )+1})_{r,-1} (Y_{j ,\xi(j )+1})_{r,1} ((Y_{j+1,\xi(j+2)})_{r} )^{1-d_{j }}} {(Y_{j ,\xi(j+1)})_{r}} ~&\text{if}~ d_{j-1}=1, t_{j-1}<0,\\
   \frac{ \widetilde{\omega}(i_1,(j-1)_\bullet+1)  (Y_{j ,\xi(j )+1})_{r,-1}(Y_{j ,\xi(j )+1})_{r,1}  ((Y_{j+1,\xi(j+2)})_{r} )^{1-d_{j }}} {(Y_{j ,\xi(j+1)})_{r}}~&\text{if}~ d_{j-1}=1, t_{j-1}=0,\\
\frac{\widetilde{\omega}(i_1, (j-1)_{\bullet}+1 ) (Y_{j, \xi(j)-1})_r ((Y_{j+1,\xi(j+2)})_{r} )^{1-d_{j }}} {(Y_{j ,\xi(j+1)})_{r}} ~&\text{if}~ d_{j-1}=0,\\
   \end{cases}\\
&=\begin{cases}
\widetilde{\omega}(i_1,j) ((Y_{j+1,\xi(j+2)})_{r} )^{1-d_{j }} ~&\text{if}~ d_{j-1}=1, \\
 \widetilde{\omega}(i_1, (j-1)_\bullet+1 )((Y_{j+1,\xi(j+2)})_{r} )^{1-d_{j }} ~&\text{if}~ d_{j-1}=0.\\
 \end{cases}
\end{align*}

{\bf Case 8.} In the case where $j>i_1$, $\xi(j)=\xi(j+1)+1$, $t_j>0$, $s=1$, it follows from Figure \ref{arrows incident with 8} that	
\begin{gather}
 \begin{align*}	
[x_{j,1}] [x_{[i_1, j],1}] = &[x_{[i_1, j-1],0}] [x_{j,2}]^{(1-d_{j-1})+d_{j-1}p^{2}(t_{j-1})} + [x_{i_1-1, 1}]^{\delta_{1,t_j}d_{j-1}\delta_{{i_1}_\bullet, (j-1)_\bullet}} [x_{[i_1, Y],0}]^{(1-\delta_{i_1,X})d_{X-1}q^{0}(t_Y)}  \\
&[x_{X,-1}]^{ (1-\delta_{i_1,X})d_{X-1} p^{0}(t_{X})+\delta_{i_1, X} }
[x_{[i_1, X],1}]^{ d_{X-1}d_{j-1}(\delta_{1,t_X}+p^{2}(t_{X})\delta_{1,t_{j}}) +(1-d_{X-1})(1- \delta_{i_1, X})d_{j-1}\delta_{1,t_j}} \\ 
&[x_{ X,2 }]^{(1-\delta_{X, i_1})d_{X-1}d_{j-1} \delta_{1,t_j} q^{1}(t_X) + \delta_{X, i_1}d_{j-1}\delta_{1,t_j}} 
 [x_{[i_1, j-1],1}]^{(1-d_{j-1}) p^{1}(t_{j-1})} [x_{[i_1, j-1],2}]^{d_{j-1}p^2(t_{j-1})} \\
 &
 [x_{j-1,2}]^{(1-d_{j-1})((1-d_{ i_1})\delta_{j-1, i_1}+(1-\delta_{j-1, i_1}) d_{j-2} q^1(t_{j-1}))}  [x_{j+1,1}],
\end{align*}
\end{gather}
where $X=\max\{x \mid i_1 \leq x<j, \xi(x)=\xi(x+1)+1\}, Y=X-1$.
Since 
 \begin{align*}
 x_{j,1}  \nmid  & (x_{[i_1, j-2],0})^{(1-\delta_{i_1,j-1})d_{j-2}q^{0}(t_{j-2})}  
 (x_{j-1,-1})^{(1-\delta_{i_1,  j-1})d_{j-2}p^{0}(t_{j-1}) +\delta_{i_1,  j-1}}   (x_{[i_1, j-1],1})^{ p^{1}(t_{j-1})} \\
 & (x_{j-1,2})^{ (1-d_{ i_1})\delta_{ i_1,j-1}+(1-\delta_{ i_1, j-1}) d_{j-2} q^1(t_{j-1}) }  x_{j+1,1}
\end{align*}
if $d_{j-1}=0$, and 
\begin{gather}
 \begin{align*}	
x_{[i_1, j-1],0} (x_{j,2})^{ p^{2}(t_{j-1})} = &(A_{(j-1)_\bullet+1,\xi((j-1)_{\bullet}+1)-2})_{r+1}(A_{(j-1)_\bullet+2,\xi((j-1)_{\bullet}+2)-2})_{r+1}\dots
 (A_{j,\xi(j)-2})_{r+1} (x_{i_1-1, 1})^{\delta_{1,t_j} \delta_{{i_1}_\bullet, (j-1)_\bullet}} \\
 &(x_{[i_1, Y],0})^{(1-\delta_{i_1,X})d_{X-1}q^{0}(t_Y)}   (x_{X,-1})^{(1-\delta_{i_1,X}) d_{X-1}p^{0}(t_{X}) +\delta_{i_1, X}}(x_{[i_1, X],1})^{ d_{X-1} (\delta_{1,t_X}+p^{2}(t_{X})\delta_{1,t_{j}}) +(1-d_{X-1})(1- \delta_{i_1, X}) \delta_{1,t_j}}   \\
&(x_{ X,2 })^{(1-\delta_{X, i_1})d_{X-1}  \delta_{1,t_j} q^{1}(t_X) + \delta_{X, i_1} \delta_{1,t_j}} (x_{[i_1, j-1],2})^{ p^2(t_{j-1})} x_{j+1,1} 
\end{align*}
\end{gather}
if $d_{j-1}=1$ by induction, we have
\[
x_{j,1} x_{[i_1, j],1} = x_{[i_1, j-1],0} (x_{j,2})^{(1-d_{j-1})+d_{j-1}p^{2}(t_{j-1})}	
\]
by Proposition \ref{lemma:equations satisfy by q-character}. Hence
\begin{gather}	
 \begin{align*}
  x_{[i_1, j],1} &= \frac{ x_{[i_1, j-1],0} (x_{j,2})^{(1-d_{j-1})+d_{j-1}p^{2}(t_{j-1})}} { x_{j,1}} \\
&  =\begin{cases}
  \frac{ \widetilde{\omega}(i_1, (j-1)_\bullet+1) (Y_{j,\xi(j)+1})_{r,1} ((Y_{j, \xi(j)+1})_{r,2})^{p^{2}(t_{j-1})} } {  (Y_{j,\xi(j)+1})_{r,1}} ~&\text{if}~d_{j-1}=1,\\
 \frac{  \widetilde{\omega}(i_1, (j-1)_\bullet+1) (Y_{j,\xi(j)-1})_{r}  (Y_{j, \xi(j)+1})_{r,2} } {  (Y_{j,\xi(j)+1})_{r,1}}   ~&\text{if}~d_{j-1}=0,
  \end{cases}\\
 &=\begin{cases}
 \widetilde{\omega}(i_1, \widehat{j} ) ((Y_{j,\xi(j)+1})_{r,2})^{p^{2}(t_{j-1})}~&\text{if}~d_{j-1}=1,\\
 \widetilde{\omega}(i_1, \widehat{j} ) (Y_{j,\xi(j)-1})_{r,1}~&\text{if}~d_{j-1}=0.   
  \end{cases}
\end{align*}
\end{gather}

{\bf Case 9.} In the case where $j>i_1$, $\xi(j)=\xi(j+1)+1$, $t_j \geq s>1$, it follows from Figure \ref{arrows incident with 9} that
\begin{gather}	
 \begin{align*}
[x_{j, s }] [x_{[i_1, j], s } ]=& [x_{[i_1, j], s -1}] [x_{j, s +1}]^{(1-d_{j-1})+d_{j-1}p^{ s +1}(t_{j-1})} + [x_{i_1-1,  s }]^{\delta_{ s ,t_j}d_{j-1}\delta_{{i_1}_\bullet, (j-1)_\bullet}} [x_{[i_1, Y],0}]^{(1-\delta_{i_1,X})d_{X-1}q^{0}(t_Y)}  \\
&[x_{X,-1}]^{ (1-\delta_{i_1, X})d_{X-1}p^{0}(t_{X})+\delta_{i_1, X}} [x_{[i_1, X], \min\{t_X, t_{j}\}}]^{d_{j-1}d_{X-1} p^{\min\{t_X, t_{j-1}\}}( s )p^{1}(t_X)}\\
 & [x_{[i_1, j-1],t_{j-1}}]^{(1-d_{j-1})d_{j-2}p^{1}(t_{j-1})p^{t_{j-1}}( s -1)}  [x_{[i_1, j-1], s }]^{(1-d_{j-1})q^{t_{j-1}}( s )} [x_{[i_1, X],t_{j}}]^{d_{j-1}(1-\delta_{X, i_1})(1-d_{X-1})\delta_{ s ,t_j}}   \\
&
[ x_{X,t_{j}+1} ]^{(1-\delta_{X, i_1})d_{X-1} d_{j-1}\delta_{ s ,t_{j}}p^{t_X}(t_{j})+\delta_{X, i_1}d_{j-1}\delta_{ s ,t_j}} [x_{[i_1, j-1], s +1}]^{d_{j-1}(1-\delta_{ s ,t_{j}})} \\
&
 [x_{j-1, s +1}]^{(1-d_{j-1})((1-d_{i_1})\delta_{i_1, j-1}+(1-\delta_{i_1, j-1})d_{j-2}p^{t_{j-1}}( s ))} [x_{j+1, s }],
\end{align*}
\end{gather}
where $X=\max\{x \mid i_1 \leq x<j, \xi(x)=\xi(x+1)+1\}, Y=X-1$. 
Since 
\begin{gather}
\begin{align*}
x_{j, s } \nmid  &  (x_{[i_1, j-2],0})^{(1-\delta_{i_1,j-1})d_{j-2}q^{0}(t_{j-2})}  (x_{j-1,-1})^{ (1-\delta_{i_1, j-1})d_{j-2} p^{0}(t_{j-1})+\delta_{i_1, j-1} } (x_{[i_1, j-1], s })^{q^{t_{j-1}}( s )}\\
&(x_{[i_1, j-1],t_{j-1}})^{ d_{j-2}p^{1}(t_{j-1})p^{t_{j-1}}( s -1)}    (x_{j-1, s +1})^{ (1-d_{i_1})\delta_{i_1, j-1}+(1-\delta_{i_1, j-1})d_{j-2}p^{t_{j-1}}( s )} x_{j+1, s }
\end{align*}
\end{gather}
if $d_{j-1}=0$, and 
\begin{gather}
\begin{align*}
x_{[i_1, j], s -1} (x_{j, s +1})^{p^{ s +1}(t_{j-1})} = &  (A_{(j-1)_\bullet+1,\xi((j-1)_{\bullet}+1)-2 s })_{ r+ s}(A_{(j-1)_\bullet+2,\xi(j_{\bullet}+2)-2 s  })_{r+ s }\dots 
(A_{j ,\xi(j )-2 s })_{r+ s }\\ 
&  (x_{i_1-1,  s })^{\delta_{ s ,t_j} \delta_{{i_1}_\bullet, (j-1)_\bullet}} (x_{[i_1, Y],0})^{(1-\delta_{i_1,X})d_{X-1}q^{0}(t_Y)}  (x_{X,-1})^{(1-\delta_{i_1, X})d_{X-1}p^{0}(t_{X})+\delta_{i_1, X}} \\
&(x_{[i_1, X],\min\{t_X, t_{j}\}})^{ d_{X-1} p^{\min\{t_X, t_{j-1}\}}( s )p^{1}(t_X)} 
 (x_{[i_1, X],t_{j}})^{(1-\delta_{X, i_1})(1-d_{X-1})\delta_{ s ,t_j}} \\
&( x_{X,t_{j}+1} )^{(1-\delta_{X, i_1})d_{X-1} \delta_{ s ,t_{j}}p^{t_X}(t_{j})+\delta_{X, i_1}\delta_{ s ,t_j}} 
(x_{[i_1, j-1], s +1})^{1-\delta_{ s ,t_{j}}} 
 x_{j+1, s } 
\end{align*}
\end{gather}
if $d_{j-1}=1$ by induction, we have
\[
x_{j, s } x_{[i_1, j], s } = x_{[i_1, j], s -1} (x_{j, s +1})^{(1-d_{j-1})+d_{j-1}p^{ s +1}(t_{j-1})}	
\]
by Proposition \ref{lemma:equations satisfy by q-character}. Hence
\begin{align*}
x_{[i_1, j], s }& = \frac{ x_{[i_1, j], s -1} (x_{j, s +1})^{(1-d_{j-1})+d_{j-1}p^{ s +1}(t_{j-1})}} {x_{j, s } }	   \\
&=\begin{cases} \frac{  \widetilde{\omega}(i_1, \widehat{j})  (Y_{j ,\xi(j )+1})_{r, s}   ((Y_{j,\xi(j)+1})_{r, s +1} )^{p^{ s +1}(t_{j-1})}}{ (Y_{j ,\xi(j )+1})_{r, s  } }   ~&\text{if}~d_{j-1}=1,\\
 \frac{ \widetilde{\omega}(i_1,  \widehat{j} )  (Y_{j,\xi(j)-1})_{r, s-1 } (Y_{j,\xi(j)+1})_{r, s +1} }{ (Y_{j,\xi(j)+1})_{r, s } }  ~&\text{if}~d_{j-1}=0,
\end{cases}\\
&=\begin{cases}  \widetilde{\omega}(i_1, \widehat{j} )   ((Y_{j ,\xi(j )+1})_{r, s +1})^{p^{ s +1}(t_{j-1})}  ~&\text{if}~d_{j-1}=1,\\
  \widetilde{\omega}(i_1, \widehat{j})  (Y_{j,\xi(j)-1})_{r, s }~&\text{if}~d_{j-1}=0.
\end{cases}
\end{align*}
\end{proof}

Given a generalized HL-module, we can always construct a mutation sequence as shown in Section \ref{equivalence classes of generalized HL-modules as cluster variables} such that, applying this mutation sequence to the initial seed of $\seed$, $[L(x_{[i_1, i_k],0})]$ (respectively, $[L(x_{[i_1,i_k+1],r_k})]$) is the desired generalized HL-module if $a_{k-1} < a_k$, or $a_{k-1} > a_k$ and $r_k = 0$, or $n=i_k$ (respectively, $i_k < n$, $a_{k-1} > a_k$, and $r_k \neq 0$).

Finally, reality and primeness of generalized HL-modules follow from Lemma \ref{theorem of real prime} and Proposition \ref{pro:cluster variables}.

\section{Appendix}\label{appendix}

In the appendix, we prove that, under the assumption of Case 10 in Section \ref{Proof of Theorem on generalized HL-modules}, all arrows incident to $(j, r+s)$ in  $Q_\xi^{s}[i_1, j]$ are as shown in Figure \ref{arrows incident with 9}. 

We prove it by induction on $j-i_1$. 

\subsection{For $j=i_1+1$}

We have two cases to consider.

{\bf Case 1.} In the case where $j=i_1+1$, $\xi(i_1+1)=\xi(i_1+2)+1$, $t_{i_1+1}\geq 2$ and $d_{i_1}=0$, we have $t_{i_1}=0$. 

We prove it by induction on $s-1$. It follows from Figure \ref{arrows incident with 8} that the arrows incident to $(i_1+1, r+1)$ and $(i_1+1, r+2)$ in $\Q ^{1}[i_1, i_1+1]$  are shown in the left hand side of Figure \ref{The proof of lemmas fig: 15}. After mutating $(i_1+1, r+1)$ in $\Q ^{1}[i_1, i_1+1]$, the arrows incident to $(i_1+1, r+2)$ in $\Q ^{2}[i_1, i_1+1]$ are as required, see Figure \ref{The proof of lemmas fig: 15}. 
\begin{figure}\centering
\resizebox{0.7\textwidth}{!}{ 
\xymatrix{
 (i_1, r-1) & &    & &(i_1, r-1)  & \\
(i_1,r)\ar[rd]& &  &   & & &  & \\ 
&(i_1+1, r+1)\ar[luu] \ar[r] \ar[ld] &(i_1+2, r+1) \ar[ld]      && & (i_1+1, r+1) \ar[d]& 	\\	
 (i_1, r+2)\ar[r]&(i_1+1, r+2)\ar[ld]\ar[r]\ar[u]& (i_1+2, r+2) & && (i_1+1, r+2)\ar@/^/[luuu] \ar[r] \ar[ld]& (i_1+2, r+2)\\
 (i_1, r+3)& (i_1+1, r+ 3)\ar[u]& \ar[r]^{ \mu_{(i_1+1, r+1)}} && (i_1, r+3)& (i_1+1, r+3) \ar[u]& 
}}
\caption{Arrows incident to $ (i_1+1, r+1)$ and $(i_1+1, r+2)$ in $Q_{\xi}^{1}[i_1, i_1+1]$ (left), and arrows incident to $(i_1+1, r+2)$ in $Q_{\xi}^{2}[i_1, i_1+1]$ (right).}\label{The proof of lemmas fig: 15}
\end{figure}
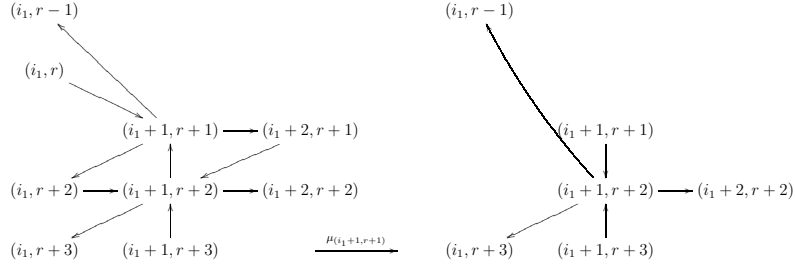	

Suppose that $2\leq s\leq t_{i_1+1}-1$, and the arrows incident to $(i_1+1, r+s)$ in $\Q^{s}[i_1, i_1+1]$ are the required arrows. By induction, we need to prove that the result holds for the arrows incident to $(i_1+1, r+s+1)$ in $\Q ^{s+1}[i_1, i_1+1]$. By induction the arrows incident to $(i_1+1,r+s)$ and $(i_1+1, r+s+1)$ in $\Q^{s}[i_1, i_1+1]$ are shown in the left hand side of Figure \ref{The proof of lemmas fig: 16}. After mutating $(i_1+1, r+s)$ in $\Q ^{s}[i_1, i_1+1]$, the arrows incident to $(i_1+1, r+s+1)$ in $\Q ^{s+1}[i_1, i_1+1]$ are as required, see Figure \ref{The proof of lemmas fig: 16}.
\begin{figure}\centering
\resizebox{\textwidth}{!}{
\xymatrix{
(i_1, r-1)& &  &    &(i_1, r-1)& &  & \\ 
& (i_1 +1, r+s-1) \ar[d]&  &  && &  & \\ 
 &(i_1 +1, r+s)\ar[luu] \ar[r] \ar[ld] &(i_1+2, r+s) \ar[ld]      && & (i_1 +1, r+s) \ar[d]&	\\	
   (i_1, r+s+1)\ar[r]&(i_1 +1, r+s+1)\ar[ld]\ar[r]\ar[u]& (i_1+2, r+s+1)  & && (i_1 +1, r+s+1) \ar[luuu] \ar[r] \ar[ld]& (i_1+2, r+s+1)\\
 (i_1, r+s+2)& (i_1 +1, r+s+2)\ar[u]& \ar[r]^{ \mu_{(i_1 +1, r+s)}} && (i_1, r+s+2)& (i_1 +1, r+s+2) \ar[u]& 
 }}
\caption{Arrows incident to $ (i_1 +1, r+s)$ and $(i_1 +1, r+s+1)$ in $Q_{\xi}^{s}[i_1, i_1+1]$ (left), and arrows incident to $(i_1 +1, r+s+1)$ in $Q_{\xi}^{s+1}[i_1, i_1+1]$ (right).}\label{The proof of lemmas fig: 16}
\end{figure}
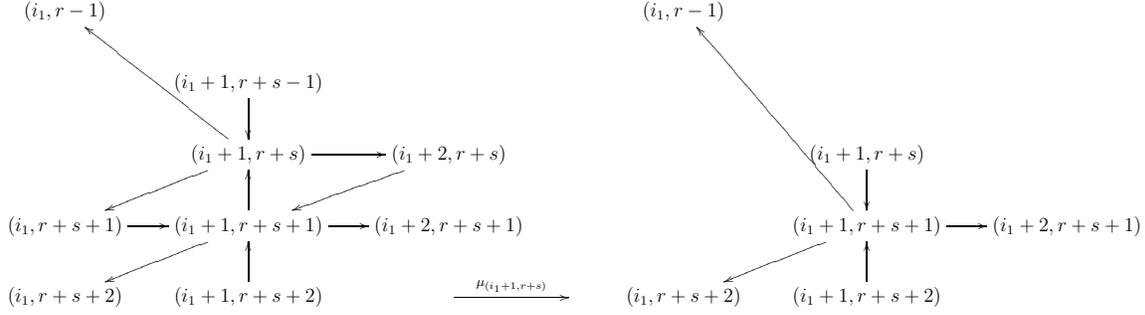		

{\bf Case 2.} In the case where $j=i_1+1$, $\xi(i_1+1)=\xi(i_1+2)+1$, $d_{i_1}=1$ and $t_{i_1+1}\geq 2$, we have $t_{i_1}=t_{i_1+1}$.

We first give the effects of the mutations of the vertices in the $i_1$th column on the arrows incident to the vertices in the $(i_1+1)$th column, as shown in Figures \ref{The proof of lemmas fig: 1} and \ref{The proof of lemmas fig: 2}. From now on, in our figures, the red arrows denote the arrows that are added during this mutation process and the red dashed arrows denote the arrows that are deleted during this mutation process.
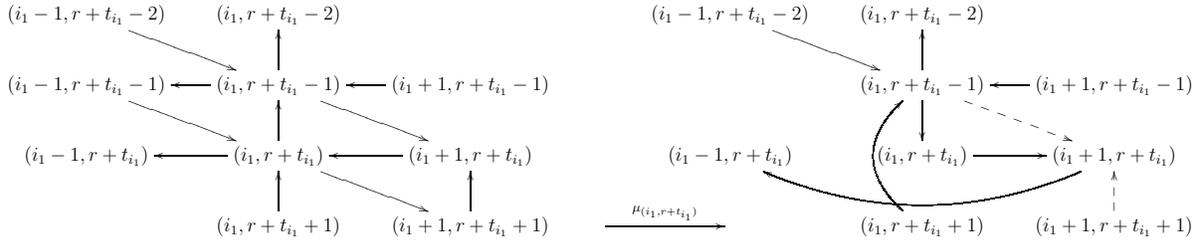
\begin{figure}
\centering
\resizebox{\textwidth}{!}{
\xymatrix{
(i_1-1, r+t_{i_1}-2) \ar[dr] &(i_1, r+t_{i_1}-2)   &  &&(i_1-1, r+t_{i_1}-2) \ar[dr] &   (i_1, r+t_{i_1}-2)&\\
(i_1-1, r+t_{i_1}-1) \ar[dr]&(i_1, r+t_{i_1}-1) \ar[u]\ar[l] \ar[rd]  & (i_1+1, r+t_{i_1}-1)\ar[l] && & (i_1, r+t_{i_1}-1)\ar@[red]@{-->}[dr] \ar[u] \ar[d]& (i_1+1,r+ t_{i_1}-1)\ar[l]\\
(i_1-1, r+t_{i_1})	&(i_1, r+t_{i_1})  \ar[u]\ar[l]\ar[rd] & (i_1+1, r+t_{i_1})\ar[l]   &&(i_1-1, r+t_{i_1})&  (i_1, r+t_{i_1}) \ar@[red][r]& (i_1+1, r+t_{i_1}) \ar@[red]@/^2.5pc/[ll]  \\
&(i_1, r+t_{i_1}+1)  \ar[u] & (i_1+1, r+t_{i_1}+1)\ar[u] & \ar[r]^{ \mu_{(i_1, r+t_{i_1})}}&& (i_1, r+t_{i_1}+1) \ar@/^2.5pc/[uu] & (i_1+1, r+t_{i_1}+1)\ar@[red]@{-->}[u]
} }
\caption{
The effect of the mutation at $(i_1, r+t_{i_1})$ on the arrows incident to $(i_1+1, r+t_{i_1})$ in $Q_{\xi}^{t_{i_1}}[i_1, i_1]$.
}\label{The proof of lemmas fig: 1}
\end{figure}

\begin{figure}
\centering
\resizebox{\textwidth}{!}{
\xymatrix{
(i_1-1, r+s-2) \ar[rd]&(i_1, r+s-2) &   &(i_1-1, r+s-2)\ar[rd]&(i_1,r+ s-2)&\\
(i_1-1,r+s-1) \ar[rd] & (i_1, r+s-1)\ar[l] \ar[u]\ar[rd]&(i_1+1, r+s-1)\ar[l]&    & (i_1, r+s-1) \ar[d]\ar[u] \ar@[red]@{-->}[rd] & (i_1+1, r+s-1)\ar[l]\\
& (i_1, r+s)\ar[u]\ar[d] &  (i_1+1, r+s)\ar[l] & &(i_1, r+s) \ar@[red][r]& (i_1+1, r+s)  \ar@[red][ld] \\
&(i_1, r+s+1) &\ar[r]^{\mu_{ (i_1, r+s)} } &&(i_1, r+s+1) &\\
& (i_1, r+t_{i_1}+1)\ar@/^3pc/[uu]& & &(i_1, r+t_{i_1}+1)  \ar@/^3.5pc/[uuu]&
} }
\caption{
The effect of the mutation at $(i_1, r+s)$ on the arrows incident to $(i_1+1, r+s)$ in $Q_{\xi}^{s}[i_1, i_1]$.
}\label{The proof of lemmas fig: 2}
\end{figure}
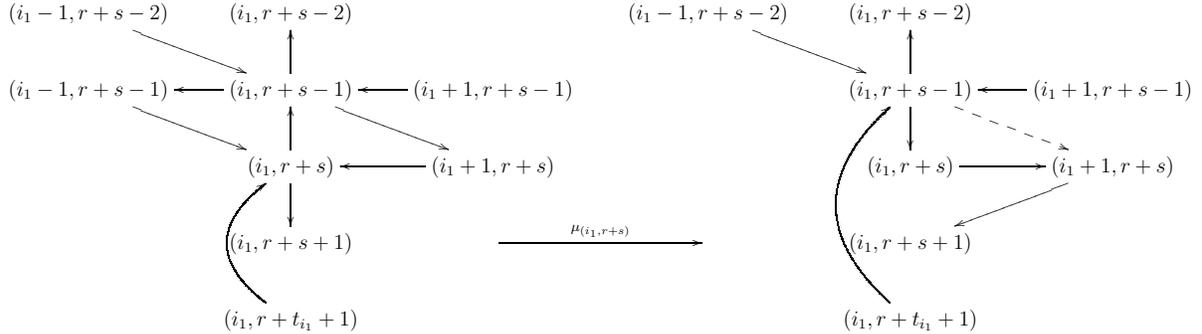	

We prove it by induction on $s-1$. It follows from Figures \ref{arrows incident with 8} and \ref{The proof of lemmas fig: 2} that the arrows incident to $(i_1+1, r+1)$ and $(i_1+1, r+2)$ in $\Q ^{1}[i_1, i_1+1]$  are shown in the left hand side of Figure \ref{The proof of lemmas fig: 11}. After mutating $(i_1+1, r+1)$ in $\Q ^{1}[i_1, i_1+1]$, the arrows incident to $(i_1+1, r+2)$ in $\Q ^{2}[i_1, i_1+1]$ are as required, see Figure \ref{The proof of lemmas fig: 11}.
 \begin{figure}\centering
\resizebox{\textwidth}{!}{
\xymatrix{ 
&(i_1, r)\ar[rd] & &  &  & & &  & \\ 
 &&(i_1+1, r+1) \ar[r] \ar[ld] &(i_1+2, r+1) \ar[ld]      && & (i_1+1, r+1) \ar[d]&	\\	
 (i_1-1, r+t_{i_1})& (i_1, r+2)\ar[r]&(i_1+1, r+2)\ar@/^1pc/[ll]^{\delta_{2,t_{i_1}}}\ar[ld]^{1-\delta_{2,t_{i_1}}}\ar[r]\ar[u]& (i_1+2, r+2)  & (i_1-1, r+t_{i_1}) && (i_1+1, r+2)   \ar[r] \ar@/^1pc/[ll]^{\delta_{2,t_{i_1}}}\ar[ld]^{1-\delta_{2,t_{i_1}}}& (i_1+2, r+2)\\
& (i_1, r+3)& (i_1+1, r+3)\ar[u]_{1-\delta_{2,t_{i_1}}}& \ar[r]^{ \mu_{(i_1+1, r+1)}} && (i_1, r+3)& (i_1+1, r+3) \ar[u]_{1-\delta_{2,t_{i_1}}}& 
}}
\caption{Arrows incident to $ (i_1+1, r+1)$ and $(i_1+1, r+2)$ in $Q_{\xi}^{1}[i_1, i_1+1]$ (left), and arrows incident to $(i_1+1, r+2)$ in $Q_{\xi}^{2}[i_1, i_1+1]$ (right).}\label{The proof of lemmas fig: 11}
\end{figure}
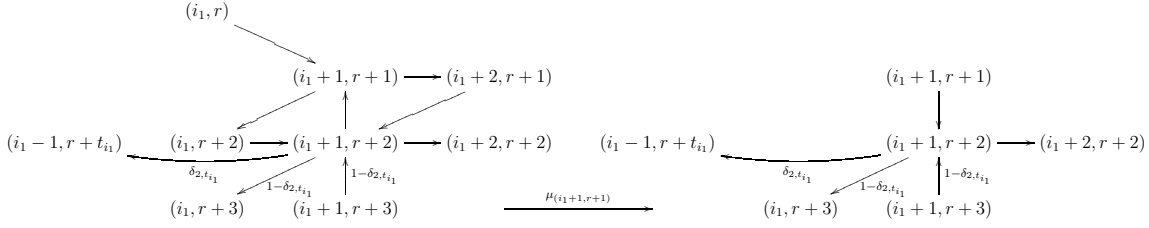

Suppose that $2\leq s\leq t_{i_1+1}-1$, and the arrows incident to $(i_1+1, r+s)$ in $\Q^{s}[i_1, i_1+1]$ are the required arrows. By induction, we need to prove that the result holds for the arrows incident to $(i_1+1, r+s+1)$ in $\Q ^{s+1}[i_1, i_1+1]$. By induction and Figures \ref{The proof of lemmas fig: 1}, \ref{The proof of lemmas fig: 2}, the arrows incident to $(i_1+1, r+s)$ and $(i_1+1, r+s+1)$ in $\Q^{s}[i_1, i_1+1]$ are shown in the left hand side of Figure \ref{The proof of lemmas fig: 12}. After mutating $(i_1+1, r+s)$ in $\Q ^{s}[i_1, i_1+1]$, the arrows incident to $(i_1+1, r+s+1)$ in $\Q ^{s+1}[i_1, i_1+1]$ are as required, see Figure \ref{The proof of lemmas fig: 12}.
\begin{figure}\centering
\resizebox{\textwidth}{!}{
\xymatrix{
&& (i_1 +1, r+s-1) \ar[d]&  &  && &  & \\ 
& &(i_1 +1, r+s) \ar[r] \ar[ld] &(i_1+2, r+s) \ar[ld]      && & (i_1 +1, r+s) \ar[d]&	\\
(i_1-1, r+ t_{i_1}) &  (i_1, r+s+1)\ar[r]&(i_1 +1, r+s+1) \ar@/^1pc/[ll]^{\delta_{s+1,t_{i_1}}}
\ar[ld]^{1-\delta_{s+1,t_{i_1}}}\ar[r]\ar[u]& (i_1+2, r+s+1)  & (i_1-1, r+t_{i_1}) && (i_1 +1, r+s+1)  \ar@/^1pc/[ll]^{\delta_{s+1,t_{i_1}}} \ar[r] \ar[ld]^{1-\delta_{s+1,t_{i_1}}} & (i_1+2, r+s+1)\\
&  (i_1, r+s+2)& (i_1 +1, r+s+2)\ar[u]_{1-\delta_{s+1,t_{i_1}}}& \ar[r]^{ \mu_{(i_1 +1, r+s)}} && (i_1, r+s+2)& (i_1+1, r+s+2) \ar[u]_{1-\delta_{s+1,t_{i_1}}}& 
}}
\caption{Arrows incident to $ (i_1 +1, r+s)$ and $(i_1 +1, r+s+1)$ in $Q_{\xi}^{s}[i_1, i_1+1]$ (left), and arrows incident to $(i_1 +1, r+s+1)$ in $Q_{\xi}^{s+1}[i_1, i_1+1]$ (right).}\label{The proof of lemmas fig: 12}
\end{figure}
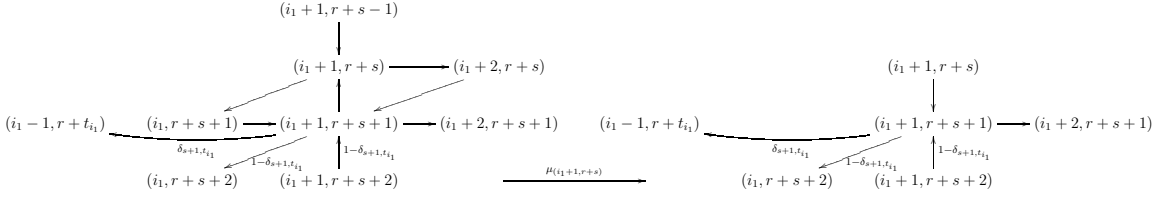

\subsection{Inductive step}	

Suppose that the arrows incident to $(j, r+s)$ in $Q_{\xi}^{s}[i_1, j]$ are the required arrows. By induction, we need to prove that the result holds for the arrows incident to $(j+1, r+s)$ in $Q_{\xi}^{s}[i_1, j+1]$. We have $3$ cases to consider.

{\bf Case 1.} In the case where $\xi(j+1)=\xi(j+2)+1$, $t_{j+1}\geq 2$, $d_j=1$, we have $t_{j+1}=t_j$. 

We first give the effects of the mutations of the vertices in the $j$th column on the arrows incident to the vertices in the $(j+1)$th column, as shown in Figures \ref{The proof of lemmas fig: 19} and \ref{The proof of lemmas fig: 20}.
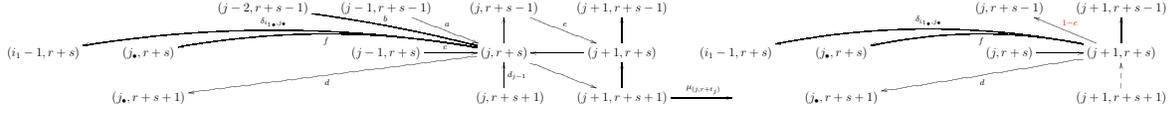
\begin{figure}\centering
\resizebox{\textwidth}{!}{
\xymatrix{
&&(j-2,  r+s-1)\ar@/^0.4pc/[rrd]^{b} & (j-1,  r+s-1) \ar[rd]^{a} &(j, r+s-1) \ar[rd]^{e}  & (j+1, r+s-1)&&&&& (j, r+s-1) &(j+1, r+s-1)&\\
(i_1-1, r+s)&(j_\bullet, r+s) & & (j-1, r+s) \ar[r]^{c}  & (j, r+s)  \ar[u]  \ar[dr] \ar@/^-1.8pc/[llll]_{\delta_{{i_1}_\bullet, j_{\bullet}}} \ar@/^-1.5pc/[lll]^{f} \ar[dlll]^{d}&(j+1, r+s)\ar[u] \ar[l]&(i_1-1, r+s)&(j_\bullet, r+s) &&&  (j, r+s)\ar@[red][r]&(j+1, r+s) \ar@[red][lu]_{\textcolor{red}{1-e}} \ar[u] \ar@[red][lllld]^{d} \ar@[red]@/^-1.5pc/[llll]^{f}\ar@[red]@/^-1.9pc/[lllll]_{\delta_{{i_1}_\bullet, j_{\bullet}}}& \\
&(j_\bullet, r+s+1) && &\quad (j, r+ s+1) \ar[u]_{d_{j-1}}  & (j+1, r+s+1) \ar[u]\ar[r]^{\qquad \qquad\mu_{(j, r+t_j)}} &&(j_\bullet, r+s+1)& &&&(j+1, r+s+1)\ar@[red]@{-->}[u]& 
}}
\caption{
The effect of the mutation at $(j, r+s)$ on the arrows incident to $(j+1, r+s)$ in $Q_{\xi}^{s}[i_1, j]$, where $s=t_j$, $a=d_{j-1}q^{t_{j-1}}(s)(1-\delta_{1,s})$, $b=d_{j-1} q^{t_{j-1}}(s)\delta_{1,s}$, $c=d_{j-1}p^{t_{j-1}}(s-1)+(1-d_{j-1})$, $d=(1-\delta_{i_1, j_{\bullet}}-\delta_{{i_1}_\bullet, j_{\bullet}})p^{t_{j_\bullet}}(s)  d_{j_\bullet-1}+\delta_{i_1, j_{\bullet}}$, $ e=\delta_{s,1}(1-d_j)+(1-\delta_{s,1})$, $f=(1-\delta_{{i_1}_\bullet, j_{\bullet}})q^{t_{j_\bullet}}(s)$. 	
}\label{The proof of lemmas fig: 19}
\end{figure}
					
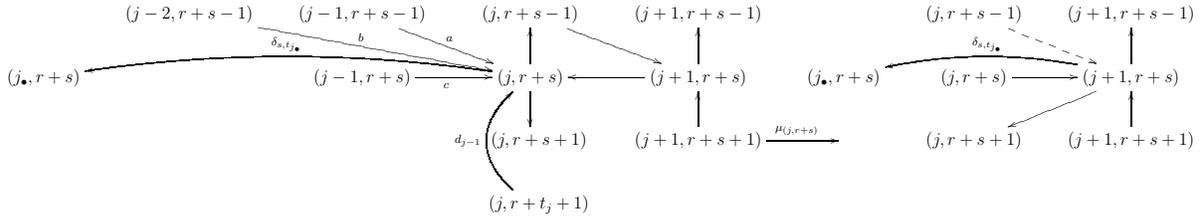
\begin{figure}\centering
\resizebox{\textwidth}{!}{
\xymatrix{
&(j-2,  r+s-1) \ar[rrd]^{b}& (j-1, r+s-1) \ar[rd]^{a} &(j, r+s-1) \ar[dr]  &  (j+1, r+s-1)& &(j, r+s-1)\ar@[red]@{-->}[rd]&(j+1, r+s-1)\\
(j_\bullet, r+s) & & (j-1, r+s) \ar[r]_{c}& (j, r+s) \ar[u]  \ar[d]  \ar@/^-1.2pc/[lll]_{\delta_{s, t_{j_\bullet}}}&(j+1, r+s) \ar[l]\ar[u]&(j_\bullet, r+s) & (j, r+s)\ar@[red][r] &(j+1, r+s)\ar@[red]@/^-1.2pc/[ll]_{\delta_{s, t_{j_\bullet}}}\ar@[red][ld]\ar[u]\\
&&  &\quad (j, r+s+1)   & (j+1, r+s+1) \ar[r]^{\quad\qquad \mu_{(j, r+s)}}  \ar[u]&&(j, r+s+1)&(j+1, r+s+1)\ar[u]\\
&&  &\quad (j, r+t_{j}+1) \ar@/^2.5pc/[uu]^{d_{j-1}}&  &&
}}
\caption{ 
The effect of the mutation at $(j, r+s)$ on the arrows incident to $(j+1, r+s)$ in $Q_{\xi}^{s}[i_1, j]$, where $1\leq s\leq t_j-1$, $a=d_{j-1} q^{t_{j-1}}(s) (1-\delta_{1,s})$, $b=q^{t_{j-1}}(s)d_{j-1}\delta_{1,s}$, $c=d_{j-1}p^{t_{j-1}}(s-1)+(1-d_{j-1})$. }\label{The proof of lemmas fig: 20}
\end{figure}		

We prove it by induction on $s-1$. It follows from Figures \ref{arrows incident with 8}, \ref{The proof of lemmas fig: 19} and \ref{The proof of lemmas fig: 20} that the arrows incident to $(j+1,r+1)$ and $(j+1, r+2)$ in $\Q^{1}[i_1, j+1]$ are shown in the left hand side of Figure \ref{The proof of lemmas fig: 27}.  After mutating $(j+1,r+1)$ in $\Q^{1}[i_1, j+1]$, the arrows incident to $(j+1, r+2)$ in $\Q ^{2}[i_1, j+1]$ are as required, see Figure \ref{The proof of lemmas fig: 27}, where $a=(1-\delta_{{i_1}_{\bullet},j_{\bullet}}-\delta_{i_1, j_{\bullet}})d_{j_{\bullet}-1}p^{0}(t_{j_\bullet})+\delta_{i_1, j_{\bullet}}$, $b=(1-\delta_{i_1, j_{\bullet}}-\delta_{{i_1}_{\bullet},j_{\bullet}})q^{0}(t_{j_\bullet})$, $c=(1-\delta_{2,t_j})\delta_{2,t_{j_{\bullet}}}+\delta_{2,t_j}(1-\delta_{{i_1}_\bullet, j_{\bullet}})q^{t_{j_\bullet}}(2)$, $d=(1-\delta_{{i_1}_\bullet, j_{\bullet}}-\delta_{i_1, j_{\bullet}})p^{t_{j_{\bullet}}}(t_j)d_{j_\bullet-1}\delta_{2,t_j}+\delta_{2,t_j}\delta_{i_1, j_{\bullet}}$, $e=1-\delta_{2,t_j}$, $f=p^3(t_j)$, $g=\delta_{2,t_j}\delta_{{i_1}_\bullet, j_{\bullet}}$.
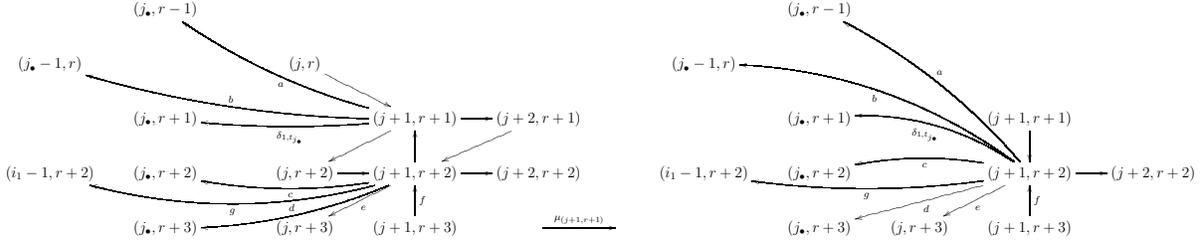
\begin{figure}\centering
\resizebox{\textwidth}{!}{
\xymatrix{
&(j_{\bullet}, r-1)  &   & & & && &(j_{\bullet}, r-1)&  & & \\
(j_{\bullet}-1, r) && & (j, r)\ar[rd]& & & &(j_{\bullet}-1, r) & &  & &  \\
&(j_{\bullet}, r+1)& &    & (j+1,r+1) \ar@/^0.5pc/[lll]^{\delta_{1,t_{j_{\bullet}}}} \ar[r]   \ar[ld]\ar@/^1pc/[llllu]_{b} \ar@/^1pc/[llluu]^{a}& (j+2,r+1)  \ar[ld]& &&(j_{\bullet}, r+1) &  &(j+1,r+1)\ar[d] &  \\
(i_1-1, r+2)	&(j_\bullet, r+2)	&   &    (j, r+2) \ar[r] & (j+1, r+2)\ar@/^1pc/[llld]_{d} \ar[ld]^{e}  \ar@/^2pc/[llll]^{g} \ar@/^1pc/[lll]^{c} \ar[u] \ar[r]& (j+2, r+2)  &&(i_1-1, r+2)&(j_\bullet, r+2) &  &(j+1, r+2)\ar[lld]^{d}\ar@/^1pc/[lll]^{g}  \ar@/_1pc/[ll]^{c}   \ar@/_1.5pc/[llu]^{\delta_{1,t_{j_{\bullet}}}} \ar@/_2pc/[llluu]^{b} \ar@/_1pc/[lluuu]_{a} \ar[r] \ar[ld]^{e} &(j+2, r+2)\\
&(j_\bullet, r+3)	&  &   (j, r+3)  &(j+1, r+3) \ar[u]_{f}&\ar[r]^{ \mu_{(j+1,r+1)}}  & && (j_\bullet, r+3) & (j, r+3) &  (j+1, r+3) \ar[u]_{f}
}}
\caption{ 
Arrows incident to $ (j+1,r+1)$ and $(j+1, r+2)$ in $Q_{\xi}^{1}[i_1, j+1]$ (left), and arrows incident to $(j+1, r+2)$ in $Q_{\xi}^{2}[i_1, j+1]$ (right).
 }\label{The proof of lemmas fig: 27}
\end{figure}	
					
Suppose that $2\leq  s\leq t_{j+1}-1$, and the arrows incident to $(j+1, r+s)$ in $\Q^{s}[i_1, j+1]$ are the required arrows. By induction, we need to prove that the result holds for the arrows incident to $(j+1, r+s+1)$ in $\Q ^{s+1}[i_1, j+1]$. By induction and Figures \ref{The proof of lemmas fig: 19}, \ref{The proof of lemmas fig: 20}, the arrows incident to $(j+1, r+s)$ and $(j+1, r+s+1)$ in $\Q^{s}[i_1, j+1]$ are shown in the left hand side of Figure \ref{The proof of lemmas fig: 28}, where $a=p^{t_{j_\bullet}}(s)p^{1}(t_{j_\bullet})$, $b=(1-\delta_{{i_1}_{\bullet},j_{\bullet}}-\delta_{i_1, j_{\bullet}})d_{j_{\bullet}-1}p^{0}(t_{j_\bullet})+\delta_{i_1, j_{\bullet}}$, $c=1-\delta_{s+1,t_{j}}$, $d=p^{s+2}(t_{j})$,  $e=(1-\delta_{s+1,t_j})\delta_{s+1,t_{j_{\bullet}}}+\delta_{s+1,t_j}(1-\delta_{{i_1}_\bullet, j_{\bullet}})q^{t_{j_\bullet}}(s+1)$, $f=(1-\delta_{{i_1}_\bullet, j_{\bullet}}-\delta_{i_1, j_{\bullet}})p^{t_{j_\bullet}}(t_j) d_{j_\bullet-1}\delta_{s+1,t_j}+\delta_{s+1,t_j}\delta_{i_1, j_{\bullet}}$, $g=\delta_{t_j, s+1}\delta_{{i_1}_\bullet, j_{\bullet}}$, $h=(1-\delta_{{i_1}_{\bullet},j_{\bullet}}-\delta_{i_1, j_{\bullet}})  q^0(t_{j_{\bullet}}) $. After mutating $(j+1, r+s)$ in $\Q^{s}[i_1, j+1]$, the arrows incident to $(j+1, r+s+1)$ in $\Q ^{s+1}[i_1, j+1]$ are as required, see Figure \ref{The proof of lemmas fig: 28}.					
\begin{figure}\centering
\resizebox{\textwidth}{!}{
\xymatrix{
& (j_{\bullet}, r-1) &   &  & & &(j_{\bullet}, r-1) &&\\
(j_{\bullet}-1, r)&   & & (j+1, r+s-1) \ar[d] & & (j_{\bullet}-1, r)&&&&\\
&  (j_\bullet, r+t_{j_{\bullet}})&  & (j+1, r+s) \ar[ld]\ar@/^-1.5pc/[ll]_{a}   \ar@/^-2pc/[lluu]_{b}  \ar@/^-1.2pc/[lllu]_{h}   \ar[r]   & (j+2, r+s) \ar[ld] && (j_\bullet, r+t_{j_{\bullet}}) &&(j+1, r+s)\ar[d]&\\
(i_1-1, r+s+1) & (j_{\bullet}, r+s+1) &  (j, r+s+1) \ar[r]& (j+1, r+s+1) \ar@/_-0.8pc/[lld]^{ f} \ar[r]\ar@/_-0.8pc/[ld]^{c} \ar[u]  \ar@/_-1.2pc/[ll]^{ e} \ar@/_-2.2pc/[lll]^{g} &  (j+2, r+s+1)   & (i_1-1, r+s+1) &(j_\bullet, r+s+1)   &  &(j+1, r+s+1) \ar[ll]_{ e} \ar@/^1pc/[lll]^{ g} \ar@/^0.5pc/[lld]^{ f}  \ar[lul]_{a}   \ar@/^-2pc/[lluuu]_{b}  \ar@/^-0.9pc/[llluu]_{h }  \ar@/^0.1pc/[ld]^{c} \ar[r]&(j+2, r+s+1) &\\
&(j_\bullet, r+t_j+1)& (j, r+s+2) &  (j+1, r+s+2) \ar[u]_{d}&\ar[r]^{ \mu_{(j+1,r+s)}} 	&&(j_\bullet, r+t_j+1) & (j, r+s+2) & (j+1, r+s+2)\ar[u] _{d}&
}}
\caption{ Arrows incident to $ (j+1, r+s)$ and $(j+1, r+s+1)$ in $Q_{\xi}^{s}[i_1, j+1]$ (left), and arrows incident to $(j+1, r+s+1)$ in $Q_{\xi}^{s+1}[i_1, j+1]$ (right).}\label{The proof of lemmas fig: 28}
\end{figure}
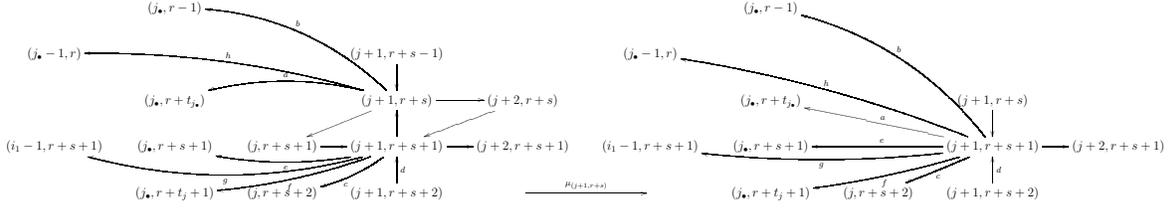						
						
{\bf Case 2.} In the case where $\xi(j+1)=\xi(j+2)+1$, $d_{j-1}=d_j=0$, $t_{j+1}\geq 2$, we have $t_j=t_{j+1}$.

We first give the effects of the mutations of the vertices in the $j$th column on the arrows incident to the vertices in the $(j+1)$th column, as shown in Figure \ref{The proof of lemmas fig: 31}, where $X=\max\{  x \mid i_1\leq x<j, \xi(x)=\xi(x+1)+1\}$, $Y=X-1$, $a=(1-d_{j-1})d_{j-2}p^{1}(t_{j-1})p^{t_{j-1}}(s-1)+ d_{j-1}d_{X-1}p^{\min\{t_X, t_{j-1}\}}(s)p^{1}(t_X)$, $b=(1-d_{j-1})q^{t_{j-1}}(s)$,  $c=(1-\delta_{X, i_1})d_{X-1}d_{j-1}\delta_{s, t_{j}}p^{t_X}(t_{j})+\delta_{X, i_1}d_{j-1}\delta_{s, t_{j}},$ $d=d_{j-1}(1-\delta_{s, t_{j}})+d_{j-2}(1-d_{j-1})p^{t_{j-1}}(s)$, $f=(1-d_{j-1})+d_{j-1}p^{s+1}(t_{j-1})$, $h=d_{j-1}(1-\delta_{i_1, X})(1-d_{X-1})\delta_{s, t_j}$.
\begin{figure}\centering
\resizebox{\textwidth}{!}{
\xymatrix{
& (X, r-1) & &  &  &  \\
(Y, r)	&  &   &  & &   \\
& (X, r+\min\{t_X,t_{j-1}\})  & &  &(j, r+s-1)\ar[d] &(j+1, r+s-1) & (j, r+s-1)\ar@[red][rd] &(j+1, r+s-1)\\
(i_1-1, r+t_{j})&  (X, r+t_{j})& (j-1, r+s) & & (j, r+s) \ar@/^-1pc/[lll]_{h}  \ar@/^-1.9pc/[llll]_{d_{j-1}\delta_{i_\bullet,(j-1)_\bullet}\delta_{s, t_j}}   \ar@/^-2pc/[llluuu]_{\qquad(1-\delta_{i_1, X})d_{X-1}p^{0}(t_{X})+\delta_{i_1, X}}   \ar@/^-1.5pc/[lllu]_{a} \ar@/^-1.7pc/[lllluu]_{d_{X-1}q^0(t_X) }  \ar[ll]_{b}\ar[r] \ar[lld]_{d}  \ar@/^-0.5pc/[llld]_{\quad }_{c} & (j+1, r+s) \ar[u]\ar[ld] \ar[r]^{\quad \mu_{(j, r+s)}} &(j, r+s) & (j+1, r+s) \ar[ld]_{1-f}\ar@[red][l] \ar[u]\\
&(X, r+t_{j}+1)  &  (j-1, r+s+1) & & (j, r+s+1)  \ar[u] _{f}&  (j+1, r+s+1) \ar[u]&(j, r+s+1)& (j+1, r+s+1)\ar[u]
}}
\caption{ 
The effect of the mutation at $(j, r+s)$ on the arrows incident to $(j+1, r+s)$ in $Q_{\xi}^{s}[i_1, j]$.
}\label{The proof of lemmas fig: 31}
\end{figure}
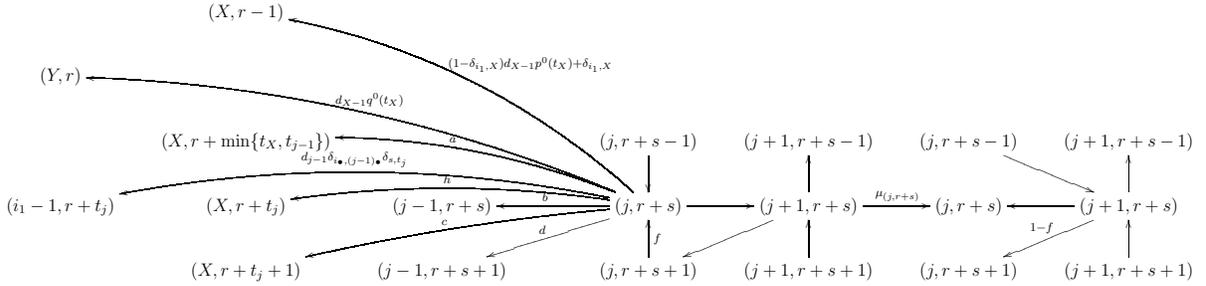	

We prove it by induction on $s-1$. It follows from Figures \ref{arrows incident with 8} and \ref{The proof of lemmas fig: 31} that the arrows incident to $(j+1,r+1)$ and $(j+1, r+2)$ in $\Q^{1}[i_1, j+1]$ are shown in the left image of Figure \ref{The proof of lemmas fig: 32}.  After mutating $(j+1,r+1)$ in $\Q^{1}[i_1, j+1]$, the arrows incident to $(j+1, r+2)$ in $\Q^{2}[i_1, j+1]$ are as required, see Figure \ref{The proof of lemmas fig: 32}.
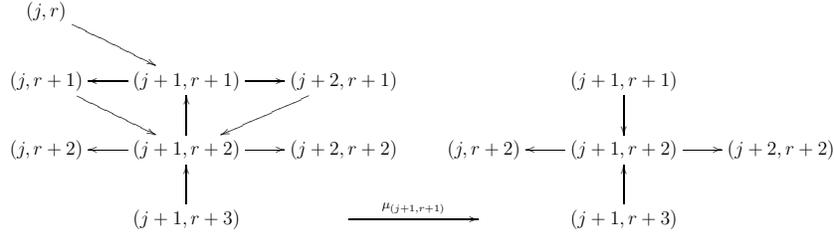
\begin{figure}
\centering
\resizebox{0.7\textwidth}{!}{ 
\xymatrix{
 (j, r) \ar[rd] & & \\
 ( j, r+1)\ar[rd] & (j+1, r+1)\ar[l]\ar[r] & (j+2, r+1)\ar[ld]& & (j+1,r+1)\ar[d] &  \\
 (j, r+2) & (j+1, r+2)\ar[u]\ar[l]\ar[r] &(j+2, r+2) 
 & (j, r+2) &  (j+1, r+2) \ar[l]\ar[r] &(j+2, r+2) \\
 & (j+1, r+3) \ar[u] & \ar[r]^{ \mu_{(j+1, r+1)}}&& (j+1, r+3)\ar[u]& 
  }}
\caption{Arrows incident to $(j+1, r+1)$ and $(j+1, r+2)$ in $Q_{\xi}^{1}[i_1, j+1]$ (left), and arrows incident to $(j+1, r+2)$ in $Q_{\xi}^{2}[i_1, j+1]$ (right).}
\label{The proof of lemmas fig: 32}
\end{figure}

Suppose that $2\leq  s\leq t_{j+1}-1$, and the arrows incident to $(j+1, r+s)$ in $\Q^{s}[i_1, j+1]$ are the required arrows. By induction, we need to prove that the result holds for the arrows incident to $(j+1, r+s+1)$ in $\Q ^{s+1}[i_1, j+1]$. By induction and Figure \ref{The proof of lemmas fig: 31}, the arrows incident to $(j+1, r+s)$ and $(j+1, r+s+1)$ in $\Q^{s}[i_1, j+1]$ are shown in the left hand side of Figure \ref{The proof of lemmas fig: 33}. After mutating $(j+1, r+s)$ in $\Q^{s}[i_1, j+1]$, the arrows incident to $(j+1, r+s+1)$ in $\Q ^{s+1}[i_1, j+1]$ are as required, see Figure \ref{The proof of lemmas fig: 33}.
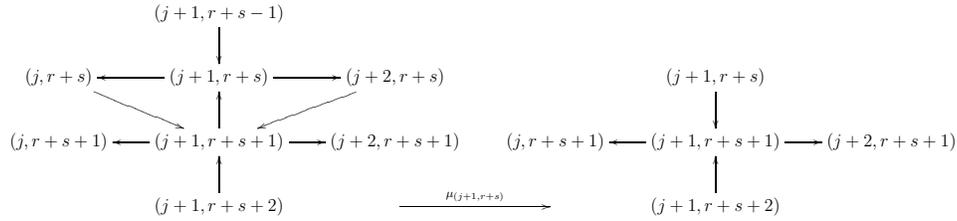
\begin{figure}
\centering
\resizebox{0.8\textwidth}{!}{ 
\xymatrix{
&(j+1, r+s-1) \ar[d]  & \\
(j, r+s)\ar[rd] & (j+1, r+s)\ar[l]\ar[r] & (j+2, r+s)\ar[ld]& & (j+1, r+s)\ar[d] &  \\
(j, r+s+1) & (j+1, r+s+1)\ar[u]\ar[l]\ar[r] &(j+2, r+s+1)  & (j, r+s+1) &  (j+1, r+s+1) \ar[l]\ar[r] &(j+2, r+s+1) \\
& (j+1, r+s+2) \ar[u] & \ar[r]^{ \mu_{(j+1, r+s)}}&& (j+1, r+s+2)\ar[u]& 
 }}
\caption{Arrows incident to $(j+1, r+s)$ and $(j+1, r+s+1)$ in $Q_{\xi}^{s}[i_1, j+1]$ (left), and arrows incident to $(j+1, r+s+1)$ in $Q_{\xi}^{s+1}[i_1, j+1]$ (right).}\label{The proof of lemmas fig: 33}
\end{figure}

{\bf Case 3.} In the case where $\xi(j+1)=\xi(j+2)+1$, $d_{j-1}=1$, $d_j=0$, $t_{j+1}\geq 2$, we have $t_{j-1}=t_j$.

We need to consider $2$ cases: $t_{j-1}=t_j \leq 0$ or $t_{j-1}=t_j >0$. In the former case, we prove it by induction on $s-1$. It follows from Figure \ref{arrows incident with 8} that the arrows incident to $(j+1,r+1)$ and $(j+1, r+2)$ in $\Q^{1}[i_1, j+1]$ are shown in the left hand side of Figure \ref{The proof of lemmas fig: 41}.  After mutating $(j+1, r+1)$ in $\Q^{1}[i_1, j+1]$, the arrows incident to $(j+1, r+2)$ in $\Q ^{2}[i_1, j+1]$ are as required, see Figure \ref{The proof of lemmas fig: 41}.
\begin{figure}
\centering
\resizebox{0.8\textwidth}{!}{ 
\xymatrix{
 &(j, r-1)   & & & &(j, r-1)   &  \\
(j-1, r)& (j, r) \ar[rd] &  & &(j-1, r)& & \\
& &(j+1, r+1) \ar[r]\ar[llu] \ar[luu]_{p^{0}(t_j)} \ar[ld]&(j+2, r+1)  \ar[ld]& && (j+1, r+1)\ar[d]  &  \\
  & (j, r+2) \ar[r]& (j+1, r+2) \ar[ld]\ar[u]\ar[r]&(j+2, r+2)  \ar[u]  
 && & (j+1, r+2) \ar@/^1pc/[luuu]_{p^{0}(t_j)} \ar[lluu] \ar[r] \ar[ld]&(j+2, r+2) \\
 &(j, r+3) & (j+1, r+3) \ar[u] & \ar[r]^{ \mu_{(j+1, r+1)}}  & & (j, r+3) & (j+1, r+3) \ar[u]  & 
}}
\caption{Arrows incident to $(j+1, r+1)$ and $(j+1, r+2)$ in $Q_{\xi}^{1}[i_1, j+1]$ (left), and arrows incident to $ (j+1, r+2)$ in $Q_{\xi}^{2}[i_1, j+1]$ (right).}
\label{The proof of lemmas fig: 41}
\end{figure}
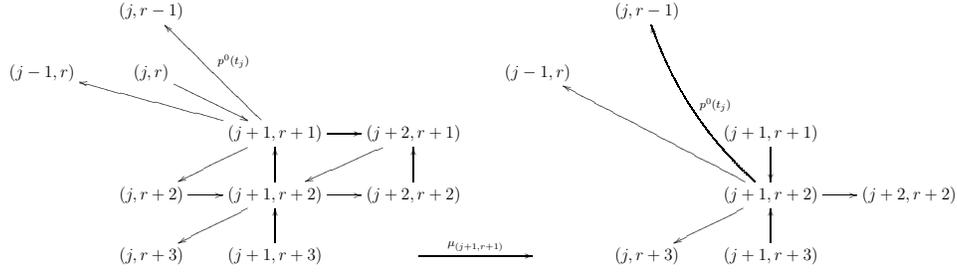

Suppose that $2\leq  s\leq t_{j+1}-1$, and the arrows incident to $(j+1, r+s)$ in $\Q^{s}[i_1, j+1]$ are the required arrows. By induction, we need to prove that the result holds for the arrows incident to $(j+1, r+s+1)$ in $\Q ^{s+1}[i_1, j+1]$. By induction the arrows incident to $(j+1, r+s)$ and $(j+1, r+s+1)$ in $\Q^{s}[i_1, j+1]$ are shown in the left hand side of Figure~\ref{The proof of lemmas fig: 42}. After mutating $(j+1, r+s)$ in $\Q^{s}[i_1, j+1]$, the arrows incident to $(j+1, r+s+1)$ in $\Q ^{s+1}[i_1, j+1]$ are as required, see Figure \ref{The proof of lemmas fig: 42}.
\begin{figure}
\centering
\resizebox{0.8\textwidth}{!}{ 
\xymatrix{
 &(j, r-1)   & & & &(j, r-1)   &  \\
(j-1, r)& &  (j+1, r+s-1)\ar[d]& &(j-1, r)& & \\
& &(j+1, r+s) \ar[r]\ar[llu] \ar@/^1.6pc/[luu]_{p^{0}(t_j)} \ar[ld]&(j+2, r+s)  \ar[ld]& && (j+1, r+s)\ar[d]  &  \\
  &(j, r+s+1)\ar[r] & (j+1, r+s+1) \ar[ld]\ar[u]\ar[r]&(j+2, r+s+1)   
 && & (j+1, r+s+1) \ar[luuu]_{p^{0}(t_j)} \ar[lluu] \ar[r] \ar[ld]&(j+2, r+s+1) \\
  &(j, r+s+2) & (j+1, r+s+2) \ar[u] &  \ar[r]^{  \mu_{(j+1, r+s)}}  & & (j, r+s+2) & (j+1, r+s+2) \ar[u]  & 
}}
\caption{Arrows incident to $(j+1, r+s)$ and $(j+1, r+s+1)$ in $Q_{\xi}^{s}[i_1, j+1]$ (left), and arrows incident to $(j+1, r+s+1)$ in $Q_{\xi}^{s+1}[i_1, j+1]$ (right).} \label{The proof of lemmas fig: 42}
\end{figure}
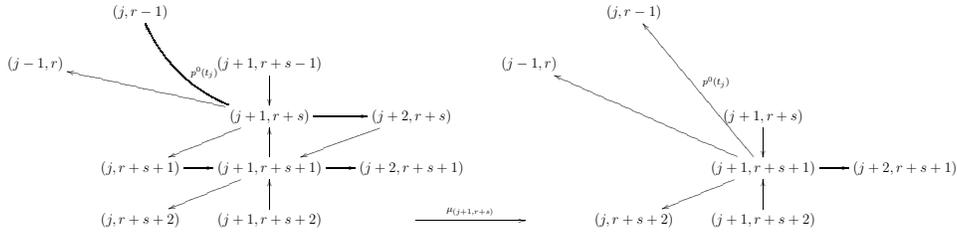

In the latter case, we first give the effects of the mutations of the vertices in the $j$th column on the arrows incident to the vertices in the $(j+1)$th column, as shown in Figure~\ref{The proof of lemmas fig: 31}.	

If $2\leq s\leq t_j-1$, we prove it by induction on $s-1$. It follows from Figures \ref{arrows incident with 8} and \ref{The proof of lemmas fig: 31} that the arrows incident to $(j+1,r+1)$ and $(j+1, r+2)$ in $\Q^{1}[i_1, j+1]$ are shown in the left hand side of Figure \ref{The proof of lemmas fig: 45} (the green arrows correspond to the case of $t_j=1$, the blue arrows correspond to the case of $t_j>1$).  After mutating $(j+1,r+1)$ in $\Q^{1}[i_1, j+1]$, the arrows incident to $(j+1, r+2)$ in $\Q ^{2}[i_1, j+1]$ are as required, see Figure \ref{The proof of lemmas fig: 45}.
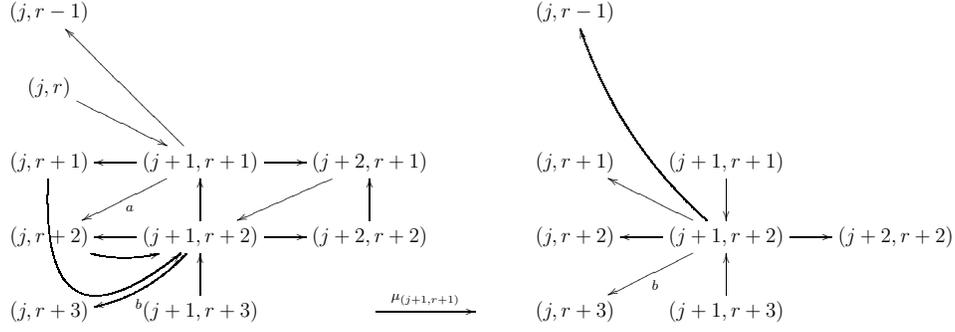
\begin{figure}
\centering
\resizebox{0.8\textwidth}{!}{ 
\xymatrix{
 (j, r-1)   & & & & (j, r -1)  &  \\
 (j, r) \ar[rd] &  & & & & \\
( j, r+1)\ar@[blue]@/^-4.95pc/[rd]&(j+1, r+1)\ar[ld]^{a} \ar[r]\ar[l] \ar[luu]  &(j+2, r+1)  \ar[ld]& &(j, r+1)& (j+1, r+1)\ar[d]  &  \\
   (j, r+2) \ar@[green]@/^-1pc/[r] & (j+1, r+2)  \ar@/_-1pc/[ld]^{b} \ar@[blue][l]\ar[u]\ar[r]&(j+2, r+2)  \ar[u]  
 && (j, r+2)& (j+1, r+2) \ar@[green][lu]\ar[ld]^{b} \ar@[blue][l]\ar@/^1pc/[luuu] \ar[r]  &(j+2, r+2) \\
  (j, r+3) & (j+1, r+3) \ar[u] &\ar[r]^{  \mu_{(j+1, r+1)}}   & & (j, r+3) & (j+1, r+3) \ar[u]  & 
}}
\caption{Arrows incident to $(j+1, r+1)$ and $(j+1, r+2)$ in $Q_{\xi}^{1}[i_1, j+1]$ (left), and arrows incident to $(j+1, r+2)$ in $Q_{\xi}^{2}[i_1, j+1]$ (right), where $a=1-p^{2}(t_{j-1})$, $b=1-p^{3}(t_{j-1})$.}
\label{The proof of lemmas fig: 45}
\end{figure}

Suppose that $2\leq s \leq t_{j}-2$, and the arrows incident to $(j+1, r+s)$ in $\Q^{s}[i_1, j+1]$ are the required arrows. By induction, we need to prove that the result holds for the arrows incident to $(j+1, r+s+1)$ in $\Q ^{s+1}[i_1, j+1]$. By induction and Figure \ref{The proof of lemmas fig: 31}, the arrows incident to $(j+1, r+s)$ and $(j+1, r+s+1)$ in $\Q^{s}[i_1, j+1]$ are shown in the left hand side of Figure \ref{The proof of lemmas fig: 46}. After mutating $(j+1, r+s)$ in $\Q^{s}[i_1, j+1]$, the arrows incident to $(j+1, r+s+1)$ in $\Q ^{s+1}[i_1, j+1]$ are as required, see Figure \ref{The proof of lemmas fig: 46}. Then our result holds for $2\leq s\leq t_{j}-1$.  
\begin{figure}
\centering
\resizebox{0.8\textwidth}{!}{ 
\xymatrix{
 (j, r-1)   & & & &  (j, r-1) &  \\
 &  (j+1, r+s-1)\ar[d]& & & & \\
(j, r+s)\ar[rd] &(j+1, r+s) \ar[r]\ar[l] \ar[luu] &(j+2, r+s)  \ar[ld]& && (j+1, r+s)\ar[d]  &  \\
 (j, r+s+1) & (j+1, r+s+1)\ar[l]  \ar[u]\ar[r]&(j+2, r+s+1)   
 && (j, r+s+1) & (j+1, r+s+1) \ar[luuu]  \ar[r] \ar[l]&(j+2, r+s+1) \\
    & (j+1, r+s+2) \ar[u] &  \ar[r]^{\mu_{(j+1, r+s)}}  && & (j+1, r+s+2) \ar[u]  & 
  }}
\caption{Arrows incident to $(j+1, r+s)$ and $(j+1, r+s+1)$ in $Q_{\xi}^{s}[i_1, j+1]$ (left), and arrows incident to $ (j+1, r+s+1)$ in $Q_{\xi}^{s+1}[i_1, j+1]$ (right).} \label{The proof of lemmas fig: 46}
\end{figure}
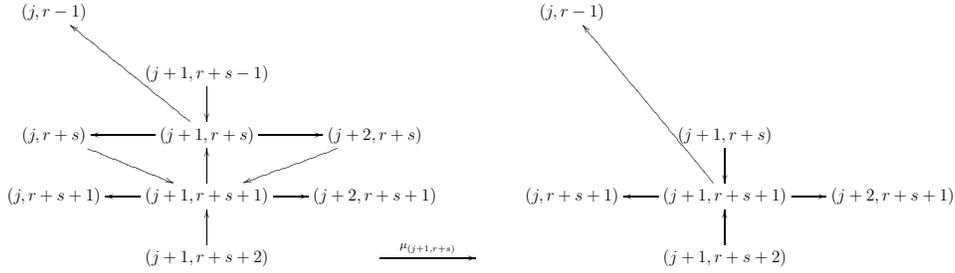
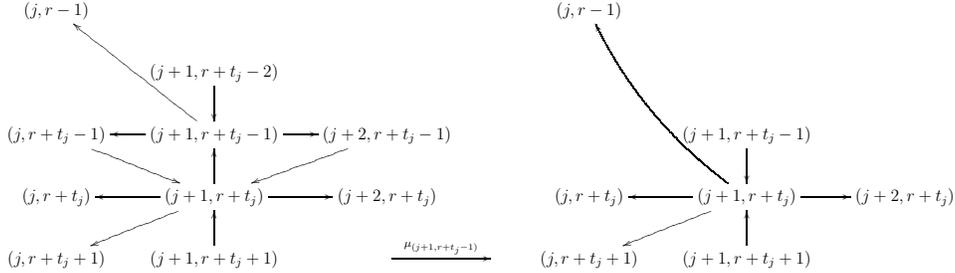
\begin{figure}
	\centering
	\resizebox{0.8\textwidth}{!}{ 
		\xymatrix{
			(j, r-1)   & & & &  (j, r-1) &  \\
			&  (j+1, r+t_j-2)\ar[d]& & & & \\
			(j, r+t_j-1)\ar[rd] &(j+1, r+t_j-1) \ar[r]\ar[l] \ar[luu] &(j+2, r+t_j-1)  \ar[ld]& && (j+1, r+t_j-1)\ar[d]  &  \\
			(j, r+t_{j}) & (j+1, r+t_{j})\ar[ld]\ar[l]  \ar[u]\ar[r]&(j+2, r+t_{j})  
			&& (j, r+t_{j}) & (j+1, r+t_{j}) \ar@/_-1pc/[luuu]   \ar[r] \ar[l]\ar[ld]&(j+2, r+t_{j}) \\
			(j, r+t_{j}+1)  & (j+1, r+t_{j}+1) \ar[u] & \ar[r]^{ \mu_{(j+1, r+t_j-1)}}   & & (j, r+t_{j}+1) & (j+1, r+t_{j}+1) \ar[u]  & 
	}}
	\caption{Arrows incident to $(j+1,r+ t_j-1)$ and $(j+1,r+ t_j)$ in $Q_{\xi}^{t_j-1}[i_1, j+1]$ (left), and arrows incident to $ (j+1, r+t_j)$ in $Q_{\xi}^{t_j}[i_1, j+1]$ (right).} \label{The proof of lemmas fig: 47}
\end{figure}

If $s=t_j$ (here $t_j\geq 3$, see Figure \ref{The proof of lemmas fig: 45} if $t_j=2$), it follows from Figures \ref{The proof of lemmas fig: 45} and \ref{The proof of lemmas fig: 46} that the arrows incident to $(j+1, r+t_j-1)$ and $(j+1, r+t_{j})$ in $\Q^{t_{j}-1}[i_1, j+1]$ are shown in the left hand side of Figure \ref{The proof of lemmas fig: 47}.  After mutating $(j+1, r+t_j-1)$ in $\Q ^{t_{j}-1}[i_1, j+1]$, the arrows incident to $(j+1, r+t_{j})$ in $\Q ^{t_{j}}[i_1, j+1]$ are as required, see Figure \ref{The proof of lemmas fig: 47}. Then our result holds for $s=t_{j}$. 
 
If $s\geq t_j+1$, we prove it by induction on $s-t_j$. It follows from Figures \ref{The proof of lemmas fig: 45} and \ref{The proof of lemmas fig: 47} that the arrows incident to $(j+1, r+t_{j})$ and $(j+1, r+t_{j}+1)$ in $\Q^{t_{j}}[i_1, j+1]$ are shown in the left hand side of Figure \ref{The proof of lemmas fig: 48}.  After mutating $(j+1, r+t_{j})$ in $\Q^{t_{j}}[i_1, j+1]$, the arrows incident to $(j+1, r+t_{j}+1)$ in $\Q^{t_{j}+1}[i_1, j+1]$ are as required, see Figure \ref{The proof of lemmas fig: 48}.
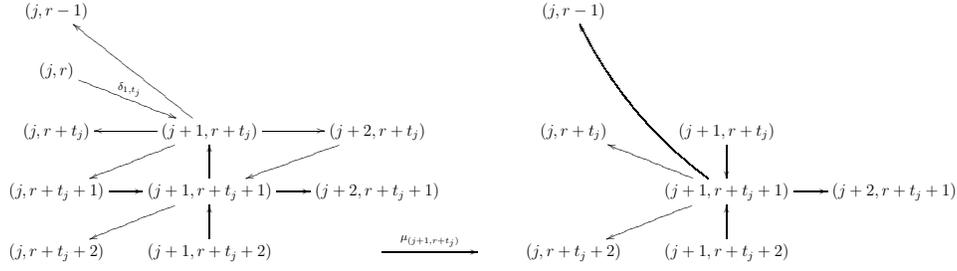
\begin{figure}
\centering
\resizebox{0.8\textwidth}{!}{ 
\xymatrix{
 (j, r-1)   & & & &  (j, r-1) &  \\
 (j, r) \ar[rd]^{\delta_{1,t_j}\qquad}&  & & & & \\
(j, r+t_{j}) &(j+1, r+t_{j})\ar[ld] \ar[r]\ar[l] \ar[luu] &(j+2, r+t_{j}) \ar[ld]& & (j, r+t_{j})& (j+1, r+t_{j})\ar[d]  &  \\
  (j, r+t_j+1)\ar[r] & (j+1, r+t_{j}+1) \ar[ld]  \ar[u]\ar[r]&(j+2, r+t_{j}+1)   
 && & (j+1, r+t_{j}+1) \ar@/_-1pc/[luuu]   \ar[r] \ar[lu]\ar[ld]&(j+2, r+t_{j}+1) \\
   (j, r+t_{j}+2)  & (j+1, r+t_{j}+2) \ar[u] &  \ar[r]^{ \mu_{(j+1, r+t_{j})}}  & & (j, r+t_{j}+2) & (j+1, r+t_{j}+2) \ar[u]  & 
 }}
\caption{Arrows incident to $(j+1, r+t_j)$ and $(j+1,r+ t_{j}+1)$ in $Q_{\xi}^{t_j}[i_1, j+1]$ (left), and arrows incident to $ (j+1, r+t_{j}+1)$ in $Q_{\xi}^{t_j+1}[i_1, j+1]$ (right).} \label{The proof of lemmas fig: 48}
\end{figure}
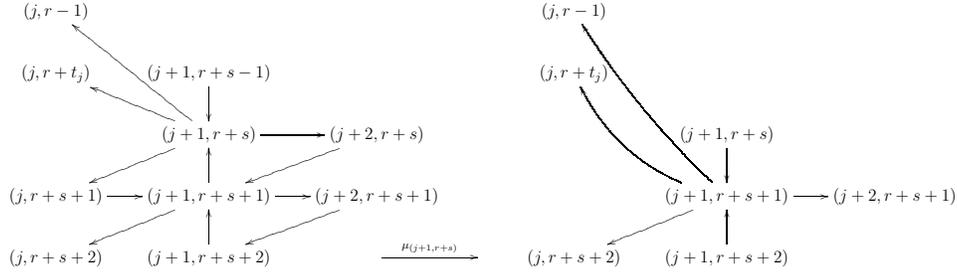
\begin{figure}
\centering
\resizebox{0.8\textwidth}{!}{ 
\xymatrix{
  (j, r-1)  & & & &  (j, r-1) &  \\
 (j, r+t_{j}) &(j+1, r+s-1) \ar[d]& &  & (j, r+t_{j}) & \\
&(j+1, r+s) \ar[ld] \ar[r]\ar[lu] \ar[luu] &(j+2, r+s)  \ar[ld]& & & (j+1, r+s) \ar[d]  &  \\
(j, r+s+1)\ar[r] & (j+1, r+s+1) \ar[ld]  \ar[u]\ar[r]&(j+2, r+s+1)\ar[ld]    && & (j+1, r+s+1) \ar@/_-0.5pc/[luuu]  \ar@/_-1.5pc/[luu]   \ar[r]  \ar[ld]&(j+2, r+s+1) \\
 (j, r+s+2)  & (j+1, r+s+2) \ar[u] &  \ar[r]^{ \mu_{(j+1, r+s)}}  & & (j, r+s+2)& (j+1, r+s+2) \ar[u]  & 
 }}
\caption{Arrows incident to $(j+1,r+s)$ and $(j+1, r+s+1)$ in $Q_{\xi}^{s}[i_1, j+1]$ (left), and arrows incident to $ (j+1, r+s+1)$ in $Q_{\xi}^{s+1}[i_1, j+1]$ (right).} \label{The proof of lemmas fig: 49}
 \end{figure}

Suppose that $t_{j}+1 \leq  s\leq t_{j+1}-1$, and the arrows incident to $(j+1, r+s)$ in $\Q^{s}[i_1, j+1]$ are the required arrows. By induction, we need to prove that the result holds for the arrows incident to $(j+1, r+s+1)$ in $\Q ^{s+1}[i_1, j+1]$. By induction the arrows incident to $(j+1, r+s)$ and $(j+1, r+s+1)$ in $\Q^{s}[i_1, j+1]$ are shown in the left hand side of Figure \ref{The proof of lemmas fig: 49}. After mutating $(j+1, r+s)$ in $\Q^{s}[i_1, j+1]$, the arrows incident to $(j+1, r+s+1)$ in $\Q ^{s+1}[i_1, j+1]$ are as required, see Figure \ref{The proof of lemmas fig: 49}.

\section*{Acknowledgements}
The work was partially supported by the National Natural Science Foundation of China (No. 12171213, 12001254) and by Gansu Province Science Foundation for Youths (No. 22JR5RA534).

\end{document}